\documentclass[11pt]{amsart}
\newtheorem{theorem}{Theorem}[section]
\newtheorem{proposition}[theorem]{Proposition}
\newtheorem{lemma}[theorem]{Lemma}
\newtheorem{corollary}[theorem]{Corollary}
\newtheorem{remark}[theorem]{Remark}
\usepackage{comment}
\newtheorem{conjecture}[theorem]{Conjecture}
\theoremstyle{definition}
\newtheorem{definition}[theorem]{Definition}
\newtheorem{claim}[theorem]{Claim}
\usepackage{hyperref}
\usepackage{color}
\newcommand{\eps}{\varepsilon}

\title[Ancient low entropy flows]{Ancient low entropy flows, mean convex neighborhoods, and uniqueness}

\author{Kyeongsu Choi, Robert Haslhofer, Or Hershkovits}

\begin{document}

\maketitle

\begin{abstract}
In this article, we prove the mean convex neighborhood conjecture for the mean curvature flow of surfaces in $\mathbb{R}^3$. Namely, if the flow has a spherical or cylindrical singularity at a space-time point $X=(x,t)$, then there exists a positive $\varepsilon=\varepsilon(X)>0$ such that the flow is mean convex in a space-time neighborhood of size $\varepsilon$ around $X$. The major difficulty is to promote the infinitesimal information about the singularity to a conclusion of macroscopic size. In fact, we prove a more general classification result for all ancient low entropy flows that arise as potential limit flows near $X$. Namely, we prove that any ancient, unit-regular, cyclic, integral Brakke flow in $\mathbb{R}^3$ with entropy at most $\sqrt{2\pi/e}+\delta$ is either a flat plane, a round shrinking sphere, a round shrinking cylinder, a translating bowl soliton, or an ancient oval. As an application, we prove the uniqueness conjecture for mean curvature flow through spherical or cylindrical singularities. In particular, assuming Ilmanen's multiplicity one conjecture, we conclude that for embedded two-spheres the mean curvature flow through singularities is well-posed.\end{abstract}

\tableofcontents

\section{Introduction}

A family of surfaces $M_t\subset \mathbb{R}^3$ moves by mean curvature flow if the normal velocity at each point is given by the mean curvature vector,\footnote{The equation of course also makes sense in higher dimension and co-dimension and in other ambient manifolds. In this paper, however, we focus on evolving surfaces in $\mathbb{R}^3$.}
\begin{equation}\label{eq_mcf}
(\partial_t x)^\perp = \mathbf{H}(x)\qquad \qquad(x\in M_t).
\end{equation}

Given any smooth initial surface $M\subset \mathbb{R}^3$, say closed and embedded, by classical theory \cite{Huisken_convex,HuiskenPolden} there exists a unique smooth solution $\mathcal{M}=\{M_t\}_{t\in [0,T)}$ with initial condition $M$ defined on a maximal time interval $[0,T)$. The first singular time $T<\infty$ is characterized by the fact that the curvature blows up, i.e.
\begin{equation}
\lim_{t\nearrow T}\max_{x\in M_t}|A|(x,t)=\infty,
\end{equation}
where $|A|$ denotes the norm of the second fundamental form. The main task in the study of mean curvature flow, both from the theoretical point of view and also in order to facilitate the most striking applications, is then to understand the structure of singularities, to find ways to continue the flow beyond the first singular time, and to analyze its properties.\\

In the mean convex case, i.e. when the mean curvature vector at every point on the surface points inwards, there is a highly developed theory. On the one hand, the flow can be continued smoothly as a surgical solution as constructed by Brendle-Huisken \cite{BH_surgery} and Haslhofer-Kleiner \cite{HaslhoferKleiner_surgery}. This in turn facilitates topological and geometric applications, see e.g. \cite{BuzanoHaslhoferHershkovits,HaslhoferKetover}. On the other hand, the flow can also be continued uniquely as a weak (generalized) solution. Weak solutions can be described either as level set solutions as in Evans-Spruck \cite{EvansSpruck} and Chen-Giga-Goto \cite{CGG}, or in the framework of geometric measure theory using Brakke solutions \cite{Brakke_book}.\footnote{In the mean convex setting these notions of weak solutions are essentially equivalent \cite{White_size,MetzgerSchulze}.} By the deep structure theory of White \cite{White_size,White_nature} (see also \cite{HaslhoferKleiner_meanconvex}) the space-time dimension of the singular set is at most one, and all blowup limits are smooth and convex. In fact, by a result of Colding-Minicozzi \cite{CM_singular_set} the space-time singular set is contained in finitely many compact embedded Lipschitz curves together with a countable set of point singularities. Moreover, the recent work of Brendle-Choi \cite{BC} and Angenent-Daskalopoulos-Sesum \cite{ADS2} provides a short list of all potential blowup limits (singularity models) in the flow of mean convex surfaces: the round shrinking sphere, the round shrinking cylinder, the translating bowl soliton \cite{AltschulerWu}, and the ancient ovals \cite{White_size,HaslhoferHershkovits_ancient}.\\

In stark contrast to the above, when the initial surface is not mean convex, the theory is much more rudimentary. This is, to some extent, an unavoidable feature of the equation. In particular, as already pointed out in the pioneering work of Brakke \cite{Brakke_book} and Evans-Spruck \cite{EvansSpruck} there is the phenomenon of non-uniqueness or fattening.  Angenent-Ilmanen-Chopp \cite{AIC} and Ilmanen-White \cite{White_ICM} gave examples of smooth embedded surfaces $M\subset\mathbb{R}^3$ whose level set flow $F_t(M)$ develops a non-empty interior at some positive time. In particular, $F_t(M)$ does not look at all like a two-dimensional evolving surface. These examples also illustrate, in a striking way, the non-uniqueness of (enhanced) Brakke flows.\\

In the present paper, we make some progress towards decreasing the gap between the theory in the mean convex case and the theory in the general case without curvature assumptions. Most importantly, in Theorem \ref{thm_mean_convex_nbd_intro} (see also Theorem \ref{cor_mean_convex_nbd_first}), we prove the mean convex neighborhood conjecture. As an application, we prove the nonfattening conjecture for mean curvature flow through cylindrical or spherical singularities (see Theorem \ref{thm_nonfattening_intro}). In particular, assuming Ilmanen's multiplicity one conjecture, we conclude that for embedded two-spheres the mean curvature flow through singularities is well-posed (see Theorem \ref{thm_uniqueness_two_spheres}). The mean convex neighborhood conjecture is in turn a consequence of a  general classification result (Theorem \ref{thm_ancient_low_entropy}) for ancient low entropy flows that arise as potential limit flows near spherical or cylindrical singularities.\\

The rest of this introduction is organized as follows. In Section \ref{sec_intro_classification}, we describe our general classification result for ancient low entropy flows. In Section \ref{sec_intro_mean_convex}, we state our results establishing the mean convex neighborhood conjecture and the uniqueness conjectures. In Section \ref{sec_outline}, we give an outline of the proofs.

\bigskip

\subsection{Classification of ancient low entropy flows}\label{sec_intro_classification}

Let $M\subset\mathbb{R}^3$ be a surface. The entropy, introduced by Colding-Minicozzi \cite{CM_generic}, is defined as the supremum of the Gaussian area over all centers and all scales, namely
\begin{equation}\label{eq_ent_smooth}
\textrm{Ent}[M]=\sup_{y\in\mathbb{R}^3,\lambda>0} \int_{M} \frac{1}{4\pi\lambda} e^{-\tfrac{|x-y|^2}{4\lambda}} dA(x).
\end{equation}
The entropy measures, in a certain sense, the complexity of the surface. For example, the values for a plane, sphere and cylinder are
\begin{equation}
\textrm{Ent}[\mathbb{R}^2]=1,\quad \textrm{Ent}[\mathbb{S}^2]=\frac{4}{e}\sim 1.47,\quad \textrm{Ent}[\mathbb{S}^1\times \mathbb{R}]=\sqrt{\frac{2\pi}{e}} \sim 1.52.
\end{equation}
If $\mathcal{M}=\{M_t\}_{t\in I}$ evolves by mean curvature flow, then $t\mapsto \textrm{Ent}[M_t]$ is nonincreasing by Husiken's monotonicity formula \cite{Huisken_monotonicity}, hence
\begin{equation}
\textrm{Ent}[\mathcal M] :=\sup_{t\in I} \textrm{Ent}[M_t] = \lim_{t\to \inf(I)}  \textrm{Ent}[M_t].
\end{equation}
For example, the entropy of a flat plane $\mathcal{P}$, a round shrinking sphere $\mathcal{S}$, a round shrinking cylinder $\mathcal{Z}$, a translating bowl soliton $\mathcal{B}$, and an ancient oval $\mathcal{O}$, are given by
\begin{align}
&\textrm{Ent}[\mathcal P]=1,\quad \textrm{Ent}[\mathcal{S}]=\frac{4}{e}\sim 1.47\nonumber\\
& \textrm{Ent}[\mathcal Z]=\textrm{Ent}[\mathcal B]=\textrm{Ent}[\mathcal O]=\sqrt{\frac{2\pi}{e}} \sim 1.52.
\end{align}

By a beautiful classification result of Bernstein-Wang \cite{BW} any self-similarly shrinking mean curvature flow in $\mathbb{R}^3$ with entropy at most $\sqrt{2\pi/e}$ is either a flat plane, round shrinking sphere, or round shrinking cylinder. By a recent uniqueness result of Hershkovits \cite{Hershkovits_translators} any self-similarly translating flow in $\mathbb{R}^3$ with entropy at most $\sqrt{2\pi/e}$ is a bowl soliton. These results in turn build on a pioneering paper by Colding-Ilmanen-Minicozzi-White \cite{CIMW} that was the genesis of the study of low entropy flows.\\

For the study of singularities, it is important that these concepts are also available in the non-smooth setting. Singular surfaces in Euclidean space are described most easily by two-rectifiable Radon measures, which generalize the area measure of smooth two-dimensional surfaces, see e.g. \cite{Simon_GMT,Ilmanen_book}. Recall that a two-rectifiable Radon measure $\mu$ in $\mathbb{R}^3$ is a Radon measure that has a two-dimensional tangent plane of positive multiplicity at almost every point. The entropy of $\mu$ is defined as
\begin{equation}
\textrm{Ent}[\mu]=\sup_{y\in\mathbb{R}^3,\lambda>0} \int \frac{1}{4\pi\lambda} e^{-\tfrac{|x-y|^2}{4\lambda}} d\mu(x).
\end{equation}
As in \cite{Brakke_book,Ilmanen_book} we consider Brakke flows $\mathcal M = \{\mu_t\}_{t\in I}$ that are given by a family of Radon measures in $\mathbb{R}^3$ that is two-rectifiable for almost all times and satisfies
\begin{equation}\label{eq_Brakke_inequality}
\frac{d}{dt} \int \varphi \, d\mu_t \leq \int \left( -\varphi {\bf H}^2 + \left(\nabla\varphi\right)^\perp \cdot {\bf H} \right)\, d\mu_t
\end{equation}
for all nonnegative test functions $\varphi$, see Section \ref{sec_Brakke_flows} for details. The class of all Brakke flows, as originally defined in \cite{Brakke_book}, is too large for most practical purposes. For example, by the very nature of the definition via the inequality \eqref{eq_Brakke_inequality}, Brakke flows can suddenly vanish without any cause. Also, given that the initial surface is embedded, it is useful to keep track of its inside and outside. For these and other reasons the definition of Brakke flows has been refined over the years by Ilmanen \cite{Ilmanen_book}, White \cite{White_regularity,White_Currents}, and their collaborators. For our purpose it is most appropriate to consider only Brakke flows that are integral (which prevents non-integer multiplicity), unit-regular (which to a certain extent prevents sudden vanishing), and cyclic (which partly keeps track of the inside and outside), see Section \ref{sec_Brakke_flows} for definitions and details. In particular, all Brakke flows starting at any closed embedded surface $M\subset\mathbb{R}^3$ that are constructed via Ilmanen's elliptic regularization \cite{Ilmanen_book} are integral, unit regular, and cyclic, and all these properties are preserved under passing to limits of sequences of Brakke flows.
The entropy of a Brakke flow is defined by
\begin{equation}
\textrm{Ent}[\mathcal M] =\sup_{t\in I} \textrm{Ent}[\mu_t].
\end{equation}
Finally, let us recall that an \emph{ancient} solution of a parabolic PDE, such as the mean curvature flow, is a solution that is defined for all $t\in (-\infty,T)$, where $T\leq \infty$. Ancient solutions are on the one hand some of the most interesting solutions by themselves, and on the other hand crucial for the analysis of singularities. In particular, every blowup limit is an ancient solution. A systematic study of ancient solutions of geometric flows has been pursued over the last decade by Daskalopoulos, Sesum, and their collaborators, see \cite{Daskalopoulos_ICM} for a nice overview of the main problems and key results.\\

We consider the following class of Brakke flows:

\begin{definition}[ancient low entropy flows]\label{def_low_entropy_flow}
The class of \emph{ancient low entropy flows} consists of all ancient, unit-regular, cyclic, integral Brakke flows $\mathcal{M}=\{\mu_t\}_{t\in (-\infty,T_E(\mathcal{M})]}$ in $\mathbb{R}^3$ with
\begin{equation}
\textrm{Ent}[\mathcal M]\leq \sqrt{\frac{2\pi}{e}},
\end{equation}
where $T_E(\mathcal{M})\leq \infty$ denotes the extinction time.
\end{definition}

Our main theorem provides a complete classification:

\begin{theorem}[classification of ancient low entropy flows]\label{thm_ancient_low_entropy}
Any ancient low entropy flow in $\mathbb{R}^3$ is either
\begin{itemize}
\item a flat plane, or
\item a round shrinking sphere, or
\item a round shrinking cylinder, or
\item a translating bowl soliton, or
\item an ancient oval.
\end{itemize}
\end{theorem}

To explain the scope of our main classification result (Theorem \ref{thm_ancient_low_entropy}) let us start by pointing out that the assumptions are essentially sharp.\footnote{The entropy assumption can be relaxed, see Corollary \ref{cor_ancient_low_entropy} below.} If one removes the assumption of being unit-regular the flow can suddenly disappear, e.g. one could have a solution that looks like a translating bowl soliton until some time, but is not eternal. If one removes the assumption of being cyclic, then one can have a static or quasistatic configuration of three half-planes meeting at $120$-degree angles, which have entropy $3/2$. If one increases the entropy too much one can get other ancient solutions, e.g. the selfsimilarly shrinking Angenent torus \cite{Angenent_torus} has entropy bigger than $\sqrt{2\pi/e}$ but less than $2$. Finally, removing the integer multipicity assumption would give rise to a zoo of solutions, in particular the infinite collection of shrinkers constructed by Kapouleas-Kleene-Moller and Nguyen \cite{KKM,Nguyen}.\\

Let us discuss some of the most important prior classification results for solutions of the mean curvature flow of surfaces. Wang proved uniqueness of the bowl soliton among entire convex translating graphs \cite{Wang_convex}. This was improved recently by Spruck-Xiao \cite{SpruckXiao}, who showed that every entire mean convex translating graph is in fact convex.\footnote{In the setting of noncollapsed solutions it was known before that every ancient noncollapsed mean convex solution of the mean curvature flow is in fact convex \cite{HaslhoferKleiner_meanconvex}.} Without assuming that the translator is graphical, Martin, Savas-Halilaj and Smoczyk proved a uniqueness result for the bowl soliton under rather strong asymptotic assumptions \cite{MSHS}. This was improved by Haslhofer \cite{Haslhofer_bowl}, who showed uniqueness of the bowl among translators that are convex and noncollapsed, and more recently by Hershkovits \cite{Hershkovits_translators}, who proved uniqueness among translators, not necessarily convex, that are asymptotic to a cylinder. Most closely related to the present article are the recent result by Brendle-Choi \cite{BC}, which proves uniqueness of the bowl among noncompact ancient solutions that are noncollapsed and convex, and the recent result by Angenent-Daskalopoulos-Sesum \cite{ADS2}, which proves uniqueness of the ancient ovals among compact ancient solutions that are noncollapsed and convex. Other highly important results that play a direct role in the present paper are the classification of low entropy shrinkers by Bernstein-Wang \cite{BW}, and the classification of genus zero shrinkers by Brendle \cite{Brendle_sphere}.\\

All prior classification results assume either convexity or self-similarity (or both), and the story is similar for other equations such as the Ricci flow, see e.g. \cite{Brendle_Bryant}. Our classification result seems to be the first one which assumes neither self-similarity nor convexity. To wit, convexity is not an assumption but a consequence:

\begin{corollary}[convexity]\label{cor_convex}
Any ancient low entropy flow in $\mathbb{R}^3$ is convex.
\end{corollary}

More generally, a key feature of our classification theorem (Theorem \ref{thm_ancient_low_entropy}) is that besides the entropy bound we assume almost nothing, essentially only that the solution is ancient. In particular, a priori, an ancient low entropy flow could be quite singular and could have pathological behaviour caused by spatial infinity, e.g. there could be topological changes caused by ``neckpinches at infinity", or there could be ``contracting cusps" (see e.g. \cite{Topping_cusps}), or the flow could ``escape to spatial infinity" (see e.g. \cite{IsenbergWu}).\\

The entropy assumption in our main classification result (Theorem \ref{thm_ancient_low_entropy}) can be relaxed. To state the sharp result, let
\begin{multline}
\Upsilon:=\inf\{ \textrm{Ent}[S] \, |\, \textrm{$S$ is a smooth embedded shrinker in $\mathbb{R}^3$,}\\
\textrm{which is not a plane, sphere of cylinder} \}.
\end{multline}
By Bernstein-Wang \cite{BW}, the difference $\delta:=\Upsilon - \sqrt{2\pi/e}>0$ is strictly positive. Theorem \ref{thm_ancient_low_entropy} immediately implies the following corollary.
\begin{corollary}[sharp ancient low entropy classification]\label{cor_ancient_low_entropy}
Any ancient, unit-regular, cyclic, integral Brakke flows $\mathcal{M}$ in $\mathbb{R}^3$ with
\begin{equation}
\textrm{Ent}[\mathcal M]< \Upsilon
\end{equation}
is either a flat plane, a round shrinking sphere, a round shrinking cylinder, a translating bowl soliton, or an ancient oval.
\end{corollary}

The proof of Theorem \ref{thm_ancient_low_entropy} involves, not surprisingly, many steps. For an overview of the main steps and ideas, please see Section \ref{sec_outline}.

\bigskip

\subsection{Mean convex neighborhoods and uniqueness}\label{sec_intro_mean_convex}

The main conjecture towards reducing the gap between the theory in the mean convex case and the general case, is the following:

\begin{conjecture}[mean convex neighborhood conjecture\footnote{See e.g. Problem 4 on Ilmanen's problem list \cite{Ilmanen_problems}, Conjecture 10.2 in the survey of Colding-Minicozzi-Pedersen \cite{CMP}, and the paragraph between Theorem 3.5 and Remark 3.6 in Hershkovits-White \cite{HershkovitsWhite}.}]\label{conj_mean_convex}
If the mean curvature flow of closed embedded surfaces has a spherical or cylindrical singularity at $(x, t)$, then there is a space-time neighborhood of $(x, t)$ in which the flow is mean convex.
\end{conjecture}

Let us explain some background. Let $\mathcal{M}=\{M_t\}$ be a Brakke flow starting at a closed embedded surface $M\subset\mathbb{R}^3$, say constructed via Ilmanen's elliptic regularization \cite{Ilmanen_book}. Given any space-time point $X=(x,t)\in \mathcal{M}$ and any factor $\lambda>0$, we denote by $\mathcal{M}_{X,\lambda}$ the Brakke flow which is obtained from $\mathcal M$ by translating $X$ to the space-time origin and parabolically rescaling by $\lambda^{-1}$. By Huisken's monotonicity formula \cite{Huisken_monotonicity} and Ilmanen's compactness theorem \cite{Ilmanen_book} for any sequence $\lambda_i\to 0$ one can find a subsequence $\lambda_{i'}$ such that $\mathcal{M}_{X,\lambda_{i'}}$ converges in the sense of Brakke flows to a limit $\hat{\mathcal{M}}_X$, called a tangent flow at $X$, and any tangent flow is backwardly selfsimilar. The only backwardly selfsimilar solutions of positive mean curvature are the round shrinking sphere and the round shrinking cylinder \cite{Huisken_shrinker,White_nature,CM_generic}. One precise formulation of the conjecture says that if some (and thus every \cite{CIM}) tangent flow $\hat{\mathcal{M}}_X$ is either a round shrinking sphere or a round shrinking cylinder of multiplicity one, then there exists some $\varepsilon=\varepsilon(X)>0$ such that the flow $\mathcal M$ is, possibly after flipping the orientation, mean convex in the parabolic ball
\begin{equation}
P(X,\varepsilon)= \{ (x',t')\in \mathbb{R}^3\times \mathbb{R} \, : \, |x'-x|<\varepsilon, \, t-\varepsilon^2<t'\leq t \}.
\end{equation}

\bigskip

The major challenge in proving the mean convex neighborhood conjecture is to upgrade the infinitesimal information about the mean convex tangent flow to a conclusion at some macroscopic scale $\varepsilon>0$. Specifically, there are certain singularities that are not fully captured by just looking at the tangent flows. Indeed, it has been clear since the work of White \cite{White_size,White_nature}, and reinforced in the work of Perelman \cite{Perelman1,Perelman2}, Bamler-Kleiner \cite{BamlerKleiner} and Hershkovits-White \cite{HershkovitsWhite}, that in order to fully capture all singularities and regions of high curvature one needs to understand all limit flows and not just all tangent flows. Given $X\in\mathcal{M}$ a limit flow is any subsequential limit of  $\mathcal{M}_{X_i,\lambda_i}$ where $X_i\to X$ and $\lambda_i\to 0$. For example, in the case of the degenerate neckpinch \cite{AV_degenerate_neckpinch} one needs to carefully choose a varying sequence of space-time points $X_i$ that ``follow the tip" in order to see the translating bowl soliton in the limit, see \cite{Hamilton_Harnack}. Potentially there could be very complex limit flows that are not self-similar, see e.g. \cite[Problem 6]{Ilmanen_problems} or \cite[Conjecture 3]{White_nature}.\\

The mean-convex neighborhood conjecture (Conjecture \ref{conj_mean_convex}) has been known to be true in some important special cases.
In the rotationally symmetric setting this is by the attracting axis theorem of Altschuler-Angenent-Giga \cite{AAG}. More recently, Colding-Ilmanen-Minicozzi \cite{CIM} proved the conjecture in the case of selfsimilarly shrinking limit flows, and Hershkovits \cite{Hershkovits_translators} proved the conjecture in the case of selfsimilarly translating limit flows. Without assuming selfsimilarity, in a very interesting recent paper based on hard PDE techniques, Gang Zhou \cite{Gang} verified the mean convex neighborhood conjecture in certain regimes. However, these regimes exclude by assumption some of the main challenges of the conjecture, e.g. the potential scenario of a degenerate neckpinch with a non-convex cap.\\

In Theorem \ref{thm_mean_convex_nbd_intro}, we establish the mean convex neighborhood conjecture for the mean curvature flow of surfaces in $\mathbb{R}^3$ at all times.\footnote{Of course our results carry over immediately to other ambient manifolds.} Since it is somewhat easier to state, let us first give the result at the first singular time.

\begin{theorem}[mean convex neighborhoods at the first singular time]\label{cor_mean_convex_nbd_first}
Let $\mathcal M=\{M_t\}_{t\in [0,T)}$ be a mean curvature flow of closed embedded surfaces in $\mathbb{R}^3$, where $T$ is the first singular time. If some tangent flow at $X=(x,T)$ is a sphere or cylinder with multilicity one, then there exists an $\varepsilon=\varepsilon(X)>0$ such that, possibly after flipping the orientation, the flow $\mathcal M$ is mean convex in the parabolic ball $P(X,\varepsilon)$. Moreover, any limit flow at $X$ is either a round shrinking sphere, a round shrinking cylinder, or a translating bowl soliton.\footnote{It is easy to see that ancient ovals cannot arise as limit flows at the first singular time. It is unknown, whether or not ancient ovals can arise as limit flows at subsequent singularities. This is related to potential accumulations of neckpinches, see e.g. \cite{CM_ArnoldThom}.}
\end{theorem}

Theorem \ref{cor_mean_convex_nbd_first} follows, via a short argument by contradiction, from our classification of ancient low entropy flows (Theorem \ref{thm_ancient_low_entropy}).\\

Before stating the solution of the mean convex neighborhoods conjecture at subsequent singularities, let us explain better what it actually means to continue the mean curvature flow through singularities. We start by recalling some facts that are explained in more detail in \cite{Ilmanen_lectures,HershkovitsWhite}. For any closed set $K\subset\mathbb{R}^3$, its level set flow $F_t(K)$ is the maximal family of closed sets starting at $K$ that satisfies the avoidance principle. Now, given any closed embedded surface $M\subset \mathbb{R}^3$ there are at least three quite reasonable ways to evolve it through singularities, namely the \emph{level set flow} $F_t(M)$, the \emph{outer flow} $M_t$ and the \emph{inner flow} $M_t'$. The latter two are defined as follows. Let $K$ be the compact domain bounded by $M$, and let $K':=\overline{\mathbb{R}^3\setminus K}$. Denote the corresponding level set flows by
\begin{equation}
K_t=F_t(K),\quad \textrm{ and } \quad K_t' = F_t(K').
\end{equation}
Let $\mathcal{K}$ and $\mathcal{K}'$ be their space-time tracks, namely
\begin{align}
\mathcal{K}&= \{(x,t)\in \mathbb{R}^3\times \mathbb{R}_+\, | \, x\in K_t \}\\
\mathcal{K}'&= \{(x,t)\in \mathbb{R}^3\times \mathbb{R}_+\, | \, x\in K_t' \}.\nonumber
\end{align}
The outer flow and inner flow are then defined by
\begin{align}
M_t&=\{x\in \mathbb{R}^3\, | \, (x,t)\in \partial \mathcal{K} \}\\
M_t'&=\{x\in \mathbb{R}^3\, | \, (x,t)\in \partial \mathcal{K}' \}.\nonumber
\end{align}
As long as the evolution is smooth, all three flows are of course the same.\\

Let $T$ be the first singular time, and for $x\in M_T$ and $\lambda>0$, denote by $\mathcal{K}_{X,\lambda}$ and $\mathcal{K}_{X,\lambda}'$ the flows that are obained from $\mathcal{K}$ and $\mathcal{K'}$, respectively, by shifting $X=(x,T)$ to the origin, and parabolically rescaling by $\lambda^{-1}$.
Observe that the assumption of Theorem \ref{cor_mean_convex_nbd_first} is equivalent to the assumption that for $\lambda\to 0$ either $\mathcal{K}_{X,\lambda}$ or $\mathcal{K}_{X,\lambda}'$ converges smoothly with multiplicity one to a round shrinking solid cylinder or round shrinking solid ball. Also observe that the conclusion of Theorem \ref{cor_mean_convex_nbd_first} is equivalent to the assertion that either
\begin{equation}
\quad K_{t_2}\cap B(x,\varepsilon)\subseteq K_{t_1}\setminus M_{t_1}\quad \textrm{ or } \quad K_{t_2}'\cap B(x,\varepsilon)\subseteq K_{t_1}'\setminus M_{t_1}'
\end{equation} 
for all $T-\varepsilon^2< t_1< t_2 \leq T$. This reformulation has the advantage that it generalizes not just beyond the first singular time, but even beyond any potential fattening-time (see equation \eqref{eq_fattening_time} below). We prove:

\begin{theorem}[mean convex neighborhoods at all times]\label{thm_mean_convex_nbd_intro}
Assume $X=(x,t)$ is a space-time point such that $\mathcal{K}_{X,\lambda}$ converges for $\lambda\to 0$ smoothly with multiplicity one to a round shrinking solid cylinder or a round shrinking solid ball. Then there exists an $\varepsilon=\varepsilon(X)>0$ such that
\begin{equation}
\quad K_{t_2}\cap B(x,\varepsilon)\subseteq K_{t_1}\setminus M_{t_1} 
\end{equation} 
for all $T-\varepsilon^2< t_1< t_2 < T+\varepsilon^2$. Similarly, if $\mathcal{K}_{X,\lambda}'$ converges for $\lambda\to 0$ smoothly with multiplicity one to a round shrinking solid cylinder or a round shrinking solid ball, then for some  $\varepsilon=\varepsilon(X)>0$ we have
\begin{equation}
\quad K_{t_2}'\cap B(x,\varepsilon)\subseteq K_{t_1}'\setminus M_{t_1}'
\end{equation}
for all $T-\varepsilon^2< t_1< t_2 < T+\varepsilon^2$. 
Furthermore, in both cases, any limit flow at $X$ is either a round shrinking sphere, a round shrinking cylinder, a translating bowl soliton, or an ancient oval.
\end{theorem}

\bigskip

Note that Theorem \ref{thm_mean_convex_nbd_intro} establishes the strongest version of the mean convex neighborhood conjecture, where the parabolic neighborhood is backward and forward in time. The proof of Theorem \ref{thm_mean_convex_nbd_intro} is again based on our classification of ancient low entropy flows (Theorem \ref{thm_ancient_low_entropy}). During the course of the proof,  we also construct a unit-regular, cyclic, integral Brakke flow whose support is $M_t$ or $M_t'$, respectively, and which has density less than $2$ in $P(X,\delta)$, for some $\delta>0$. This highlights another advantage of our formulation of Theorem \ref{thm_mean_convex_nbd_intro}: its applicability even without assuming a priori that a Brakke flow starting from $M$ has multiplicity one almost everywhere.\\

Let us now discuss the applications concerning uniqueness of mean curvature flow through singularities. We recall from Ilmanen-White \cite{White_ICM} that there are examples of singularities with curvature of mixed signs, which cause non-uniqueness, i.e. fattening. The main conjecture regarding uniqueness at cylindrical or spherical singularities can be phrased as follows:

\begin{conjecture}[{nonfattening conjecture\footnote{The conjecture has been indicated first in White's ICM lecture \cite{White_ICM}, see also the discussion in the recent work of Hershkovits-White \cite{HershkovitsWhite} (indeed, the nonfattening conjecture is a combination of \cite[Conj. 1.2]{HershkovitsWhite} and the ``very interesting open problem'' described in the paragraph between Theorem 3.5 and Remark 3.6 in \cite{HershkovitsWhite}).}}]\label{conj_nonfattening}
If all singularities of the level set flow are either cylindrical or spherical then the level set flow does not fatten.
\end{conjecture} 

We recall that the \emph{fattening time} is defined by
\begin{equation}\label{eq_fattening_time}
T_{\textrm{fat}}= \inf \{ t > 0 \, | \, \textrm{$F_t(M)$ has non-empty interior} \},
\end{equation}
and the \emph{discrepancy time} is defined by
\begin{equation}
T_{\textrm{disc}}= \inf \{ t > 0 \, | \, \textrm{$M_t$, $M_t'$, and $F_t(M)$ are not all equal} \},
\end{equation}
see \cite{HershkovitsWhite}. It follows directly from the definitions that
\begin{equation}
T_{\textrm{fat}}\geq T_{\textrm{disc}}.
\end{equation}
It is unknown whether or not $T_{\textrm{fat}}= T_{\textrm{disc}}$ holds, see \cite[Ex. 2.5]{HershkovitsWhite}. In any case, if one proves the (potentially) stronger result that there is no discrepancy, then one obtains nonfattening as a consequence. We prove:

\begin{theorem}[nonfattening]\label{thm_nonfattening_intro}
Suppose that $0<T\leq T_{\textrm{disc}}$, and suppose that all the backward singularities of the outer flow $\{M_t\}$ at time $T$ are cylindrical or spherical. Then $T< T_{\textrm{disc}}$. In particular, the level set flow does not fatten as long as all singularities are cylindrical or spherical.
\end{theorem}

Theorem \ref{thm_nonfattening_intro} follows immediately by combining the main theorem of Hershkovits-White \cite{HershkovitsWhite}, which establishes nondiscrepancy assuming the existence of mean convex neighborhoods around all singularities a priori, and Theorem \ref{thm_mean_convex_nbd_intro}, which proves the existence of such mean convex neighborhoods.\\

Our final application in this paper concerns the well-posedness problem for the mean curvature flow of embedded two-spheres through singularities. 
The main conjecture in this regard is the following.

\begin{conjecture}[evolution of embedded two-spheres\footnote{See e.g. White's ICM lecture \cite{White_ICM}, and the introduction of Bamler-Kleiner \cite{BamlerKleiner}.}]\label{conj_uniqueness_two_spheres}
The mean curvature flow of embedded two-spheres through singularities is unique.
\end{conjecture}

One reason why Conjecture \ref{conj_uniqueness_two_spheres} is highly intriguing, is that the mean curvature flow of embedded two-spheres can be viewed as (somewhat more approachable) cousin of the Ricci flow of four-spheres, whose analysis seems to be out of reach with current technology.\footnote{When comparing mean curvature flow and Ricci flow one has to multiply the dimension by two. For example, convergence to a round limit holds for curve-shortening flow and two-dimensional Ricci flow, respectively. Also, Huisken's monotonicity formula combined with Gauss-Bonnet gives local bounds for the $L^2$-norm of the second fundamental form, see \cite{Ilmanen_monotonicity}, which is scale invariant in dimension two, while on the other hand Perelman's monotonicity formula combined with Gauss-Bonnet gives local bounds for the $L^2$-norm of the Riemann tensor, see \cite{HaslhoferMuller}, which is scale invariant in dimension four.} A solution of Conjecture \ref{conj_uniqueness_two_spheres} would also facilitate several topological and geometric applications. E.g. there is a proposal of Yau for a flow proof of the Smale conjecture \cite{Rubinstein} (see also the work of Bamler-Kleiner \cite{BamlerKleiner} and Buzano-Haslhofer-Hershkovits \cite{BuzanoHaslhoferHershkovits}). Moreover, a well-posed flow of embedded-two spheres would also be useful for the Lusternik-Schnirelman problem of finding four embedded minimal two-spheres in three-spheres equipped with an arbitrary Riemannian metric (the first such minimal two-spheres has been found by Simon-Smith \cite{SimonSmith} in 1983, and a second one has been found recently by Haslhofer-Ketover \cite{HaslhoferKetover} using combined efforts of mean curvature flow and min-max theory).\\

Some very important recent progress towards Conjecture \ref{conj_uniqueness_two_spheres} has been made by Brendle \cite{Brendle_sphere}, who proved that the only nontrivial shrinkers of genus zero are the round shrinking sphere and the round shrinking cylinder. Combining Brendle's uniqueness result with our solution of the nonfattening conjecture (Theorem \ref{thm_nonfattening_intro}), we can prove Conjecture \ref{conj_uniqueness_two_spheres} modulo (a special case of) Ilmanen's multiplicity one conjecture:

\begin{conjecture}[multiplicity one conjecture for two-spheres\footnote{Multiplicity one was called an hypothesis in Brakke's pioneering monograph \cite{Brakke_book}, and promoted to a conjecture in Ilmanen's paper \cite{Ilmanen_monotonicity}, where he proved that for surfaces one can always find smooth tangent flows; see also the approach by Ecker \cite{Ecker}.}]\label{conj_mult_one} For the mean curvature flow starting at any closed embedded two-sphere $M$,\footnote{To be concrete, we mean by that a unit-regular, cyclic, integral Brakke flow, whose support is the outer flow, and whose initial data is $\mathcal{H}^{n}\lfloor M$. A posteriori, this will turn out to imply that the (nonfattening) matching motion has only multiplicity one tangents.} all tangent flows have multiplicity one.
\end{conjecture}

\begin{theorem}[evolution of embedded two-spheres]\label{thm_uniqueness_two_spheres}
Assuming the multiplicity one conjecture (Conjecture \ref{conj_mult_one}), mean curvature flow of embedded two-spheres through singularities is unique.
\end{theorem}

\bigskip

Let us conclude this section by sketching a heuristic picture for the mean curvature flow of embedded surfaces of higher genus: By Ilmanen's strict genus reduction conjecture \cite[Conjecture 13]{Ilmanen_problems} one expects only a controlled number (bounded by the genus of the initial surface) of singularities modelled on asymptotically conical shrinkers of higher genus, see in particular the important work of Brendle \cite{Brendle_sphere} and Wang \cite{Wang_shrinker}. At these singularities the evolution can be non-unique, and one has to decide how to flow out of the conical singularities, either by hand or via stability or via a stochastic selection principle. Our solution of the nonfattening conjecture (Theorem \ref{thm_nonfattening_intro}) suggests that the evolution is determined completely by these finitely many choices.

\bigskip

\subsection{Outline of the proofs}\label{sec_outline}

For the outline of the proof of the classification theorem for ancient low entropy flows (Theorem \ref{thm_ancient_low_entropy}), let $\mathcal{M}$ be an ancient low entropy flow (see Definition \ref{def_low_entropy_flow}) which is not a flat plane, round shrinking sphere, or round shrinking cylinder. The task is to prove that $\mathcal M$ is either a translating bowl soliton or an ancient oval.\\

\noindent In Section \ref{sec_coarse_properties}, we start by establishing several coarse properties:

First, we prove a partial regularity result (Theorem \ref{partial_regularity}), which shows that $\mathcal{M}$ is smooth except for at most countably many spherical singularities. This is mostly based on ideas from Bernstein-Wang \cite{BW}, and uses the assumptions that our Brakke flow is integral, unit-regular and cyclic.

Next, given any space-time point $X=(x,t)\in \mathcal M$ we consider the flow on backwards parabolic annuli around $X$ with radii $r_j=2^j$. We prove a rough neck theorem (Theorem \ref{thm_finding_necks}), which shows that there exists a scale $Z(X)=2^{J(X)}$, the \emph{cylindrical scale}, at which the flow starts to look $\varepsilon$-cylindrical, and a controlled number $N=N(\varepsilon)<\infty$, such that the flow looks $\varepsilon$-cylindrical at all scales $j\geq J(X)+N$. The proof uses the Bernstein-Wang classification of low entropy shrinkers \cite{BW}, Huisken's monotonicity formula \cite{Huisken_monotonicity}, and ideas from quantitative differentiation, see e.g. \cite{CHN_stratification}.

Third, we prove the vanishing asymptotic slope theorem (Theorem \ref{thm_asymptotic_slope}), which shows that at spatial infinity, the surfaces $M_t$ open up slower than any cone of positive angle. The proof is based on a blowdown argument, which shows that the cylindrical scale grows sublinearly. The result facilitates barrier arguments in later sections.\\

\noindent In Section \ref{sec_fine_neck_analysis}, which is the longest section of the proof, we carry out a fine neck analysis.  Given any point $X=(x_0,t_0)\in \mathcal M$, we consider the rescaled flow
\begin{equation}
\bar M^{X}_\tau = e^{\frac{\tau}{2}} \, \left( M_{-e^{-\tau}} - x_0\right),
\end{equation}
where $\tau=-\log (t_0-t)$. The surfaces $\bar M^{X}_\tau$ can be written as the graph of a function $u_{X}(\cdot,\tau)$ with small norm over a cylindrical domain of length $\rho(\tau)$, where $\rho(\tau)\to \infty$ as $\tau\to -\infty$. By Colding-Minicozzi \cite{CM_uniqueness}, we can assume that the axis of each cylinder is in $z$-direction. The goal is to derive very sharp asymptotic estimates for the function $u_{X}(\cdot,\tau)$.

To get started, in Section \ref{sec_fine_neck_setup} we set up the fine neck analysis similarly as in Angenent-Daskalopoulos-Sesum \cite{ADS} and Brendle-Choi \cite{BC}. The analysis is governed by the linear operator
\begin{equation}
\mathcal{L} = \Delta  - \tfrac{1}{2} x^{\textrm{tan}} \cdot \nabla +1
\end{equation}
on the round cylinder. The $\mathcal L$-operator already played a fundamental role in the work on generic mean curvature flow by Colding-Minicozzi \cite{CM_generic}. It has four positive eigenfunctions ($1$, $z$, $\sin \theta$, $\cos \theta$), three zero-eigenfunctions ($z^2-2$, $z\cos\theta$, $z\sin\theta$), and countably many negative eigenfunctions. Using the ODE-lemma from Merle-Zaag \cite{MZ}, we see that for $\tau\to -\infty$ either the plus mode is dominant or the neutral mode is dominant.

In Section \ref{sec_plus_mode}, we analyze the case where the plus mode is dominant. Our key result in that section is the fine neck theorem (Theorem \ref{thm Neck asymptotic}). It says that there exists a constant $a=a(\mathcal M)\neq 0$, independent of the base point $X$, such that, after re-centering, we have the asymptotic expansion
\begin{equation}\label{fine_neck_est}
u_{X}(x,\tau)= a z e^{\tau/2} + o(e^{\tau/2})
\end{equation}
whenever $|x|\leq 100$ and $\tau \ll \log Z(X)$. Assuming $a=1/\sqrt{2}$ without loss of generality, the estimate \eqref{fine_neck_est} says that the area of the cross sections of the fine necks increases at unit rate if one moves in positive $z$-direction.\footnote{It is an instructive exercise to check that the translating bowl soliton indeed satisfies the fine neck estimate, where $a$ is proportional to the reciprocal of the speed of the translating soliton.} To prove the fine neck theorem, we first use barrier arguments and the results from Section \ref{sec_coarse_properties} to show that the cylindrical radius $\rho(\tau)$ grows exponentially as $\tau\to -\infty$. We then project onto the plus mode (after multiplying by a cutoff function) and consider the evolution equations for the coefficients of
\begin{equation}
P_+ \hat{u}_X = a_X z + b_X \cos\theta + c_X \sin \theta + d_X.
\end{equation}
Carefully analyzing these evolution equations and estimating all the error terms, after re-centering, we obtain the estimate \eqref{fine_neck_est}.

In Section \ref{sec_fine_neutral}, we analyze the case where the neutral mode is dominant. Our main result in that section is the inwards quadratic neck theorem (Theorem \ref{thm_rotation}), which gives an inwards quadratic bending of the central neck. The result is related to the main result of Angenent-Daskalopoulos-Sesum \cite{ADS}, but we assume neither convexity nor rotational symmetry. Roughly speaking, the method of our proof of the inwards quadratic neck theorem is as follows. We consider the expansion
\begin{equation}
\hat{u}_X = \alpha_1 (z^2 - 2) + \alpha_2 z \cos \theta + \alpha_3 z \sin \theta + o(|\vec \alpha|).
\end{equation}
We show that $\alpha_2,\alpha_3$ are rapidly decaying, by using the Lojasiewicz inequality from Colding-Minicozzi and the Brendle-Choi neck improvement theorem. We then derive a differential inequality for $\alpha_1$. Analyzing this differential inequality we show that there is some positive constant $A>0$ such that for $\tau\to -\infty$ we have
\begin{equation}
\alpha_1(\tau) = \frac{-A+o(1)}{|\tau|}. 
\end{equation}
The proof is quite technical, since we have to estimate the optimal graphical radius, estimate the error terms, and analyze the differential inequality, and all those steps are interrelated. Finally, combining the inwards quadratic neck theorem with a barrier argument, we see that in the neutral mode case the solution must be compact (Corollary \ref{thm_compact}).\\

In Section \ref{sec_cap_size}, we assume that the plus mode is dominant.

In Section \ref{sec_curv_bound}, we prove a global curvature estimate (Theorem \ref{curv_bound}), which says that the flow is eternal with uniformly bounded curvature. In particular, this rules out potential pathologies caused by spatial infinity. Roughly speaking, the idea is that by the fine neck estimate \eqref{fine_neck_est}, the necks open up at spatial infinity, and thus the curvature decays to zero at spatial infinity. The actual proof is somewhat more involved, since we have to relate the cylindrical scale and the regularity scale. This gives a curvature bound on compact time intervals. Together with a local type I estimate, which follows from a simple blowdown argument, this can be upgraded to a global curvature bound. In particular, the solution is eternal.

In Section \ref{sec_cap_size_asymptotics}, we estimate the cap size and analyze the asymptotics. To begin with, we consider the ``height of the tip" function
\begin{equation}
\psi(t):= \inf_{x\in M_t} x_3.
\end{equation}
This is a strictly increasing function and the infimum is attained at some point $p_t\in M_t$. In the cap size control theorem (Theorem \ref{thm_asympt_par}), we prove that there is a uniform constant $C=C(\mathcal M)<\infty$ such that every point in $M_t\setminus B_C(p_t)$ lies on a fine neck. We also show that $M_t\setminus B_C(p_t)$ is the graph of a function $r$ in cylindrical coordinates around the $z$-axis satisfying  
\begin{equation}\label{expansion_cylindrical_intro}
r(t,z,\theta)=\sqrt{2(z-\psi(t))}+o(\sqrt{z-\psi(t)}),
\end{equation}
and that the height of the tip function $\psi$ satisfies
\begin{equation}
\psi(t)= t + o(|t|).
\end{equation}
Roughly speaking, the idea is that the fine neck estimate \eqref{fine_neck_est} can be integrated to get the shape of a parabola. In order to do this, we first prove, via a blowdown argument, that the tip does not slow down too much, which gives us enough ``fast points" in space-time around which the fine neck estimate can be applied.\\

In Section \ref{sec_rot_symm}, we still assume that the plus mode is dominant.

First, in the fine asymptotics theorem (Theorem \ref{thm_neck_asympt}) we prove that, after shifting $\mathcal M$ in space-time, the function $r$ from \eqref{expansion_cylindrical_intro} becomes rotationally symmetric at a very fast rate, namely
\begin{equation}
|\partial_\theta r| = O(1/r^{100})
\end{equation}
uniformly in time. The proof is based on the Brendle-Choi neck improvement theorem \cite{BC} and some careful estimates controlling how the fine necks align with each other. In particular, the estimate uniformly bounds the motion of the tip in the $xy$-plane. Combining the fine asymptotics theorem with a parabolic version of the moving plane method, we show that the solution is rotationally symmetric (Theorem \ref{thm_rot_symm}).\\

In Section \ref{sec_classification_completion}, we complete the classification of ancient low entropy flows.

In Theorem \ref{thm_class_noncompact}, we show that if the plus mode is dominant, then the solution must be a bowl soliton. To this end, we analyze the rotationally symmetric solution from Section \ref{sec_rot_symm}. Using elementary geometric arguments we show that the height function $z:M_t\to \mathbb{R}$ does not have local maxima, and similarly that the radius function $r=r(z,t)$ does not have local maxima. Thus, the function $f=\langle \nu, e_3\rangle$ is positive. Together with the fine asymptotic theorem and the maximum principle, we conclude that $H=f$. Hence, the solution is a mean convex noncollapsed translating soliton, which by the uniqueness result for translators from Haslhofer \cite{Haslhofer_bowl} yields that the solution is a bowl soliton.

Finally, in Theorem \ref{thm_class_neutral} we show that if the neutral mode is dominant, then the solution must be an ancient oval. The idea is that via blowup around the tips we can find eternal low entropy flows, which by the above must be bowl solitons. We also have a central cylinder. Using the maximum principle, we show that the solution is mean convex and noncollapsed also in the region in between. Using the result of Angenent-Daskalopoulos-Sesum \cite{ADS2}, we can then conclude that the solution is an ancient oval.\\

\bigskip

Finally, let us outline the proof of the mean convex neighborhood conjecture (Theorem \ref{thm_mean_convex_nbd_intro}). Under the assumptions of Theorem \ref{thm_mean_convex_nbd_intro}, we first construct a unit-regular, cyclic, integral Brakke flow $\mathcal{M}=\{\mu_t\}_{t\geq t_0-\delta}$ whose support is equal to the outer flow $\{M_t\}_{t\geq t_0-\delta}$. Moreover, we can arrange that $\mathcal{M}$ has a multiplicity one cylindrical tangent flow at $X_0$, and that $\mathcal{M}$ is smooth in $P(X_0,\delta)$ at almost every time, with only spherical and cylindrical singularities. The latter properties follow from upper semicontinuity of Huisken's density and the basic regularity theory from Section \ref{sec_coarse_properties}.

We then show that, possibly after decreasing $\delta$, the mean curvature does not vanish at any regular point in $P(X_0,\delta)$. If this did not hold, then we could pass to a limit flow which on the one hand had a point with vanishing mean curvature, but on the other hand had strictly positive mean curvature by our classification result of ancient low entropy flows (Theorem \ref{thm_ancient_low_entropy}).

Next, we show that, after decreasing $\delta$ again, the space-time connected component that reaches $X_0$ has positive mean curvature at all regular points. The proof is based on a continuity argument in time and uses that the singular set is small.

Finally, arguing as in \cite{HershkovitsWhite} we show that the mean curvature is positive also on the nonsmooth points in the sense that the surface moves everywhere strictly inwards. This concludes the outline of the proof.

\bigskip

\subsection*{Acknowledgments}
KC has been supported by NSF Grant DMS-1811267 and KIAS Individual Grant MG078901. RH has been supported by an NSERC Discovery Grant (RGPIN-2016-04331) and a Sloan Research Fellowship. OH has been supported by an AMS-Simons Travel Grant.
We thank Brian White for sharing his knowledge and insights about the mean curvature flow with us, as well as for many useful discussions regarding the problem.

\section{Preliminaries}\label{sec_prelim}

\subsection{Brakke flows}\label{sec_Brakke_flows}
As in Ilmanen \cite[Def. 6.2, 6.3]{Ilmanen_book} a two-dimensional \emph{Brakke flow} in $\mathbb{R}^{3}$ is a family of Radon measures $\mathcal M = \{\mu_t\}_{t\in I}$ in $\mathbb{R}^{3}$ that is two-rectifiable for almost all times and satisfies
\begin{equation}
\frac{d}{dt} \int \varphi \, d\mu_t \leq \int \left( -\varphi {\bf H}^2 + (\nabla\varphi)^{\perp} \cdot {\bf H} \right)\, d\mu_t
\end{equation}
for all test functions $\varphi\in C^1_c(\mathbb{R}^{3},\mathbb{R}_+)$. Here, $\tfrac{d}{dt}$ denotes the limsup of difference quotients, $\perp$ denotes the normal projection, and $\bf{H}$ denotes the mean curvature vector of the associated varifold $V_{\mu_t}$, which is defined via the first variation formula and exists almost everywhere at almost all times. The integral on the right hand side is interpreted as $-\infty$ whenever it does not make sense literally.\\

All Brakke flows $\mathcal M = \{\mu_t\}_{t\in I}$ that we encounter in the present paper have the following three additional good properties of being:
\begin{itemize}
\item \emph{integral} (c.f. \cite{Brakke_book,Ilmanen_book}): $\mu_t$ is integer two-rectifiable for almost all $t$.
\item \emph{unit-regular} (c.f. \cite{White_regularity,SchulzeWhite}): Every spacetime point of Gaussian density one is a regular point, i.e. for all $X=(x,t)\in \mathcal{M}$ with $\Theta_{X}=1$, there exists an $\varepsilon=\varepsilon(X)>0$ such that $t'\mapsto\mathrm{spt}(\mu_{t'})\cap B(x,\varepsilon)$ is a smooth mean curvature flow for $t'\in [t-\varepsilon^2,t+\varepsilon^2]$.
\item \emph{cyclic} (c.f. \cite{White_Currents}): For almost all $t$ the associated $\mathbb{Z}_2$ flat chain $[V_{\mu_t}]$ is cyclic. Here, the $\mathbb{Z}_2$ flat chain $[V]$ associated to an integral varifold $V$ is called cyclic if $\partial [V]=0$.
\end{itemize}

By Allard's closure theorem \cite[Thm. 6.4]{Allard} and a result of White \cite[Thm. 3.3]{White_Currents}, respectively, being integral and cyclic is preserved under varifold convergence with locally bounded first variation.  Limits of sequences of Brakke flows can be taken via Ilmanen's compactness theorem \cite[Sec. 7]{Ilmanen_book}. 
As a consequence of the above quoted results (see also White \cite[Thm. 4.2, Rmk. 4.4]{White_Currents}), being integral and cyclic is preserved under limits of Brakke flows. By the local regularity theorem \cite{White_regularity} (see also Schulze-White \cite[Thm. 4.2]{SchulzeWhite}), being unit-regular is also preserved under limits of Brakke flows. In particular, Brakke flows starting at any closed embedded surface $M\subset\mathbb{R}^3$ that are constructed via Ilmanen's elliptic regularization \cite{Ilmanen_book} are integral, unit-regular, and cyclic. We also recall from Ilmanen \cite[Sec. 10]{Ilmanen_book} that $M_t:=\mathrm{spt}(\mu_t)$ satisfies the avoidance principle.

\bigskip

\subsection{Monotonicity formula and tangent flows}\label{sec_prel_monotonicity}

Let $\mathcal{M}=\{\mu_t \}_{t\in I}$ be a two-dimensional unit-regular, integral Brakke flow in $\mathbb{R}^3$, say with bounded area ratios.\footnote{In particular, any flow with finite entropy has bounded area ratios.} Given a space-time point $X_0=(x_0,t_0)\in \mathcal{M}$, let
\begin{equation}
\rho_{X_0}(x,t)=\frac{1}{4\pi(t_0-t)} e^{-\frac{|x-x_0|^2}{4(t_0-t)}} \qquad (t<t_0).
\end{equation}
By Huisken's monotonicity formula \cite{Huisken_monotonicity} (see also \cite{Ilmanen_monotonicity}) we have
\begin{equation}\label{eq_huisken_mon}
\frac{d}{dt} \int \rho_{X_0} \, d\mu_t \leq -\int \left|{\bf H}-\frac{(x-x_0)^\perp}{2(t-t_0)}\right|^2 \rho_{X_0}\, d\mu_t.
\end{equation}
The Gaussian density of $\mathcal M$ at $X_0$ is defined by
\begin{equation}
\Theta_{X_0}(\mathcal M)=\lim_{t\nearrow t_0} \int \rho_{X_0}(x,t) \, d\mu_t(x).
\end{equation}
It follows from the monotonicity formula that the Gaussian density is a well defined real number $\geq 1$ and that the map
\begin{equation}
(x,t)\to \Theta_{(x,t)}(\mathcal M)
\end{equation}
is upper-semicontinous. Moreover, by the local regularity theorem for the mean curvature flow \cite{Brakke_book,White_regularity} there exists a universal constant $\varepsilon_0>0$ such that any $X\in\mathcal{M}$ with $\Theta_X \leq 1+\varepsilon_0$ is a regular point.\\

Given $X\in\mathcal{M}$ and $\lambda_i\to 0$, let $\mathcal{M}^i$ be the Brakke flow which is obtained from $\mathcal{M}$ by translating $X$ to the space-time origin and parabolically rescaling by $\lambda_i^{-1}$. By the compactness theorem for Brakke flows \cite{Ilmanen_book} one can pass to a subsequential limit $\hat{\mathcal{M}}_X$, which is called a tangent flow at $X$. By the monotonicity formula every tangent flow is backwardly selfsimilar, i.e. $\hat{\mathcal{M}}_X\cap \{t\leq 0\}$ is invariant under parabolic dilation $\mathcal{D}_\lambda(x,t)=(\lambda x,\lambda^2 t)$. If $\mathcal{M}$ is ancient, then for any $\lambda_i\to \infty$ one can also pass along a subsequence to a backwardly selfsimilar limit $\check{\mathcal{M}}$, which is called a tangent flow at infinity.

\bigskip

\section{Coarse properties of ancient low entropy flows}\label{sec_coarse_properties}

\subsection{Partial regularity}\label{sec_partial_reg}

The goal of this section is to prove a partial regularity result for ancient low entropy flows. To this end, we start with the following lemma.

\begin{lemma}[low entropy cyclic minimal cones]\label{lemma_cones}
Let $\mu$ be an integer two-rectifiable Radon measure in $\mathbb{R}^3$ with $\textrm{Ent}[\mu]<2$. If the associated varifold $V_\mu$ is a cyclic minimal cone, then $\mu=\mathcal{H}^2\lfloor P$ for some flat plane $P$.
\end{lemma}

\begin{proof}
The proof is along the lines of \cite[Lem. 4.1]{BW}. If $x\in \mathrm{spt}(\mu)-\{ o\}$, since being cyclic is preserved under weak limits, any tangent cone $C$ at $x$ is a cyclic minimal cone which splits off a line. Such a minimal cone consists of a static configuration of even number of half-planes meeting along a common line. The entropy assumption implies that there are only two of those half-planes, and since the configuration is static, it follows that $C$ is a multiplicity one plane. Hence, by Allard's regularity theorem \cite{Allard}, $x$ is a regular point, and so the only potential singularity is at the vertex $o$. As the link of $\textrm{spt}(\mu)$ is a smooth multiplicity one geodesic in $S^2$, it is a great circle, and so $\mu$ is indeed the Hausdorff measure of a flat multiplicity one plane.      
\end{proof}

\begin{theorem}[partial regularity]\label{partial_regularity}
Let $\mathcal M$ be an ancient low entropy flow. Then there is either (1) a cylindrical singularity in which case $\mathcal M$ must be a round shrinking cylinder, or (2) no singularity, or (3) an at most countable number of spherical singularities.\footnote{At this point of the paper we cannot exclude the possibility of more than one spherical singularity, since we cannot exclude yet the potential scenario that several compact connected components form via ``neckpinches at infinity" and ``contracting cusps at infinity".}
\end{theorem}

\begin{proof}
Let $\check{\mathcal{M}}$ be a tangent flow at infinity (see Section \ref{sec_prel_monotonicity}). By preservation under weak limits, $\check{\mathcal{M}}$ is integral, unit regular, and cyclic. Using Lemma \ref{lemma_cones} it follows that any tangent cone to $\check{M}_{-1}$ is a flat multiplicity one plane.
Thus, by Allard's regularity theorem \cite{Allard}, $\check{M}_{-1}$ is smooth.
Hence, by the classification of low entropy shrinkers by Bernstein-Wang \cite{BW}, $\check{\mathcal{M}}$ is either a flat plane, a round shrinking sphere, or a round shrinking cylinder (all of multiplicity one). The same reasoning implies that any tangent flow $\hat{\mathcal{M}}_X$ at any space time point $X\in\mathcal{M}$ is either a flat plane, a round shrinking sphere, or a round shrinking cylinder (all of multiplicity one).\\
If a tangent flow at $X$ is a flat multiplicity one plane, then $X$ is a smooth point by the local regularity theorem \cite{Brakke_book,White_regularity}.
If there is some $X\in\mathcal{M}$ with a cylindrical tangent flow, then $\mathcal M$ is a round shrinking cylinder by the equality case of Huisken's monotonicity formula (see Section \ref{sec_prel_monotonicity}). Finally, spherical tangent flows can happen at most countably many times since the set of strict local
maxima of the map $X\to \Theta_X(\mathcal M)$ is countable, see e.g. \cite{White_stratification}. This proves the theorem.
\end{proof}

By Theorem \ref{partial_regularity} any ancient low entropy flow $\mathcal M=\{\mu_t \}_{t\in (-\infty, T_E(\mathcal M)]}$ is smooth except for at most countably many spherical singularities until it becomes extinct. Hence, recalling also that higher multiplicities are ruled out by the low entropy assumption, it is safe from now on to conflate the compact sets $M_t=\textrm{spt}(\mu_t)$ and the Radon-measures $\mu_t$ in the notation.

\begin{corollary}[extinction time]
If $\mathcal{M}$ is an ancient low entropy flow then exactly one of the following happens:\footnote{At this stage of the paper we cannot yet exclude the scenarios that the solution ``becomes extinct in more than one point" or that the solution ``escapes to infinity".}
\begin{enumerate}
\item[(i)] $T_E(\mathcal{M})=\infty$, i.e. the flow $\mathcal M$ is eternal, or
\item[(ii)] $T_E(\mathcal{M})<\infty$ and $\mathcal{M}$ is a round shrinking cylinder, or
\item[(iii)] $T_E(\mathcal{M})<\infty$ and $\mathcal{M}$ becomes extinct in at most countably many round points, or
\item[(iv)] $T_E(\mathcal{M})<\infty$ and for every $R<\infty$ there exist $T(R)<T_E(\mathcal{M})$ such that $B(0,R)\cap M_t=\emptyset$ for every $t\in(T(R),T_E(\mathcal{M})]$.
\end{enumerate}
\end{corollary}
\begin{proof}
Assume $T_E(\mathcal{M})<\infty$, and suppose (iv) does not hold. Then by upper semi-continuity of Huisken's density (see Section \ref{sec_prel_monotonicity}), there exist a point $X=(x_0,T_E(\mathcal{M)})$ which $\Theta_X\geq 1$. If the tangent flow at $X$ is a plane, then this contradicts the definition of $T_E(\mathcal{M})$ by unit regularity. If the tangent flow at $X$ is a cylinder, then we are in case (ii). If the tangent flow at $X$ is a round shrinking sphere, then we are in case (iii). This proves the corollary.
\end{proof}

\bigskip

\subsection{Finding necks back in time}\label{sec_necks_back_in_time}

Let $\mathcal M$ be an ancient low entropy flow. Given a point $X=(x,t)\in \mathcal{M}$ and a scale $r>0$, we consider the flow
\begin{equation}
\mathcal{M}_{X,r}=\mathcal{D}_{1/r}(\mathcal M -X),
\end{equation}
which is obtained from $\mathcal M$ by translating $X$ to the space-time origin and parabolically rescaling by $1/r$. Here, $\mathcal{D}_\lambda (x,t) = (\lambda x,\lambda^2 t)$.

\begin{definition}
Fix $\varepsilon>0$. We say that $\mathcal M$ is \emph{$\varepsilon$-cylindrical around $X$ at scale $r$}, if $\mathcal{M}_{X,r}$ is $\varepsilon$-close in $C^{\lfloor1/\varepsilon \rfloor}$ in $B(0,1/\varepsilon)\times [-1,-2]$ to the evolution of a round shrinking cylinder with radius $r(t)=\sqrt{-2t}$ and axis through the origin. The notions \emph{$\varepsilon$-planar} and \emph{$\varepsilon$-spherical} are defined similarly.
\end{definition}

Given any point $X=(x,t)\in\mathcal{M}$, we analyze the solution around $X$ at the diadic scales $r_j=2^j$, where $j\in \mathbb{Z}$. 

\begin{theorem}\label{thm_finding_necks}
For any small enough $\varepsilon>0$, there is a positive integer $N=N(\varepsilon)<\infty$ with the following significance. If $\mathcal M$ is an ancient low entropy flow, which is not a round shrinking sphere, round shrinking cylinder or a flat plane, then for every $X\in \mathcal{M}$ there exists an integer $J(X)\in\mathbb{Z}$ such that
\begin{equation}\label{eq_thm_quant1}
\textrm{$\mathcal M$ is not $\varepsilon$-cylindrical around $X$ at scale $r_j$ for all $j<J(X)$},
\end{equation}
and
\begin{equation}\label{eq_thm_quant2}
\textrm{$\mathcal M$ is $\tfrac{\varepsilon}{2}$-cylindrical around $X$ at scale $r_j$ for all $j\geq J(X)+N$}.
\end{equation}
\end{theorem}

\begin{proof}
In the following argument we will frequently use the local regularity theorem for the mean curvature flow (see \cite{Brakke_book,White_regularity}) without explicitly mentioning it.
Given $\varepsilon>0$, for any $X=(x,t)\in \mathcal{M}$ we define
\begin{equation}\label{def_of_z}
J(X):=\inf \{ j\in \mathbb{Z}\, | \textrm{ $\mathcal M$ is $\varepsilon$-cylindrical around $X$ at scale $r_j$} \}.
\end{equation}
Recall that by Huisken's monotonicity formula \cite{Huisken_monotonicity,Ilmanen_monotonicity} the quantity
\begin{equation}
\Theta_{(x,t)} (r)=\int_{M_{t-r^2}} \frac{1}{4\pi r^2}  e^ {-\frac{|x-y|^2}{4r^2}}\,   dA(y)
\end{equation}
is monotone, with equality only in the self-similarly shrinking case. Recall also that by Bernstein-Wang \cite{BW}, the only self-similarly shrinking solutions with low entropy are the flat plane, the round shrinking sphere and the round shrinking cylinder. 

Since $\mathcal{M}$ is non-flat and not a round shrinking sphere, its tangent flow at infinity $\check{\mathcal{M}}$ (see Section \ref{sec_prel_monotonicity}) must be a round shrinking cylinder. Hence, $J(X)<\infty$. Similarly, if $\mathcal{M}$ had a cylindrical tangent flow $\hat{\mathcal{M}}_X$, then by the equality case of Huisken's monotonicity formula $\mathcal M$ would be a round shrinking cylinder. Hence, $J(X)> -\infty$. Therefore,  $J(X)\in\mathbb{Z}$ and the statement \eqref{eq_thm_quant1} holds true by the definition from equation \eqref{def_of_z}.

To prove \eqref{eq_thm_quant2}, note first that (recalling that our fixed $\varepsilon$ is small) we have
\begin{equation}\label{density_zero}
\Theta_{X} \left(r_{J(X)}\right) \geq \mathrm{Ent}[S^1]-\tfrac{1}{100},
\end{equation}
and
\begin{equation}\label{density_infinity}
\lim_{j\to \infty} \Theta_{X} (r_j) =\mathrm{Ent}[S^1].
\end{equation}
Next, observe that by the equality case of Huisken's monotonicity formula, if $\Theta_{X}(r_{j+1})-\Theta_{X}(r_{j-1})=0$ then $\mathcal M$ is $0$-selfsimilar around $X$ at scale $r_j$. This can be made quantitative (c.f. Cheeger-Haslhofer-Naber \cite{CHN_stratification}). Namely, there exists a $\delta=\delta(\varepsilon)>0$ such that if
\begin{equation}\label{density_drop}
\Theta_{X}(r_{j+1})-\Theta_{X}(r_{j-1})\leq\delta,
\end{equation}
then
\begin{equation}
\textrm{$\mathcal M$ is $\tfrac{\varepsilon}{2}$-selfsimilar around $X$ at scale $r_j$.}
 \end{equation}
 For $j\geq J(X)$ in our context $\tfrac{\varepsilon}{2}$-selfsimilar simply means $\tfrac{\varepsilon}{2}$-cylindrical.

Finally, using again monotonicity and quantitative rigidity we see that after going from scale $J(X)$ to scale $J(X)+N$, where $N=N(\varepsilon)<\infty$, we have
\begin{equation}\label{density_zero}
\Theta_{X} \left(r_{J(X)+N(\varepsilon)}\right) \geq \mathrm{Ent}[S^1]-\delta.
\end{equation}
Combining the above facts, we conclude that \eqref{eq_thm_quant2} holds. This finishes the proof of the theorem. 
\end{proof}

\bigskip

We fix a small enough parameter $\varepsilon>0$ quantifying the quality of the necks for the rest of the paper.

\begin{definition}
The \emph{cylindrical scale} of $X\in\mathcal{M}$ is defined by
\begin{equation}
Z(X)=2^{J(X)}.
\end{equation}
\end{definition}

\begin{remark}
Recall that the regularity scale $R(X)$ is defined as the maximal radius $r$ such that $|A|\leq 1/r$ in the parabolic ball $P(X,r)$. It follows from the definition of regularity scale and the low entropy assumption that there exists a positive integer $N=N(\varepsilon)<\infty$ such that
\begin{equation}\label{eq_thm_quant4}
\textrm{$\mathcal M$ is ${\varepsilon}$-planar around $X$ at scale $r_j$ for all $j\leq \log_2 R(X)-N$}.
\end{equation}
\end{remark}

\bigskip

\subsection{Asymptotic slope}\label{sec_coarse}

Let $\mathcal M$ be an ancient low entropy flow, which is not a round shrinking sphere, round shrinking cylinder or a flat plane. Recall from above that $\mathcal M$ has an asymptotic cylinder $\check{\mathcal{M}}$ for $t\to -\infty$. By Colding-Minicozzi \cite{CM_uniqueness} the axis is unique. We can assume without loss of generality that the axis is in $x_3$-direction.

\begin{proposition}[asymptotic cylindrical scale]\label{prop_reg_growth}
For every $\delta>0$ there exists $\Lambda=\Lambda(\delta)<\infty$ such that if $(p_0,t_0)\in \mathcal M$ then
\begin{equation}
Z(p,t_0)\leq \delta |p-p_0|
\end{equation}
for all $p\in M_{t_0}$ with $|p-p_0|\geq \Lambda Z(p_0,t_0)$.
\end{proposition}

\begin{proof}
Assume without loss of generality that $(p_0,t_0)=(0,0)$ and $Z(0,0)=1$. If the assertion fails, then there is a sequence $p_i\in M_{0}$ with $|p_i|\geq i$, but 
\begin{equation}\label{reg_point}
Z(p_i,0)\geq \delta |p_i|
\end{equation}
for some $\delta>0$.
Let $\mathcal M^i$ be the flow which is obtained by parabolically rescaling by $1/|p_i|$ around $(0,0)$ and pass to a limit $\mathcal M^\infty$.
The limit $\mathcal M^\infty$ is an ancient low entropy flow, which has a cylindrical singularity at $(0,0)$. Hence, by the equality case of Huisken's monotonicity formula $\mathcal M^\infty$ is a round shrinking cylinder that becomes extinct at time $T=0$, in contradiction with \eqref{reg_point}. This proves the proposition.
\end{proof}

We normalize $\mathcal M$ such that $X_0=(0,0)\in\mathcal{M}$ and $Z(X_0)\leq 1$.

\begin{theorem}[asymptotic slope]\label{thm_asymptotic_slope}
For every $\delta>0$ there exists $\Lambda=\Lambda(\delta)<\infty$ such that
\begin{equation}
\frac{\sup\{ \sqrt{x_1^2 + x_2^2} | (x_1,x_2,x_3)\in M_t\}}{|x_3|} < \delta
\end{equation}
whenever $t\leq -10$ and $|(x_1,x_2,x_3)|\geq \Lambda \sqrt{-t}$.
\end{theorem}

\begin{proof}
We first claim that for every $\delta>0$ there exist $\Lambda_1=\Lambda_1(\delta)<\infty$ and $T_1=T_1(\delta)>-\infty$ such that every $(p,t)\in \mathcal M$ with $t\leq T_1$ and $\frac{|p|}{\sqrt{-t}}\geq \Lambda_1$ satisfies
\begin{equation}\label{eq_show_first}
x_1^2+x_2^2\leq \frac{\delta^2}{2} x_3^2.
\end{equation}
Indeed, if $(p,t)\in \mathcal M$ is any point with $t\leq -10$ and $\frac{|p|}{\sqrt{-t}}\gg 1$, then by Proposition \ref{prop_reg_growth} it has cylindrical scale $Z(p,t)\ll |p|$. If \eqref{eq_show_first} was violated, then going back further in time this neck would intersect the central neck of our asymptotic cylinder; a contradiction.

Consequently, for every $\delta>0$ there exist $T_2=T_2(\delta)> -\infty$ and $L(\delta)<\infty$ such that
\begin{equation}
\frac{1}{\sqrt{-t}}M_t\subset \left\{ |z|\leq L, x_1^2+x_2^2\leq 3 \right\} \cup \left\{ |x_3|\geq L, x_1^2+x_2^2\leq \frac{\delta^2}{2} x_3^2 \right\}
\end{equation}
for all $t\leq T_2$. Using suitable large spheres as barriers, this implies the assertion.
\end{proof}

\begin{corollary}[barrier for the rescaled flow]\label{cor_barrier}
There exists an even smooth function $\varphi :\mathbb{R}\to \mathbb{R}_+$ with $\lim_{z\to \pm \infty}\varphi'(z)=0$ such that the rescaled mean curvature flow
$\bar M^{X}_\tau = e^{\frac{\tau}{2}}( M_{-e^{-\tau}} - x_0)$,
where $\tau=-\log (t_0-t)$, satisfies
\begin{equation}
\bar{M}^{X}_\tau \subset \left\{ x_1^2+x_2^2 \leq \varphi(x_3)^2 \right\}
\end{equation}
for $\tau\leq \mathcal{T}(Z(X))$.
\end{corollary}

By the corollary, any potential ends must be in direction $x_3\to \pm \infty$. 

\bigskip

\section{Fine neck analysis}\label{sec_fine_neck_analysis}

\subsection{Setting up the fine neck analysis}\label{sec_fine_neck_setup}

Let $\mathcal M$ be a noncompact ancient low entropy flow in $\mathbb{R}^3$, which is not a round shrinking sphere, round shrinking cylinder or a flat plane. Given any point $X_0=(x_0,t_0)\in \mathcal M$, we consider the rescaled flow
\begin{equation}
\bar M^{X_0}_\tau = e^{\frac{\tau}{2}} \, \left( M_{-e^{-\tau}} - x_0\right),
\end{equation}
where $\tau=-\log (t_0-t)$. By Theorem \ref{thm_finding_necks} and Colding-Minicozzi \cite{CM_uniqueness}, the rescaled flow converges for $\tau\to -\infty$ to the cylinder $\Sigma = \{x_1^2+x_2^2=2\}$. Moreover, the convergence is uniform in $X_0$ once we normalize such that $Z(X_0)\leq 1$. Hence, we can find universal functions $\sigma(\tau)>0$ and $\rho(\tau)>0$ with
\begin{equation}\label{univ_fns}
\lim_{\tau \to -\infty} \sigma(\tau)=0,\quad\lim_{\tau \to -\infty} \rho(\tau)=\infty, \quad \textrm{and} -\rho(\tau) \leq \rho'(\tau) \leq 0,
\end{equation}
such that $\bar M^{X_0}_\tau$ is the graph of a function $u(\cdot,\tau)$ over $\Sigma \cap B_{2\rho(\tau)}(0)$, namely 
\begin{equation}
\{x + u(x,\tau) \nu_\Sigma(x): x \in \Sigma \cap B_{2\rho(\tau)}(0)\} \subset \bar{M}^{{X_0}}_\tau,
\end{equation} 
where $\nu_\Sigma$ denotes the outward pointing unit normal to $\Sigma$, and
\begin{equation}\label{small_graph}
\|u(\cdot,\tau)\|_{C^4(\Sigma \cap B_{2\rho(\tau)}(0))} \leq \sigma(\tau) \, \rho(\tau)^{-1}.
\end{equation}
We will now set up a fine neck analysis following \cite{ADS} and \cite{BC}.

\bigskip

We denote by $C<\infty$ and $\mathcal{T}>-\infty$ constants that can change from line to line and can depend on various other quantities (such as the function $\rho$ from above, the neck parameter $\varepsilon$ from Section \ref{sec_necks_back_in_time}, etc), but are independent of the point $X_0$ with $Z(X_0)\leq 1$. We also fix a nonnegative smooth cutoff function $\varphi$ satisfying $\varphi(z)=1$ for $|z| \leq \frac{1}{2}$ and $\varphi(z)=0$ for $|z| \geq 1$, and set
\begin{equation}
\hat{u}(x,\tau)=u(x,\tau)\varphi\left(\frac{x_3}{\rho(\tau)}\right).
\end{equation}

\bigskip

We recall from Angenent-Daskalopoulos-Sesum that there are shrinkers
\begin{align}
\Sigma_a &= \{ \textrm{surface of revolution with profile } r=u_a(z), 0\leq z \leq a\},\\
\tilde{\Sigma}_b &= \{ \textrm{surface of revolution with profile } r=\tilde{u}_b(z), 0\leq z <\infty\}\nonumber
\end{align}
as illustrated in \cite[Fig. 1]{ADS}, see also \cite{KM}. We will refer to these shrinkers as ADS-shrinkers and KM-shrinkers, respectively. The parameter $a$ captures where the curves $u_a$ meet the $z$-axis, namely $u_a(a)=0$, and the parameter $b$ is the asymptotic slope of the curves $u_b$, namely $\lim_{z\to \infty} u_b'(z)=b$. A detailed description of these shrinkers can be found in \cite[Sec. 8]{ADS}. In particular, the shrinkers can be used for barrier arguments as well as for calibration arguments. To describe the latter, we fix suitable parameters $a_0,b_0,L_0<\infty$ and consider the region $F^+\subset \mathbb{R}^3$ bounded by $\Sigma_{a_0}$, $\Sigma_{b_0}$ and $\{z= L_0\}$. The region $F^+$ is foliated by $\{\Sigma_a\}_{a\geq a_0}$, $\Sigma$ and $\{\tilde{\Sigma}_b \}_{b\leq b_0}$. Denoting by $\nu_{\textrm{fol}}$ the outward unit normal of this family, by \cite[Lemma 4.10]{ADS} we have that
\begin{equation}\label{eq_calibration}
\textrm{div}(e^{-x^2/4} \nu_{\textrm{fol}})=0,
\end{equation}
i.e. the shrinker family forms a calibration for the Gaussian area. Similarly, the region $F^-=\{ (x_1,x_2,-x_3)\, |\,  x\in F^+\}$ is calibrated. Let $F=F^-\cup F^+$.

\begin{proposition}\label{calibrated region}
There exists $\mathcal{T}>-\infty$ independent of $X_0$ such that\footnote{Recall that we assume that $Z(X_0)\leq 1$.}
\begin{align}
\bar{M}_{\tau} \cap \{|x_3|  \geq L_0\} \subset F
\end{align}
for all $\tau\leq \mathcal{T}$.
\end{proposition}

\begin{proof}
By Corollary \ref{cor_barrier} (barrier for the rescaled flow), $\bar M_{\tau}\cap \{x_3 \geq L_0\}$ is contained in the region bounded by $\tilde{\Sigma}_{b_0}$ and $\{x_3= L_0\}$ for sufficiently negative $\tau$. Moreover, since $\bar{M}_{\tau}$ is a graph with small norm over a long cylinder, it can not cross $\Sigma_{a_0}\cap\{ x_3\geq L_0\}$. This implies the assertion.
\end{proof}

Denote by $\Delta_\tau$ the region bounded by $\bar{M}_\tau$ and $\Sigma$.

\begin{proposition}[{c.f. \cite[Prop. 2.2]{BC}}]\label{prop_calibration}
\label{calibraion}
For all $L \in [L_0,\rho(\tau)]$ and $\tau \leq \mathcal{T}$ we have the Gaussian area estimate
\begin{multline}
\int_{\bar{M}_\tau \cap \{ |x_3| \geq L  \}} \, e^{-\frac{|x|^2}{4}} - \int_{\Sigma \cap \{|x_3| \geq L  \}} \, e^{-\frac{|x|^2}{4}}\\
\geq - \int_{\Delta_\tau\cap \{|x_3|=L\} }  \, e^{-\frac{|x|^2}{4}}\, | N \cdot \nu_{\text{fol}}|.
\end{multline}
\end{proposition}

\begin{proof}
Consider the region $\Delta_\tau\cap \{ |x_3|\geq L\} \cap \{ |x|\leq R \}$. Integrating \eqref{eq_calibration} over this region, and using the divergence theorem and Proposition \ref{calibrated region}, we obtain
\begin{multline}
\int_{\bar{M}_\tau \cap \{ |x_3| \geq L  \}} \, e^{-\frac{|x|^2}{4}} - \int_{\Sigma \cap \{|x_3| \geq L  \}} \, e^{-\frac{|x|^2}{4}}\\
 \geq - \int_{\Delta_\tau\cap \{|x_3|=L\} }  \, e^{-\frac{|x|^2}{4}}\, | N \cdot \nu_{\text{fol}}|
 - \int_{\Delta_\tau\cap \{|x|=R\} }  \, e^{-\frac{|x|^2}{4}}\, | N \cdot \nu_{\text{fol}}| ,
\end{multline}
where $N$ denotes the unit normal of the boundary. By the entropy bound we have that
\begin{equation}
\int_{\Delta_\tau\cap \{|x|=R\} }  \, e^{-\frac{|x|^2}{4}}\, | N \cdot \nu_{\text{fol}}|\leq C R^2 e^{-\frac{R^2}{4}},
\end{equation}
and passing $R\to \infty$ the assertion follows.
\end{proof}

The next proposition shows that closeness to the cylinder in the region $\{|x_3|\leq\frac{L}{2}\}$ implies closeness to the cylinder in the larger region $\{|x_3|\leq L\}$:

\begin{proposition}[{c.f. \cite[Prop. 2.3]{BC}, \cite[Lem. 4.7]{ADS}}]\label{Gaussian density analysis}
The graph function $u$ satisfies the integral estimates
\begin{equation}
\int_{\Sigma \cap \{|x_3| \leq L\}} e^{-\frac{|x|^2}{4}} \, |\nabla u(x,\tau)|^2 \leq C \int_{\Sigma \cap \{|x_3| \leq \frac{L}{2}\}} e^{-\frac{|x|^2}{4}} \, u(x,\tau)^2
\end{equation}
and 
\begin{equation}
\int_{\Sigma \cap \{\frac{L}{2} \leq |x_3| \leq L\}} e^{-\frac{|x|^2}{4}} \, u(x,\tau)^2 \leq CL^{-2} \int_{\Sigma \cap \{|x_3| \leq \frac{L}{2}\}} e^{-\frac{|x|^2}{4}} \, u(x,\tau)^2
\end{equation}
for all $L \in [L_0,\rho(\tau)]$ and $\tau \leq \mathcal{T}$, where $C<\infty$ is a numerical constant.
\end{proposition}

\begin{proof}
By the low entropy assumption we have
\begin{equation}
\int_{\bar{M}_\tau} e^{-\frac{|x|^2}{4}} \leq \int_\Sigma e^{-\frac{|x|^2}{4}}.
\end{equation}
Using this and Proposition \ref{prop_calibration}, the rest of the proof is as in \cite[proof of Prop. 2.3]{BC}.
\end{proof}

Recall that $\bar{M}_\tau$ is expressed as graph of a function $u(x,\tau)$ over $\Sigma\cap B_{2\rho(\tau)}(0)$ satisfying the estimate \eqref{small_graph}. Using that $\bar{M}_\tau$ moves by rescaled mean curvature flow one obtains:

\begin{lemma}[{c.f. \cite[Lem. 2.4]{BC}}]\label{Error u-PDE}
The function $u(x,\tau)$ satisfies 
\begin{equation}
\partial_\tau u = \mathcal{L} u  + E,
\end{equation}
where $\mathcal L$ is the linear operator on $\Sigma$ defined by
\begin{equation}\label{def_oper_ell}
\mathcal{L} f = \Delta f - \frac{1}{2} \, \langle x^{\text{\rm tan}},\nabla f \rangle + f,
\end{equation}
and where the error term can be estimated by
\begin{equation}\label{error-C0 norm}
|E| \leq C\sigma(\tau)\rho^{-1}(\tau)( |u| + |\nabla u|)
\end{equation}
for $\tau \leq \mathcal{T}$.
\end{lemma}

\begin{proof}
The proof is similar to \cite[proof of Lem. 2.4]{BC}.
\end{proof}  

Denote by $\mathcal{H}$ the Hilbert space of all functions $f$ on $\Sigma$ such that 
\begin{equation}\label{def_norm}
\|f\|_{\mathcal{H}}^2 = \int_\Sigma \frac{1}{4\pi} e^{-\frac{|x|^2}{4}} \, f^2 < \infty.
\end{equation}
 
\begin{lemma}[{c.f. \cite[Lem. 2.5]{BC}}]
\label{Error hat.u-PDE 1}
The function $\hat{u}(x,\tau) = u(x,\tau) \, \varphi \big ( \frac{x_3}{\rho(\tau)} \big )$ satisfies 
\begin{equation}
\partial_\tau \hat{u} = \mathcal{L} \hat{u}  + \hat{E},
\end{equation}
where
\begin{equation}
\|\hat{E}\|_{\mathcal{H}} \leq C\rho^{-1} \, \|\hat{u}\|_{\mathcal{H}}
\end{equation}
for $\tau \leq \mathcal{T}$.
\end{lemma}

\begin{proof}
As in \cite[page 8]{BC} we compute
\begin{multline} 
\hat{E} = E \, \varphi \Big ( \frac{x_3}{\rho(\tau)} \Big ) - \frac{2}{\rho(\tau)} \, \frac{\partial u}{\partial z} \, \varphi' \Big ( \frac{x_3}{\rho(\tau)} \Big ) - \frac{1}{\rho(\tau)^2} \, u \, \varphi'' \Big ( \frac{x_3}{\rho(\tau)} \Big ) \\
+ \frac{x_3}{2\rho(\tau)} \, u \, \varphi' \Big ( \frac{x_3}{\rho(\tau)} \Big ) - \frac{x_3 \rho'(\tau)}{\rho(\tau)^2} \, u \, \varphi' \Big ( \frac{x_3}{\rho(\tau)} \Big ). 
\end{multline} 
If $|x_3| \leq \frac{\rho(\tau)}{2}$, then Lemma \ref{Error u-PDE} gives
\begin{equation}
|\hat{E}|=|E| \leq C\sigma(\tau)\rho^{-1}(\tau)( |u| + |\nabla u| ).
\end{equation}
If $\frac{\rho(\tau)}{2} \leq |x_3| \leq \rho(\tau)$, then using equation \eqref{univ_fns} we obtain
\begin{equation}
|\hat{E}| \leq C |u| + C \rho^{-1}(\tau)  |\nabla u|.
\end{equation}
Using also Proposition \ref{Gaussian density analysis} and equation \eqref{univ_fns}, we infer that
\begin{align} 
\int_\Sigma e^{-\frac{|x|^2}{4}} \, |\hat{E}|^2 
&\leq \frac{C\sigma^2}{\rho^2} \int_{\Sigma \cap \{|x_3| \leq \frac{\rho(\tau)}{2}\}} e^{-\frac{|x|^2}{4}} \, u^2 +C \int_{\Sigma \cap \{\frac{\rho(\tau)}{2} \leq |x_3| \leq \rho(\tau)\}} e^{-\frac{|x|^2}{4}} \, u^2\nonumber \\ 
&\quad+\frac{C }{\rho^2} \int_{\Sigma \cap \{|x_3| \leq \rho(\tau)\}} e^{-\frac{|x|^2}{4}} \, |\nabla u|^2\nonumber \\ 
&\leq \frac{C}{\rho^2} \int_{\Sigma \cap \{|x_3| \leq \frac{\rho(\tau)}{2}\}} e^{-\frac{|x|^2}{4}} \, u^2 \nonumber \\ 
&\leq \frac{C}{\rho^2}  \int_\Sigma e^{-\frac{|x|^2}{4}} \, \hat{u}^2 
\end{align}
for $\tau\leq\mathcal{T}$. Thus, we obtain the desired result.
\end{proof}

Let us recall some facts from \cite{BC} about the operator $\mathcal L$ defined in \eqref{def_oper_ell}. In cylindrical coordinates this operator takes the form
\begin{equation}
\mathcal L=\frac{\partial^2}{\partial z^2} f + \frac{1}{2} \, \frac{\partial^2}{\partial \theta^2} f - \frac{1}{2} \, z \, \frac{\partial}{\partial z} f + f.\end{equation}
Analysing the spectrum of $\mathcal L$, the Hilbert space $\mathcal H$ from \eqref{def_norm} can be decomposed as
\begin{equation}
\mathcal H = \mathcal{H}_+\oplus \mathcal{H}_0\oplus \mathcal{H}_-,
\end{equation}
where $\mathcal{H}_+$ is spanned by the four positive eigenmodes $1, z, \sin \theta, \cos \theta$, and $\mathcal{H}_0$ is spanned by the three zero-modes $z^2-2,z\cos\theta,z\sin\theta$. We have
\begin{align} 
&\langle \mathcal{L} f,f \rangle_{\mathcal{H}} \geq \frac{1}{2} \, \|f\|_{\mathcal{H}}^2 & \text{\rm for $f \in \mathcal{H}_+$,} \nonumber\\ 
&\langle \mathcal{L} f,f \rangle_{\mathcal{H}} = 0 & \text{\rm for $f \in \mathcal{H}_0$,} \\ 
&\langle \mathcal{L} f,f \rangle_{\mathcal{H}} \leq -\frac{1}{2} \, \|f\|_{\mathcal{H}}^2 & \text{\rm for $f \in \mathcal{H}_-$.} \nonumber
\end{align}

Consider the functions 
\begin{align}
&U_+(\tau) := \|P_+ \hat{u}(\cdot,\tau)\|_{\mathcal{H}}^2, \nonumber\\ 
&U_0(\tau) := \|P_0 \hat{u}(\cdot,\tau)\|_{\mathcal{H}}^2,\label{def_U_PNM} \\ 
&U_-(\tau) := \|P_- \hat{u}(\cdot,\tau)\|_{\mathcal{H}}^2, \nonumber
\end{align}
where $P_+, P_0, P_-$ denote the orthogonal projections to $\mathcal{H}_+,\mathcal{H}_0,\mathcal{H}_-$, respectively. Using Lemma \ref{Error hat.u-PDE 1} we obtain
\begin{align} 
&\frac{d}{d\tau} U_+(\tau) \geq U_+(\tau) - C\rho^{-1} \, (U_+(\tau) + U_0(\tau) + U_-(\tau)), \nonumber\\ 
&\Big | \frac{d}{d\tau} U_0(\tau) \Big | \leq C\rho^{-1} \, (U_+(\tau) + U_0(\tau) + U_-(\tau)), \label{U_PNM_system}\\ 
&\frac{d}{d\tau} U_-(\tau) \leq -U_-(\tau) + C\rho^{-1} \, (U_+(\tau) + U_0(\tau) + U_-(\tau)). \nonumber
\end{align}

To proceed, we need the following ODE-lemma.

\begin{lemma}[{Merle-Zaag \cite{MZ}}]\label{MZ ODE}
Let $x(\tau),y(\tau),z(\tau)$ be nonnegative absolutely continuous functions satisfying $(x+y+z)(\tau)>0$ and  
\begin{equation}
\liminf_{\tau\to-\infty}y(\tau)=0.
\end{equation}
Moreover, suppose for each $ \varepsilon\in (0,\frac{1}{100}]$, there exists a $\tau_0=\tau_0(\varepsilon)$ such that for $\tau\leq \tau_0$ the following holds
\begin{align}
& |x_\tau| \leq \varepsilon(x+y+z),\nonumber \\
& y_\tau \leq -y+\varepsilon(x+z),\\
&z_\tau \geq z-\varepsilon(x+y).\nonumber
\end{align}
Then, we have $y \leq 2\varepsilon(x+z)$ for $\tau\leq \tau_0$, and either $x$ is dominant in the sense that
\begin{equation}
y+z=o(x)
\end{equation}
as $\tau \to -\infty$, or $z$ is dominant in the more precise sense that
\begin{equation}
x+y \leq 100\varepsilon z
\end{equation}
for all $\tau \leq \tau_0(\varepsilon)$.
\end{lemma} 

\begin{remark}
The statement in \cite{MZ} is slightly different, but a careful inspection of their proof gives the variant of the lemma that we stated, see also \cite[Appendix B]{CM_ancient} for a detailed proof.
\end{remark}

\bigskip

If $(U_+ + U_0 + U_-)(\hat\tau) = 0$ for some $\hat{\tau}$, then Lemma \ref{Error hat.u-PDE 1} implies that $(U_+ + U_0 + U_-)(\tau) = 0$ for all $\tau\leq \hat \tau$, and by analytic continuation it follows that $\mathcal M$ is a round shrinking cylinder; a contradiction. We can thus apply the Merle-Zaag lemma (Lemma \ref{MZ ODE}) to conclude that either the neutral mode is dominant, i.e.
\begin{equation}
U_++U_- =o(U_0)
\end{equation}
or the plus mode is dominant, i.e.
\begin{equation}
U_-+U_0 \leq C\rho^{-1}U_+.
\end{equation}
We will analyze these two cases in turn in the following two sections. Recall also that our analysis above depends on a choice of $X_0$ and $\rho$. However, the following proposition shows that either the neutral mode is always dominant or the plus mode is always dominant.

\begin{proposition}\label{Plus or Neutral}
If a neck centered at some point $X_0$ satisfies $U_++U_-=o(U_0)$ for some admissible choice of $\rho_0$, then any other neck centered at any point $X$ also satisfy $U_++U_-=o(U_0)$ for any admissible choice of $\rho$. 
\end{proposition}

\begin{proof}
Any two necks converge to the same neck after rescaling as $\tau\to -\infty$. Thus, the above statement is obvious.
\end{proof}

\bigskip

\subsection{Fine analysis in the plus mode}\label{sec_plus_mode}

In this section, we assume that the plus mode is dominant. We recall from above this means that after fixing a center $X_0\in\mathcal{M}$ with $Z(X_0)\leq 1$, and a graphical scale function $\rho$, we have that
\begin{equation}\label{assumption_plus_dom}
U_-+U_0 \leq C\rho^{-1}U_+
\end{equation}
for all $\tau\leq \mathcal{T}$. As before, $C<\infty$ and $\mathcal{T}>-\infty$ denote constants that can change from line to line and are independent of the point $X_0\in\mathcal{M}$ with $Z(X_0)\leq 1$. The main goal of this section is to prove Theorem \ref{thm Neck asymptotic}.

\subsubsection{Graphical radius}

Using \eqref{U_PNM_system} and \eqref{assumption_plus_dom} we obtain
\begin{align}
\frac{d}{d\tau} U_+ \geq U_+ - C\rho^{-1} \, U_+.
\end{align}
By integrating this differential inequality, for every $\mu>0$ we can find a constant $\mathcal{T}(\mu)>-\infty$ such that
\begin{align}
U_+(\tau) \leq Ce^{(1-\mu)\tau}. 
\end{align}
for all $\tau\leq \mathcal{T}(\mu)$.
Recalling that $U_+=||P_+\hat{u}||_{\mathcal H}^2$ and using \eqref{assumption_plus_dom} we infer that
\begin{equation}
|| \hat u ||_{\mathcal H} \leq Ce^{\frac{(1-\mu)\tau}{2}}.
\end{equation}
By standard interpolation inequalities this implies
\begin{align}\label{eq u weak estimate}
\|u(\cdot,\tau)\|_{C^4([-10L_0,10L_0] \times [0,2\pi])} \leq Ce^{\frac{(1-\mu)\tau}{2}}
\end{align}
for all $\tau \leq \mathcal{T}(\mu)$. The estimate \eqref{eq u weak estimate} is not sharp. To improve it to a sharp estimates (i.e. to remove the $\mu$) we start with the following $C^0$-estimate:

\begin{proposition}\label{prop_c0est}
The rescaled mean curvature flow $\bar{M}_\tau$ satisfies the estimate
\begin{equation}\label{eq C^0 estimate}
\sup_{\bar{M}_\tau \cap \{|x_3| \leq e^{-\frac{\tau}{10}}\}} \big |x_1^2+x_2^2-2\big| \leq   e^{\frac{\tau}{10}}
\end{equation}
for all $\tau \leq \mathcal{T}$.
\end{proposition}
\begin{proof}

Fix $\mu$ sufficiently small, e.g. $\mu=\frac{1}{100}$. Recall from Section \ref{sec_fine_neck_setup} that there are  shrinkers $\Sigma_a$ and $\tilde{\Sigma}_b$ with profile functions $u_a$ and $\tilde{u}_b$.

As preparation for the barrier argument, observe that \eqref{eq u weak estimate} implies
\begin{align}\label{barrier_prep}
\sup_{\bar{M}_\tau\cap \{ |x_3|\leq L_0\} } (x^2_1+x_2^2) \leq 2+Ce^{\frac{(1-\mu)\tau}{2}},
\end{align}
and
\begin{align}\label{barrier_prep_inf}
\inf_{\bar{M}_\tau\cap \{ |x_3|\leq L_0\} } (x^2_1+x_2^2) \geq 2-Ce^{\frac{(1-\mu)\tau}{2}}.
\end{align}

To prove an upper bound, we will use the shrinkers $\tilde{\Sigma}_b$ as outer barriers. To this end, we first observe that for $z \geq 2\sqrt{2}$ the profile function $\tilde{u}_b$ satisfies the estimate
\begin{align}\label{prec_est_for_profile}
2+e^{-1}b^2z^2  \leq \tilde{u}_b(z)^2\leq 2+b^2z^2.
\end{align}
Indeed, \cite[Lem. 4.11]{ADS} and \cite[Sec. 8.12]{ADS} tell us that $\tilde{u}_b$ satisfies the differential equation
\begin{align}\label{ADS_prec1}
\tilde{u}_b'=\frac{w}{2 z \tilde{u}_b }(\tilde{u}_b^2-2),
\end{align}
for $z\geq 2\sqrt{2}$, where $w=w(z)$ is a function that satisfies
\begin{align}\label{ADS_prec2}
2 \leq w(z) \leq 2+\frac{16}{z^2}.
\end{align}
Combining \eqref{ADS_prec1} and \eqref{ADS_prec2} we see that
\begin{align}\label{profile_diff_inn}
0\leq \frac{d}{d z}\log \left( \frac{\tilde{u}_b(z)^2-2}{z^2}\right) \leq \frac{16}{z^3}.
\end{align}
Integrating \eqref{profile_diff_inn} from $z$ to $\infty$ and using that the asymptotic slope of $\tilde{u}_b$ equals $b$ we infer that
\begin{align}
0\leq \log b^2-\log \left( \frac{\tilde{u}_b(z)^2-2}{z^2}\right)\leq \int_{2\sqrt{2}}^{\infty}\frac{16}{x^3}dx=1,
\end{align}
which proves \eqref{prec_est_for_profile}.

Now, fixing $\hat{\tau}\leq \mathcal{T}$ we consider the shrinker $\tilde{\Sigma}_{b}$ with parameter
\begin{equation}
b=2\sqrt{eC} L_0^{-1} e^{\frac{(1-\mu)\hat{\tau}}{4}}.
\end{equation}
Using \eqref{barrier_prep} and \eqref{prec_est_for_profile}  we see that $\tilde{\Sigma}_b\cap \{x_3=L_0\}$ is outside of $\bar{M}_{\tau }\cap \{x_3=L_0\}$ for all $\tau \leq \hat{\tau}$. Moreover, recalling that $\tilde{u}_b$ has asymptotic slope $b$ and that $\bar{M}_{\tau}$ has vanishing asymptotic slope (see Corollary \ref{cor_barrier}) we also see that $\tilde{\Sigma}_b\cap \{x_3=h\}$ is outside of $\bar{M}_{\tau }\cap \{x_3=h\}$ for $h\geq h_0(\hat{\tau})$ large enough, and for all $\tau\leq \hat{\tau}$. Finally, $\tilde{\Sigma}_b\cap \{L_0\leq x_3 \leq h\}$ is outside of $\bar{M}_{\tau}\cap \{L_0\leq x_3 \leq h\}$ for $-\tau$ large enough.
Hence, by the avoidance principle, $\tilde{\Sigma}_{b_1}\cap \{x_3\geq L_0\}$ is outside of $\bar{M}_\tau\cap \{x_3\geq L_0\}$ for all $\tau \leq \hat{\tau}$.
Repeating the same argument for $x_3$ replaced by $-x_3$, and using \eqref{prec_est_for_profile}, we conclude that
\begin{equation}\label{upper_c0}
\sup_{\bar{M}_\tau \cap \{|x_3| \leq e^{-\frac{\tau}{10}}\}} (x_1^2+x_2^2) \leq 2 + e^{\frac{\tau}{10}}\end{equation}
for $\tau\leq\mathcal{T}$.

Similarly, considering the shrinkers $\Sigma_a$ as inner barriers as in \cite[p. 10--11]{BC} and using \eqref{barrier_prep_inf} we obtain
\begin{equation}
\inf_{\bar{M}_\tau \cap \{|x_3|\leq e^{-\frac{\tau}{10}}\}} (x_1^2+x_2^2) \geq 2 - e^{\frac{\tau}{10}}
\end{equation}
for $\tau\leq\mathcal{T}$. Together with \eqref{upper_c0} this proves the assertion.
\end{proof}

We will now use the $C^0$-estimate to prove that the rescaled flow $\bar{M}_\tau$ is cylindrical over an exponentially expanding region for $\tau\to -\infty$.

\begin{proposition}\label{higherderivatives} For $\tau\leq\mathcal{T}$ the rescaled mean curvature flow $ \bar{M}_\tau$ can be written as graph of a function $v(\cdot,\tau)$ over the cylinder $\Sigma\cap \{|x_3|\leq \tfrac12 e^{-\tau/10}\}$ with the estimate
\begin{equation}
|| v||_{C^6(\Sigma\cap \{|x_3|\leq \tfrac12 e^{-\tau/10}\})} \leq Ce^{\tau/10}.
\end{equation}
\end{proposition}

\begin{proof}
We first observe that for $ \tau\leq\mathcal{T}$, every point $\bar{X}=(\bar{x},\tau)$, where $\bar{x}\in \bar{M}_\tau\cap \{|x_3|\leq \tfrac34 e^{-\tau/10}\}$, has regularity scale comparable to $1$, namely
\begin{equation}\label{eq_reg_scale_bound}
C^{-1}\leq  R(\bar{X})\leq C
\end{equation}
for some uniform $C<\infty$, where $R(\bar{X})$ denotes the regularity scale of the rescaled flow.

Indeed, if $R(\bar X)\gg 1$ then we could write $\bar{M}_\tau$ is graph with small norm over a disk of radius $2$ centered at $\bar X$, contradicting Proposition \ref{prop_c0est}. If on the other hand $R(\bar X)\ll 1$ then by the argument from Section \ref{sec_necks_back_in_time} for some $\tau'<\tau$ the rescaled flow would be $\varepsilon$-close to a cylinder or sphere with center $\bar X$ and radius $\tfrac{1}{100}$, contradicting again Proposition \ref{prop_c0est}.

To continue, note that the geometric meaning of Proposition \ref{prop_c0est} is that our surface $\bar{M}_\tau$ is trapped between two cylinders of almost equal radii. Combining this with the regularity scale bound \eqref{eq_reg_scale_bound} we see that at every point $\bar{x}\in \bar{M}_\tau\cap \{|x_3|\leq \tfrac34 e^{-\tau/10}\}$ the tangent plane must be almost parallel to the one of the cylinder $\Sigma$ at $\pi(\bar X)$, where $\pi$ denotes the nearest neighbor projection. It follows that $\bar{M}_\tau\cap \{|x_3|\leq \tfrac34 e^{-\tau/10}\}$ can be written as graph of a function $v(\cdot,\tau)$ over the cylinder $\Sigma$ with small $C^1$-norm. It is clear by looking at the middle region, that there is one single layer. Moreover, the $C^1$-norm bound and the regularity scale bound \eqref{eq_reg_scale_bound} imply
\begin{equation}
\|\nabla^2 v\|_{C^2(|x_3| \leq \frac{3}{4} e^{-\tau/10})} \leq C,
\end{equation}
and thus $v$ is the solution of a linear uniformly parabolic equation\footnote{$v_t=a_{ij}v_{ij}+b_iv_i+cv$ where $a=a(\nabla v,v)$, $b=b(\nabla v,v)$, $c=c(\nabla v,v)$. c.f. \eqref{MCF_graph}.} with coefficients of uniformly small $C^1$-norm. Hence, using standard interior estimates the assertion follows.
\end{proof}

\bigskip

We now repeat the process from Section \ref{sec_fine_neck_setup} with improved functions $\rho$ and $\sigma$. Namely, by Proposition \ref{higherderivatives} we can now choose
\begin{equation}
\rho(\tau)=e^{-\tau/20}, \textrm{ and } \sigma(\tau)=C e^{\tau/20}.
\end{equation}
With this new choice of $\rho$ and $\sigma$, we can write $\bar{M}_\tau$ as graph of a function $u(\cdot, \tau)$ defined over the exponentially large domain $\Sigma\cap B_{2\rho(\tau)}$ such that it satisfies the estimate \eqref{small_graph} for $\tau\leq \mathcal{T}$.

\begin{proposition}\label{thm_sharp_decay_plus}
For $\tau \leq \mathcal{T}$ the function $\hat{u}(x,\tau)=u(x,\tau)\varphi\left(\frac{x_3}{e^{-\tau/20}}\right)$ satisfies the estimate
\begin{equation}
\|\hat u\|_{\mathcal{H}}  \leq Ce^{\frac{\tau}{2}}.
\end{equation}
In particular, we have 
\begin{align}
\sup_{\bar M_{\tau}\cap B_{10L_0}(0)}|x_1^2+x_2^2-2| \leq Ce^{\frac{\tau}{2}}.
\end{align}
\end{proposition}

\begin{proof}
We define $U_+,U_0,U_-$ by the formulas from \eqref{def_U_PNM}, where now $\hat{u}(x,\tau)=u(x,\tau)\varphi(\tfrac{x_3}{e^{-\tau/20}})$. Using Lemma  \ref{Error hat.u-PDE 1} and Proposition \ref{Plus or Neutral} we then get the inequalities
\begin{equation}\label{somebodydominant}
U_0+U_{-}\leq Ce^{\tfrac{\tau}{20}} U_+
\end{equation}
and
\begin{align}
\frac{d}{d\tau} U_+ \geq U_+-Ce^{\frac{\tau}{20}}U_+.
\end{align} 
Rewriting the latter inequality as $\frac{d}{d \tau}\log (e^{-\tau} U_+) \geq  -Ce^{\frac{\tau}{20}}$
and integrating from $\tau$ to $\mathcal{T}$ yields the estimate
\begin{equation}
U_+(\tau)\leq Ce^\tau,
\end{equation}
which together with \eqref{somebodydominant} implies that
\begin{equation}\label{est_hatu}
\|\hat u\|_{\mathcal{H}}   \leq Ce^{\frac{\tau}{2}}.
\end{equation}
Using standard interior estimates we conclude that
\begin{align}
\sup_{\bar M_{\tau}\cap B_{10L_0}(0)}|x_1^2+x_2^2-2| \leq Ce^{\frac{\tau}{2}}.
\end{align}
This proves the proposition.
\end{proof}

\bigskip

\subsubsection{Constant functions cannot be dominant}\label{const_non_don_sec}

To obtain refined information it is useful to decompose
\begin{equation}
P_+=P_{1/2}+P_1,
\end{equation}
where 
 $P_{1/2}$ is the projection to the span of $z,\cos\theta,\sin\theta$, and $P_1$ is the projection to the span of the constant function $1$. Accordingly, we can decompose
\begin{equation}
U_+:=\|P_+\hat u\|^2_{\mathcal{H}}=\|P_{1/2}\hat u\|^2_{\mathcal{H}}+\|P_{1}\hat u\|^2_{\mathcal{H}}
=:U_{1/2}+U_{1}.
\end{equation}
Using this decomposition, Lemma \ref{Error hat.u-PDE 1}, and the assumption \eqref{assumption_plus_dom} that the plus mode is dominant\footnote{By Proposition \ref{Plus or Neutral} we can indeed use this inequality with $\rho(
\tau)=e^{-\tau/20}$.} we obtain
\begin{align}
\Big|\frac{d}{d\tau} U_{1/2}-U_{1/2}\Big| \leq C\rho^{-1}(U_{1/2}+U_{1}),\\
\Big|\frac{d}{d\tau} U_1-2U_1\Big| \leq C\rho^{-1}(U_{1/2}+U_{1}).
\end{align}
Applying the Merle-Zaag lemma (Lemma \ref{MZ ODE}) with $x=e^{-\tau}U_{1/2}$, $y=0$, $z=e^{-\tau}U_1$ we infer that either the $U_{1/2}$ is dominant, i.e.
\begin{equation}\label{subcase1}
U_1=o(U_{1/2})
\end{equation}
or the constant function $1$ is dominant, i.e.
\begin{equation}\label{subcase2}
U_{1/2}\leq C\rho^{-1} U_1.
\end{equation}

The next proposition shows that $U_1$ cannot be dominant.

\begin{proposition}\label{thm_const_not_dom}
It must be the case that $U_1=o(U_{1/2})$.
\end{proposition}

\begin{proof}
Assume towards a contradiction that the estimate $U_1=o(U_{1/2})$ does not hold. By the above discussion, we then have 
\begin{equation}\label{assumption_second_plus_dom}
U_-+U_0+U_{1/2}\leq Ce^{\frac{\tau}{20}}U_{1}
\end{equation}
for all $\tau \leq  \mathcal{T}$. Using Lemma \ref{Error hat.u-PDE 1} it follows that
\begin{align}\label{U2 rough bound}
\bigg|\frac{d}{d\tau} U_1-2U_1\bigg| \leq Ce^{\frac{\tau}{20}}U_{1}.
\end{align}
These two inequalities imply that
\begin{equation}
|| \hat u ||_{\mathcal H} \leq Ce^{\tau },
\end{equation}
(see the proof of Proposition \ref{thm_sharp_decay_plus}), and thus in particular
\begin{align} 
\|u(\cdot,\tau)\|_{C^4([-10L_0,10L_0] \times [0,2\pi])} \leq Ce^{ \tau}
\end{align}
for all $\tau\leq \mathcal{T}$.

Moreover, integrating \eqref{U2 rough bound} and using \eqref{assumption_second_plus_dom} yields that $e^{-2\tau}U_1$ converges to some positive limit $K_0 > 0$ for $\tau\to-\infty$ with the estimate
\begin{equation}
|e^{-2\tau}U_1-K_0^2|\leq Ce^{\tfrac{\tau}{20}}
\end{equation}
for $\tau \leq \mathcal{T}$.

We will now argue similarly as in \cite[proof of Lem. 5.11]{ADS}: By the above estimates, there exists a constant $K$ such that the rescaled flow $\bar M_\tau$ satisfies
\begin{align}
x_1^2+x_2^2=2(1+Ke^{\tau})+O(e^{\frac{21}{20}\tau})
\end{align}
uniformly on the bounded interval $[-100,100]$, and more precisely
\begin{equation}
 K_0 = |K|\int_{\Sigma}(\tfrac{e}{8\pi})^{\frac{1}{4}}\tfrac{1}{4\pi} e^{-\frac{|x|^2}{4}}=|K|(\tfrac{\pi}{2e})^{\frac{1}{4}}.
\end{equation}
Hence, recalling that $\tau = -\log(-t)$, the original flow satisfies
\begin{align}
x_1^2+x_2^2&=(-2t)(1+ K (-t)^{-1} )+O(|t|^{-\frac{1}{20}})\\
&= 2( K-t) +O(|t|^{-\frac{1}{20}})
\end{align}
for $|x_3|\leq 100 \sqrt{-t}$, where we assume for ease of notation that $X_0=(0,0)$. Now, if we instead rescale with respect to the new center $\tilde X_0=(0,K)$, i.e. if we set $\tau= - \log( K-t)$, then the corresponding rescaled flow $\tilde M_{\tau}$ satisfies
\begin{align}\label{eqnothat1}
x_1^2+x_2^2 = 2 +O(e^{\frac{21}{20}\tau}).
\end{align}
Write $\tilde M_\tau$ as the graph of a function $\tilde u(\cdot,\tau)$ over $\Sigma\cap \{-\rho(\tau)\leq x_3 \leq\rho(\tau) \}$ with $\rho(\tau)=e^{-\frac{\tau}{20}}$ and consider the spectrum.  
Arguing as in the proof of Proposition \ref{Plus or Neutral} we see that $\tilde U_1$ is dominant. Hence,
\begin{equation}\label{eq_hat_k}
|e^{-2\tau}\tilde{U}_1-\tilde{K}|\leq Ce^{\tfrac{\tau}{20}}
\end{equation}
for some $\tilde K>0$. However, using \eqref{eqnothat1} we can directly compute 
\begin{align}
\tilde U_1= \bigg|\int \varphi  \tilde u \cdot (\tfrac{e}{2\pi})^{\frac{1}{4}} \tfrac{1}{4\pi}  e^{-\frac{|x|^2}{4}}\bigg|\leq Ce^{\frac{21}{10}\tau},
\end{align}
This contradicts \eqref{eq_hat_k}, and finishes the proof of the proposition.
\end{proof}

\bigskip

\subsubsection{The fine neck theorem}

By Proposition \ref{thm_sharp_decay_plus} and Proposition \ref{thm_const_not_dom} we can now assume that 
\begin{equation}\label{assumption_plus_dom1}
U_-+U_0 \leq C\rho^{-1}U_+,
\end{equation}
and
\begin{equation}\label{u12dom}
U_1=o(U_{1/2}),
\end{equation}
where $\rho(\tau)=e^{-\tau/20}$.
Recall in particular that Proposition \ref{thm_sharp_decay_plus}  gives
\begin{equation}\label{decay_hatuu}
\|\hat u\|_{\mathcal{H}}  \leq Ce^{\frac{\tau}{2}}.
\end{equation}
Moreover, using in addition equation \eqref{u12dom} and the assumption that our solution is not the round shrinking cylinder we see that
\begin{equation}\label{12isdom}
\lim_{\tau\to -\infty} e^{-\tau}U_{1/2} >0.
\end{equation}

Furthermore, by using \eqref{decay_hatuu}, we establish the following coarse estimate in a compact region.
\begin{lemma}\label{lemma_u_C2_coarse estimate} For sufficiently large $-\tau$, we have
\begin{equation}\label{u_C2_coarse estimate}
\|u(\cdot,\tau)\|_{C^2(\Sigma \cap B_{10L_0}(0))} \leq C e^{ \frac{40}{81}\tau}.
\end{equation}
\end{lemma}

\begin{proof}
Given a fixed sufficiently large $-\tau_0$, we define parabolic regions $Q_r$ in $\Sigma \times \mathbb{R}$ by
\begin{equation}
Q_r=\{(\theta,z,\tau): \theta \in (0,2\pi], |z| \leq 100L_0r, -(100L_0r)^2\leq \tau_0-\tau \leq 0 \}.
\end{equation}

Let $\delta=(\tfrac{1}{10})^{1/8}$, and consider a smooth nonnegative cutoff function $\eta_1(\theta,z,\tau)$ satisfying $\eta_1=0$ on $\partial Q_1$ and $\eta_1=1$ in $Q_\delta$. Then, the function $\bar u_1= u\eta_1 $ satisfies
\begin{equation}
\partial_\tau \bar u_1=\mathcal{L} \bar u_1 +g_1,
\end{equation}
where $|g_1| \leq C( |u|+|\nabla u|) $ due to Lemma \ref{Error u-PDE}. Using Proposition \ref{Gaussian density analysis} this implies
\begin{equation}
\int_{Q_1} |g_1|^3 \leq C \int_{Q_1}|u|^3+|\nabla u|^3  \leq C \|u\|^2_{L^2(Q_1)}\|u\|_{C^1(Q_1)}\leq Ce^{\tau_0}.
\end{equation}
Therefore, \cite[Thm. 3.14]{Wang_regularity} yields
\begin{equation}
\|u\|_{L^{\infty}(Q_\delta)}\leq \|\bar u_1\|_{L^{\infty}(Q_1)} \leq Ce^{\frac{\tau_0}{3}}.
\end{equation} 
Hence, the standard interior regularity implies $\|u\|_{C^{2}(Q_{\delta^2})} \leq Ce^{\frac{\tau_0}{3}}$.

Next, we consider a new cutoff $\eta_2$ satisfying $\eta_2=0$ on $\partial Q_{\delta^2}$ and $\eta_2=1$ in $Q_{\delta^3}$. Then, $\bar u_2= u\eta_2 $ satisfies $\partial_\tau \bar u_2=\mathcal{L} \bar u_2 +g_2$, where 
\begin{equation}
\int_{Q_{\delta^2}} |g_2|^3 \leq C \int_{Q_{\delta^2}}|u|^3+|\nabla u|^3  \leq C \|u\|^2_{L^2(Q_{\delta^2})}\|u\|_{C^1(Q_{\delta^2})}\leq Ce^{\frac{4}{3} \tau_0}.
\end{equation}
Hence, we have $\|u\|_{C^{2}(Q_{\delta^4})} \leq Ce^{\frac{4}{9}\tau_0}$. We repeat this process twice more so that we can obtain the desired result.
\end{proof}

\bigskip

We will now express $P_+\hat u\in\mathcal{H}_+$ as linear combination of the four eigenfunctions $z,\cos\theta,\sin\theta$ and $1$. Namely, let
\begin{align}\label{coeffs_abcd}
&a^X(\tau)=(\tfrac{e}{8\pi})^{\frac{1}{4}}\int z \hat u^X(x,\tau) \tfrac{1}{4\pi}e^{-\frac{|x|^2}{4}},\\
&b^X(\tau)=(\tfrac{2e}{\pi})^{\frac{1}{4}}\int \cos\theta \hat u^X(x,\tau) \tfrac{1}{4\pi}e^{-\frac{|x|^2}{4}},\\
&c^X(\tau)=(\tfrac{2e}{\pi})^{\frac{1}{4}}\int \sin\theta \hat u^X(x,\tau) \tfrac{1}{4\pi}e^{-\frac{|x|^2}{4}},\\
&d^X(\tau)=(\tfrac{e}{2\pi})^{\frac{1}{4}}\int  \hat u^X(x,\tau)\tfrac{1}{4\pi} e^{-\frac{|x|^2}{4}},
\end{align}
where the superscript $X$ is to remind us that all these coefficients depend (a priori) on the choice of base point $X$.
Then, we have
\begin{align}
P_+\hat u^X=a^Xz+b^X\cos\theta+c^X\sin\theta+d^X. \label{def P_+hat u}
\end{align}
Moreover, $U_+^X=\|P_+\hat u^X\|_{\mathcal{H}}^2$ is given by a sum of coefficients squared:
\begin{align}
U_+^X=2^{-\frac{1}{2}}\pi^{\frac{1}{2}}e^{-\frac{1}{2}} \Big(4|a^X|^2+ |b^X|^2+ |c^X|^2+2|d^X|^2\Big) \label{U_+ in terms of a,b,c,d}.
\end{align}

\begin{proposition}\label{est_coeffs}
There coefficients defined in \eqref{coeffs_abcd} satisfy the estimates
\begin{equation}\label{coeff_est1}
|d^X(\tau) |\leq Ce^{\tfrac{11}{20}\tau},
\end{equation}
and
\begin{equation}\label{coeff_est2}
|e^{-\tfrac{\tau}{2}} a^X(\tau)-\bar{a}^X|+|e^{-\tfrac{\tau}{2}}b^X(\tau)-\bar{b}^X|+|e^{-\tfrac{\tau}{2}}c^X(\tau)-\bar{c}^X | \leq Ce^{\tfrac{\tau}{20}}
\end{equation}
for some numbers $\bar{a}^X, \bar{b}^X, \bar{c}^X$ that might depend on $X$.
\end{proposition}

\begin{proof}
Using Lemma \ref{Error hat.u-PDE 1} and $\mathcal L 1 = 1$ we compute
\begin{align}
\frac{d}{d\tau} d^X(\tau)
&=(\tfrac{e}{2\pi})^{\frac{1}{4}}  \int (\mathcal{L}\hat u+\hat E) \tfrac{1}{4\pi} e^{-\frac{|x|^2}{4}}\\
&=(\tfrac{e}{2\pi})^{\frac{1}{4}}  \int (\hat u+\hat E)  \tfrac{1}{4\pi}e^{-\frac{|x|^2}{4}}
=d^X(\tau)+(\tfrac{e}{2\pi})^{\frac{1}{4}}  \int   \tfrac{\hat E}{4\pi}e^{-\frac{|x|^2}{4}}.
\end{align}
Hence,
\begin{align}
\Big|\frac{d}{d\tau}\left(e^{-\tau} d^X(\tau)\right) \Big| \leq Ce^{-\tau}\|\hat E\|_{\mathcal{H}}
\leq Ce^{-\tau}\rho^{-1}\|\hat u\|_{\mathcal{H}}\leq
Ce^{-\frac{9}{20}\tau}.
\end{align}
Integrating from $\tau$ to $\mathcal{T}$ this implies \eqref{coeff_est1}.

In a similar manner, we compute
\begin{multline}
\Big|\frac{d}{d\tau}\left(e^{-\frac{\tau}{2}} a^X(\tau)\right) \Big|
+\Big|\frac{d}{d\tau}\left(e^{-\frac{\tau}{2}} b^X(\tau)\right) \Big|
+\Big|\frac{d}{d\tau}\left(e^{-\frac{\tau}{2}} c^X(\tau)\right) \Big|
\\
\leq C e^{-\frac{\tau}{2}}\rho^{-1}\|\hat u\|_{\mathcal{H}}\leq Ce^{\frac{\tau}{20}}.
\end{multline}
Integrating from $-\infty$ to $\tau$ this implies \eqref{coeff_est2}.
\end{proof}

We are now ready to prove the main theorem of this section:

\begin{theorem}\label{thm Neck asymptotic}
Let $\mathcal M$ be an ancient low entropy flow which is not the round shrinking sphere, round shrinking cylinder or flat plane. If the plus mode is dominant, then there are constants $\bar{a}=\bar{a}(\mathcal M)\neq 0$, $C=C(\mathcal M)<\infty$ and a decreasing function $\mathcal{T}:\mathbb{R}_+\to\mathbb{R}_{-}$ (depending on $\mathcal M$) with the following significance.

For every $X\in\mathcal M$ the graph function $u^X(\cdot,\tau)$ of the rescaled flow $\bar{M}^X_\tau$ satisfies the estimates
\begin{equation}\label{main_thm_est1}
\|e^{-\frac{\tau}{2}}\hat{u}^X(x,\tau)-\bar a z-\bar b^X \cos\theta-\bar c^X \sin \theta\|_{\mathcal{H}}\leq C e^{\frac{\tau}{40}},
\end{equation}
and
\begin{align}\label{main_thm_est2}
\sup_{|x_3|\leq 10L_0}\big| e^{-\frac{\tau}{2}}u^X(x,\tau)-\bar a z-\bar b^X \cos\theta-\bar c^X \sin \theta \big| \leq C e^{\frac{\tau}{160}}
\end{align}
for $\tau \leq \mathcal{T}(Z(X))$. Here, the constant $\bar{a}$ is independent of $X$, and $\bar b^X$ and $\bar c^X$ are numbers that may depend on $X$ and satisfy
\begin{equation}\label{main_thm_est3}
|\bar b^X| + |\bar c^X|\leq C.
\end{equation}
\end{theorem}

\begin{proof}
By scaling we can pretend without essential loss of generality that we only work with center points $X$ satisfying $Z(X)\leq 1$ (this condition is only used to figure out more easily which constants are uniform in $X$).

Consider the difference
\begin{equation}
D^X:= \hat{u}^X-e^{\tau/2}\left( \bar{a}^X z + \bar{b}^X \cos \theta + \bar{c}^X \sin \theta\right).
\end{equation}
Using equation \eqref{def P_+hat u} and Proposition \ref{est_coeffs} we see that
\begin{equation}
|D^X|\leq |\hat{u}^X - P_{+}\hat{u}^X| + C(1+|z|)e^{\tfrac{11}{20}\tau}.
\end{equation}
Since by \eqref{assumption_plus_dom1} and \eqref{decay_hatuu} we have
\begin{equation}
U_{-}+ U_{0}\leq Ce^{\tfrac{21}{20}\tau},
\end{equation}
it follows that
\begin{equation}\label{est_for_DX}
\|D^X\|_{\mathcal{H}} \leq Ce^{\frac{21}{40}\tau},
\end{equation}
which proves \eqref{main_thm_est1} modulo the claim about the coefficients.

Combining the equations \eqref{decay_hatuu} and \eqref{U_+ in terms of a,b,c,d} and Proposition \ref{est_coeffs} we see that
\begin{equation}\label{bc bound}
|\bar{b}^X|+|\bar{c}^X|\leq C,
\end{equation}
which proves \eqref{main_thm_est3}.

We recall that $e^{-\frac{\tau}{2}}u$ corresponds to the original scale. Hence, if instead of $X=(x_1,x_2,x_3,t)$ we consider the new origin
\begin{equation}\label{eq_recenter}
X'=(x_1-\bar b^X \cos\theta,x_2-\bar c^X\sin \theta, x_3,t)
\end{equation}
then the estimate \eqref{est_for_DX} simplifies to
\begin{equation}\label{est_simplified}
\|\hat u^{X'}(x,\tau) - e^{\tau/2} \bar{a}^X z \|_{\mathcal{H}} \leq Ce^{\frac{21}{40}\tau},
\end{equation}
i.e. the estimate \eqref{main_thm_est1} holds with $\bar{a}^{X'}=\bar{a}^X$, $\bar b^{X'}=0$, and $\bar c^{X'}=0$. If $\bar{a}^X=0$, then \eqref{est_simplified} implies $\|\hat u^{X'} \|_{\mathcal{H}}^2 \leq Ce^{\frac{21}{20}\tau}$, contradicting \eqref{12isdom}. Here, we have used Proposition \ref{Plus or Neutral}, as well as Proposition \ref{thm_const_not_dom} to show that, even after re-centering, the $\frac{1}{2}$ mode dominates. Hence, $\bar{a}^X\neq 0$. Since the estimate \eqref{est_simplified} holds for any $X$ and since $\bar{a}^X$ does not vanish for any $X$, we see that $\bar a^X=:\bar{a}$ is independent of $X$.

It remains to prove the pointwise estimate \eqref{main_thm_est2}. To this end, we start with
\begin{align}
\|D^X \|_{L^2(\Sigma\cap \{ |z|\leq 10L_0\})} \leq C\| D^X \|_{\mathcal{H}} \leq Ce^{\frac{21}{40}\tau}.
\end{align}
Next, using the inequality \eqref{bc bound} and Lemma \ref{lemma_u_C2_coarse estimate} we estimate
\begin{multline}
\|\nabla D^X \|_{L^2(\Sigma\cap \{ |z|\leq 10L_0\})}+\|\nabla^2 D^X \|_{L^2(\Sigma\cap \{ |z|\leq 10L_0\})} \\
\leq C\| \hat{u}^X \|_{C^2(\Sigma\cap \{ |z|\leq 10L_0\})}+Ce^{\tau/2}  \leq Ce^{\frac{40}{81}\tau}.
\end{multline}
Applying Agmon's inequality we conclude that
\begin{align}
\sup_{|x_3|\leq 10L_0}|D^X|\leq C e^{\tfrac{81}{160}\tau}.
\end{align}
This finishes the proof of the theorem.
\end{proof}

\bigskip

After a change of coordinates we can assume without loss of generality that our ancient low entropy flow $\mathcal M$ satisfies $\bar{a}=\bar{a}(\mathcal M)=1/\sqrt{2}$. Then, after recentering as in \eqref{eq_recenter}, Theorem \ref{thm Neck asymptotic} tells us that the graph $u^X(\cdot,\tau)$ of the rescaled flow $\bar{M}_\tau^X$ satisfies
\begin{align}\label{main_thm_est2pp}
\sup_{|x_3|\leq 10L_0}\big| e^{-\frac{\tau}{2}}u^X(x,\tau)- \tfrac{1}{\sqrt{2}} z\big| \leq C e^{\frac{\tau}{160}}
\end{align}
for $\tau\leq \mathcal{T}(Z(X))$.

\begin{corollary}\label{fine_neck_cor}
If the rescaled flow $\bar{M}_\tau^X$ satisfies \eqref{main_thm_est2pp}, then we have the estimates
\begin{align}
&\sup_{\{x_3\leq -L_0\}} (x_1^2+x_2^2) \leq 2,\\
&\inf_{\{x_3\geq L_0\}} (x_1^2+x_2^2) \geq 2,
\end{align}
and
\begin{align}
&\inf_{\bar{M}_\tau^X} x_3 \leq -\mu e^{-\tau/2}
\end{align}
for $\tau\leq \mathcal{T}(Z(X))$, where $\mu>0$ is a numeric constant.
In particular, the unrescaled mean curvature flow $\mathcal{M}=\{M_t\}$ satisfies
\begin{equation}
\inf_{p\in M_t} x_3(p) >-\infty,\quad \textrm{ and }\quad \sup_{p\in M_t} x_3(p) =\infty.
\end{equation}
\end{corollary}

\begin{proof}
The first two estimates follow easily by using the KM-shrinkers $\tilde{\Sigma}_b$ as outer barriers (using also Corollary \ref{cor_barrier}) and the ADS-shrinkers $\Sigma_a$ as inner barriers; since $X_0$ is arbitrary, the latter also implies that $\sup_{p\in M_t} x_3(p) =\infty$.

The third estimate follows from the improved barrier argument (where the surfaces $\Sigma_a$ are shifted along the $z$-axis) from \cite[Sec. 3]{BC}.

Finally, for any $X_0=(x_0,t_0)\in \mathcal{M}$, after recentering in space, the estimates tells us that $\bar{M}^{X_0}_\tau \cap\{ x_3\leq -L_0 \}$ is contained inside the cylinder $\Sigma$ of radius $\sqrt{2}$ for $\tau \leq \mathcal{T}$. Hence, by comparison $M_{t_0}\cap \{ x_3 \leq -L_0\}$ is compact (see also Corollary \ref{thm_compact} and its proof). Since $X_0$ is arbitrary, this yields $\inf_{p\in M_t} x_3(p) >-\infty$.
\end{proof}

\bigskip

\subsection{Fine analysis in the neutral mode}\label{sec_fine_neutral}

In this section, we assume that the neutral mode is dominant. The goal is to prove Corollary \ref{thm_compact}, which show that the solution is compact. This will follow from Theorem \ref{thm_rotation}, which gives a precise inwards quadratic expansion, and a barrier argument.\\

Given any center $X_0\in \mathcal M$, there exists some functions $\sigma$ and $\rho$ satisfying \eqref{univ_fns} such that \eqref{small_graph} holds, and we have
\begin{align}\label{Neutral mode}
&U_-+U_+=o(U_0), && \big|\partial_\tau U_0\big| \leq o(U_0)
\end{align}
for $\tau\leq \mathcal{T}$. In this section, $C<\infty$ and $\mathcal{T}>-\infty$ denote constants that can change from line to line and may also depend on $X_0\in\mathcal M$.
To distinguish the initial choice of $\rho$, we set 
\begin{equation}\label{rho_0_eq}
\rho_0(z)=\rho(z).
\end{equation}
We will later use improved scale functions, but $\rho_0$ will never change.

\bigskip

\subsubsection{Graphical radius}
The first goal is to prove a lower bound for the optimal graphical radius. To begin with, we consider the positive function
\begin{equation}\label{def_alpha}
\alpha(\tau)=\left(\int_{|x_3|\leq L_0}u^2(x,s)\tfrac{1}{4\pi} e^{-\frac{|x|^2}{4}}\right)^{1/2}.
\end{equation}

\begin{lemma}\label{alpha0 bound L2}
For $L \in [L_0,\rho(\tau)]$, we have the estimate
\begin{align}
\alpha(\tau)^2 \leq \int_{\Sigma \cap \{|x_3|\leq L\}} u^2(x,\tau)\tfrac{1}{4\pi}e^{-\frac{|x|^2}{4}} \leq C\alpha(\tau)^2.
\end{align}
\end{lemma}

\begin{proof}
By Proposition \ref{Gaussian density analysis} we have
\begin{equation}
\int_{\Sigma \cap \{{ 2^{k-1}L_0\leq |x_3|\leq 2^{k}L_0} \}} u^2e^{-\frac{|x|^2}{4}} \leq \frac{C_0}{4^kL_0^2} \int_{\Sigma \cap \{|x_3|\leq 2^{k{-1}}L_0\}} u^2e^{-\frac{|x|^2}{4}}
\end{equation}
where $C_0$ is a constant. Hence,
\begin{equation}
\int_{\Sigma \cap \{|x_3|\leq 2L_0\}} u^2\tfrac{1}{4\pi}e^{-\frac{|x|^2}{4}} \leq \Big(1+\tfrac{C_0}{4L_0^2}\Big) \alpha(\tau)^2 \leq \exp\Big(\tfrac{C_0}{4L_0^2}\Big) \alpha(\tau)^2,
\end{equation}
and
\begin{align}
\int_{\Sigma \cap \{|x_3|\leq 2^2L_0\}} u^2\tfrac{1}{4\pi}e^{-\frac{|x|^2}{4}} & \leq \Big(1+\tfrac{C_0}{4^2L_0^2}\Big) \exp\Big(\tfrac{C_0}{4L_0^2}\Big) \alpha(\tau)^2\\
&\leq \exp\big(\left(\tfrac14+\tfrac{1}{4^2}\right)\tfrac{C_0}{L_0^2}\big) \alpha(\tau)^2.\nonumber
\end{align}
Thus, if $2^{I-1} L_0 \leq \rho(\tau)\leq 2^{I}L_0$ then we inductively obtain
\begin{equation}
\int_{\Sigma \cap \{|x_3|\leq \rho(\tau)\}} u^2\tfrac{1}{4\pi}e^{-\frac{|x|^2}{4}} \leq \exp\bigg(\sum_{i=1}^I \tfrac{1}{4^i}\tfrac{C_0}{L_0^2}\bigg) \alpha(\tau)^2\leq \exp\Big(\tfrac{2C_0}{L_0^2}\Big) \alpha(\tau)^2.
\end{equation} 
This completes the proof.
\end{proof}

\bigskip
Now, define an increasing continuous function by
\begin{equation}\label{beta1 def}
\bar{\alpha}(\tau)= \sup_{\sigma \leq \tau}\alpha(\sigma).
\end{equation}
By standard interior estimates, we have 
\begin{equation}\label{C0 bound by beta1}
|u|(x,\tau) \leq C\bar{\alpha}(\tau)
\end{equation}
for $|x_3|\leq L_0$ and $\tau\leq\mathcal{T}$.

For technical reasons, it will be best to work with a monotone function $\beta$, which simultaneously has controlled derivatives. 
To this end, we define
\begin{equation}
\beta(\tau)=\sup_{\sigma \leq \tau}\left(\int_{\Sigma} u^2(x,\sigma)\varphi^2\big( \tfrac{x_3}{\rho_0(\sigma)}\big)\tfrac{1}{4\pi}e^{-\frac{|x|^2}{4}}\right)^{1/2},
\end{equation}
where we recall that $\rho_0$ is defined in \eqref{rho_0_eq} to be the original graphical scale function, which is the input to this section. 
Clearly, $\beta$ is a locally Lipschitz, increasing function. By equation \eqref{Neutral mode}, we have $\beta'=o(\beta)$ at almost every time,\footnote{If one prefers thinking about smooth objects, one could further regularize $\beta$ to be $C^1$, maintaining similar properties.} so in particular
\begin{align}\label{beta3 gradient}
0\leq \beta'(\tau)\leq \tfrac{1}{5}\beta(\tau).
\end{align}
Moreover, by Lemma \ref{alpha0 bound L2} we have
\begin{equation}
\bar{\alpha}(\tau) \leq \beta(\tau)= \left(\int_{\Sigma} u^2(x,\sigma)\varphi^2\big( \tfrac{x_3}{\rho_0(\sigma)}\big)\tfrac{1}{4\pi}e^{-\frac{|x|^2}{4}}\right)^{1/2} \leq C \alpha(\sigma) \leq C\bar{\alpha}(\tau),
\end{equation}
where $\sigma$ is chosen such that the second equality holds. To recapitulate, we have obtained
\begin{equation}\label{beta23 ratio}
\bar{\alpha}(\tau)\leq \beta(\tau)\leq C\bar{\alpha}(\tau).
\end{equation}

We now prove a $C^0$-estimate in terms of $\beta(\tau)$.

\begin{proposition}\label{C0 estimate beta3}
There are constants $c>0$ and $C<\infty$ such that
\begin{equation}
|u|(x,\tau) \leq C\beta(\tau)^{\frac{1}{2}}
\end{equation}
holds whenever $|x_3| \leq c\beta(\tau)^{-\frac{1}{4}}$ and $\tau \leq \mathcal{T}$.
\end{proposition}

\begin{proof}
By the estimates \eqref{C0 bound by beta1} and \eqref{beta23 ratio} there is a constant $K<\infty$ such that
\begin{equation}\label{usq_est}
|u|(x,\tau)\leq K\beta(\tau)
\end{equation}
whenever $|x_3|\leq L_0$ and $\tau\leq\mathcal{T}$.

Fixing $\hat\tau\leq \mathcal{T}$, we consider the ADS-shrinker $\Sigma_a$ with parameter 
\begin{equation}\label{ads_beta}
a=c(K\beta(\hat\tau))^{-\frac{1}{2}},
\end{equation}
where $c>0$ is a numerical constant depending on $L_0$. Then, by \cite[Lemma 4.4]{ADS} the profile function $u_a$ satisfies
\begin{equation}\label{ads_godown}
u_{a}(L_0)\leq \sqrt{2}-K\beta(\hat\tau).
\end{equation}
Comparing \eqref{usq_est} and \eqref{ads_godown} we infer that $\bar{M}_{\tau}\cap \{x_3=L_0\}$ lies outside of $\Sigma_a\cap \{x_3=L_0\}$ for $\tau\leq\hat\tau$.

On the other hand, similarly as in \eqref{profile_diff_inn} we have
\begin{equation}
0\leq \frac{d}{dz}\log \bigg(\frac{2-u_a(z)^2}{z^2}\bigg) \leq \frac{16}{z^3}
\end{equation}
Integration from $\sqrt{a}$ to $a$ yields
\begin{equation}
0\leq \log \bigg(\frac{2}{a^2}\bigg)-\log \bigg(\frac{2-u_a(\sqrt{a}\,)^2}{a}\bigg) \leq 1,
\end{equation}
hence,
\begin{equation}\label{uatsqrta}
u_a(\sqrt{a}\,)^2\geq 2-\tfrac{2}{a}.
\end{equation}
Since $\bar{M}_{\hat\tau}\cap \{x_3 \geq L_0\}$ lies outside of $\Sigma_a\cap \{x_3 \geq L_0\}$, comparing the inequalities \eqref{ads_beta} and \eqref{uatsqrta} yields
\begin{equation}\label{u_lower_beta}
u(x,\hat\tau) \geq -C\beta(\hat{\tau})^{\frac{1}{2}}
\end{equation}
for $\{L_0\leq x_3\leq\sqrt{a}\}$ (and similarly for negative $x_3$). Hence, \eqref{usq_est} implies that \eqref{u_lower_beta} holds for $|x_3| \leq \sqrt{a}=C_0^{\frac{1}{2}}(K\beta(\hat{\tau}))^{-\frac{1}{4}}$.

\bigskip
Next, we establish the upper bound. By using \eqref{profile_diff_inn}, we have
\begin{equation}
e^{-1}b^2L_0^2 \leq \tilde{u}_b^2(L_0)-2 \leq 4(\tilde{u}_b(L_0)-\sqrt{2}\,),
\end{equation}
for sufficiently small $b$. Therefore, if we choose $b^{2}=4eL_0^{-2}K \beta(\hat \tau )$, then we have 
\begin{equation}
 K \beta(\hat \tau) =\tfrac{1}{4e} L_0^2 b^2 \leq  \tilde{u}_b(L_0)-\sqrt{2}.
\end{equation}
Hence, \eqref{usq_est} implies that  $\bar M_\tau \cap \{x_3 \geq L_0\}$ lies inside of $\tilde{\Sigma}_b \cap \{x_3 \geq L_0\}$ for $\tau  \leq \hat \tau$. Then, \eqref{profile_diff_inn} implies
\begin{equation}
\tilde{u}_b^2(1/\sqrt{b}\,)\leq 2+ b.
\end{equation}
Thus, we can complete the proof by arguing similarly with the KM-shrinkers $\tilde{\Sigma}_b$ as outer barriers.
\end{proof}

\bigskip

Similarly as in Proposition \ref{higherderivatives} the $C^0$-estimate from Proposition \ref{C0 estimate beta3} can be upgraded to a $C^4$-estimate. Hence, we can now repeat the process from Section \ref{sec_fine_neck_setup} with better functions $\rho$ and $\sigma$. Namely, by we can now choose
\begin{equation}\label{improved_rho}
\rho(\tau)=\beta(\tau)^{-\frac{1}{5}}, \textrm{ and } \sigma(\tau)=\beta(\tau)^{\frac{1}{5}}.
\end{equation}
and write $\bar{M}_\tau$ as a graph of a function $u$ over $\Sigma\cap B_{2\rho(\tau)}$ such that
\begin{equation}\label{C4 estimate}
\|u(\cdot,\tau)\|_{C^4(\Sigma\cap B_{2\rho(\tau)}(0))} \leq \rho(\tau)^{-2}
\end{equation}
for $\tau\leq \mathcal T$. Note that by equation \eqref{beta3 gradient} the derivative $\rho'$ indeed satisfies
\begin{equation}\label{contr_der}
-\rho(\tau) \leq \rho'(\tau) \leq 0,
\end{equation}
as required by condition \eqref{univ_fns}. From now on we work with the function
\begin{equation}
\hat{u}(x,\tau)=u(x,\tau)\varphi \left( \frac{x_3}{\rho(\tau)}\right),
\end{equation}
where $\rho$ is the improved graphical radius from \eqref{improved_rho}.\\

The following gives a lower bound for the improved graphical radius:

\begin{proposition}\label{prop_improved_graphical_radius}
There exists  constants $\gamma>0$ and $c>0$ such that
\begin{equation}\label{alpha_rho_pol}
\rho(\tau)\geq c|\tau|^\gamma.
\end{equation}
holds for $\tau\leq \mathcal{T}$.
\end{proposition}

\begin{proof}

By Colding-Minicozzi \cite[Theorem 6.1]{CM_uniqueness}, there exist $\eta\in (1/3,1)$ and $K<\infty$ such that
\begin{equation}
\left(F(\Sigma)-F(\bar{M}_{\tau})\right)^{1+\eta} \leq K\left(F(\bar{M}_{\tau-1})-F(\bar{M}_{\tau+1})\right)
\end{equation}
for $\tau\leq \mathcal{T}$. Using the discrete Lojasiewicz lemma \cite[Lemma 6.9]{CM_uniqueness} and also \cite[(6.19),(6.20)]{CM_uniqueness} this yields
\begin{equation}
\left(F(\Sigma)-F(\bar{M}_{\tau})\right) \leq C|\tau|^{-1/\eta},
\end{equation}
and 
\begin{equation}\label{tail_decay_sum}
\sum_{j=J}^{\infty} \left(F(\bar{M}_{-j-1})-F(\bar{M}_{-j})\right)^{1/2} \leq C(\nu)J^{-\nu}
\end{equation}
for every $\nu \in (0,\frac{1}{2\eta}-\frac{1}{2})$. Thus, as in the proof of \cite[Theorem 0.2]{CM_uniqueness}, we obtain
\begin{equation}
\int_{-\infty}^{\tau}\int_{\bar{M}_{\tau'}} \left|\mathbf{H}+\frac{x^{\perp}}{2}\right|e^{-\frac{|x|^2}{4}} d\mu_{\tau'} d\tau' \leq C|\tau|^{-\nu}
\end{equation}
for $\tau\leq\mathcal{T}$. Using also \cite[Lemma A.48]{CM_uniqueness} this yields
\begin{equation}
\int_{\Sigma\cap \{|x_3|\leq \rho(\tau)/2\}}|\hat{u}(x,\tau)|e^{-|x|^2/4} \leq C|\tau|^{-\nu}.
\end{equation}
Together with \eqref{C4 estimate} this implies
\begin{equation}
\alpha(\tau)\leq C|\tau|^{-\nu/2}.
\end{equation}
Taking also into account \eqref{beta23 ratio} we infer that
\begin{equation}\label{eqn_decay_beta}
\beta(\tau)\leq C|\tau|^{-\nu/2}.
\end{equation}
Recalling \eqref{improved_rho}, we conclude that \eqref{alpha_rho_pol} holds with $\gamma=\nu/10$. 
\end{proof}

\bigskip

\subsubsection{Estimates for some error terms}

In the following lemma we Taylor expand the rescaled mean curvature flow to second order:

\begin{lemma}\label{lemma_taylor}
The function $\hat{u}(x,\tau)=u(x,\tau)\varphi \big( \tfrac{x_3}{\rho(\tau)}\big)$ satisfies
\begin{equation}
\partial_\tau \hat u = \mathcal{L}\hat u-\tfrac{1}{2\sqrt{2}}\hat u^2-\tfrac{1}{2\sqrt{2}} \hat u_\theta^2-\tfrac{1}{ \sqrt{2}}\hat u \hat u_{\theta\theta} +E,
\end{equation}
where the error term can be estimated by
\begin{align}\label{est_eq}
|E|\leq &C\varphi(|u|+|\nabla u|)^2(|u|+|\nabla u|+|\nabla^2u|)\nonumber\\
&+ C |\varphi'|\rho^{-1} \big(|\nabla u|+|zu|\big)\nonumber\\
&+ C |\varphi''|\rho^{-2}|u|\nonumber\\
&+ C \varphi(1-\varphi)\big(|u|^2+|\nabla u|^2+|\nabla^2u|^2\big).
\end{align}
\end{lemma}

\begin{proof}
The rescaled mean curvature flow for graphs over the cylinder is given by
\begin{align}
\partial_\tau u=&\frac{\Big(1+\frac{  u_\theta ^2}{(\sqrt{2}+u)^2}\Big)u_{zz}-\frac{ u_\theta^2u_{\theta\theta}}{(\sqrt{2}+u)^4}-\frac{2 u_\theta u_zu_{z\theta} }{(\sqrt{2}+u)^2}-\frac{ u_\theta ^2}{(\sqrt{2}+u)^3}}{1+ u_z ^2+  u_\theta ^2 (\sqrt{2}+u)^{-2}}\nonumber\\
&+\,\frac{u_{\theta\theta}}{(\sqrt{2}+u)^2}-\frac{1}{\sqrt{2}+u}+\frac{1}{2}\big(\sqrt{2 }+u-zu_z\big),
\end{align}
c.f. \cite[App. A]{GKS}.  Hence,
\begin{align}
\partial_\tau u =  u_{zz}-\tfrac{ u_\theta ^2}{(\sqrt{2}+u)^3}+\tfrac{u_{\theta\theta}}{(\sqrt{2}+u)^2}-\tfrac{1}{\sqrt{2}+u}+\tfrac{1}{2}\big(\sqrt{2 }+u-zu_z\big)+E_1,
\end{align}
where
\begin{equation}
|E_1|\leq C|\nabla u|^2(|\nabla u|+|\nabla^2 u|).
\end{equation}
This implies
\begin{align}
   \partial_\tau u=\mathcal{L}  u+\mathcal{Q}(u)+E_2,
\end{align}
where
\begin{equation}
\mathcal{Q}(u)=-\tfrac{1}{2\sqrt{2}}  u^2-\tfrac{1}{2\sqrt{2}}   u_\theta^2-\tfrac{1}{ \sqrt{2}} u   u_{\theta\theta},
\end{equation}
and
\begin{equation}
|E_2|\leq C|\nabla u|^2(|u|+|\nabla u|+|\nabla^2 u|)+C|u|^2|\nabla^2u|+C|u|^3.
\end{equation}

We now want to see what this computation translates to in terms of $\hat{u}$.
For the linear term, using \eqref{contr_der} we get
\begin{equation}\label{lin_est}
\begin{aligned}
\big|\partial_\tau\hat{u}&-\mathcal{L} \hat{u}-\varphi (\partial_\tau u-\mathcal{L}u )\big| \leq 
\rho^{-2}|\varphi''||u|+2\rho^{-1}|\varphi'|(|\nabla u|+|zu|).
\end{aligned}
\end{equation}

For the quadratic term, using that $\rho$ is $\theta$-independent we get
\begin{equation}\label{quad_est}
|\mathcal{Q}(\hat{u})-\varphi\mathcal{Q}(u)|=|\varphi^2\mathcal{Q}(u)-\varphi\mathcal{Q}(u)|= \varphi (1-\varphi)|\mathcal{Q}(u)|.
\end{equation}

Putting everything together, this implies the assertion.
\end{proof}

\bigskip

For our function
\begin{equation}
\hat{u}(x,\tau)=u(x,\tau)\varphi \big( \tfrac{x_3}{\rho(\tau)}\big),
\end{equation}
where $\rho$ is the improved graphical scale from \eqref{improved_rho}, we let
\begin{align}
& U_+=\|P_+\hat u\|_{\mathcal{H}}^2, && U_0=\|P_0\hat u\|_{\mathcal{H}}^2, && U_-=\|P_-\hat u\|_{\mathcal{H}}^2.
\end{align}
By assumption \eqref{Neutral mode} and Proposition \ref{Plus or Neutral} we still have
\begin{equation}\label{neutral mode in check u}
 U_+ + U_- =o(U_0).
\end{equation}
Therefore, we can now expand
\begin{equation}\label{hatu_expansion}
\hat u = \alpha_1\psi_1+ \alpha_2\psi_2+ \alpha_3\psi_3 + o(|\vec\alpha|),
\end{equation}
where
\begin{equation}
\vec{\alpha}(\tau)=(\alpha_1(\tau),\alpha_2(\tau),\alpha_3(\tau))
\end{equation}
are time dependent coefficients, and where $\psi_1,\psi_2$ and $\psi_3$ are the three zero eigenfunctions satisfying $\mathcal L \psi_i =0$,
explicitly
\begin{align}
\psi_1 &= 2^{-\frac{3}{2}}(\tfrac{e}{2\pi})^{\frac{1}{4}}(z^2-2),\\
\psi_2 &= (\tfrac{e}{2\pi})^{\frac{1}{4}} z\cos\theta,\\
\psi_3 &= (\tfrac{e}{2\pi})^{\frac{1}{4}}  z \sin\theta.
\end{align}
Moreover, Lemma \ref{alpha0 bound L2} and equation \eqref{neutral mode in check u} yield that 
\begin{equation}\label{alpha0 and alpha}
C^{-1}\alpha(\tau) \leq |\vec{\alpha}|(\tau) \leq C\alpha(\tau)
\end{equation} 
for $\tau\leq\mathcal{T}$, where $\alpha(\tau)$ is the function defined in equation \eqref{def_alpha}.

\begin{proposition}\label{prop_errortaylorest}
The error term $E$ from Lemma \ref{lemma_taylor} satisfies the estimate
\begin{equation}
|\langle E,\psi_i\rangle |\leq C \beta(\tau)^{2+\frac{1}{5}}
\end{equation}
for $\tau\leq\mathcal{T}$.
\end{proposition}

\begin{proof}
Using equation \eqref{C4 estimate} and Lemma \ref{lemma_taylor} we see that
\begin{equation}
|E| \leq C\rho(\tau)^{-2}. 
\end{equation}
Together with $|\psi_i| \leq C\rho(\tau)^2$ on the support of $E$ this yields the coarse estimate
\begin{equation}
|E| |\psi_i| \leq C.
\end{equation}
Using this, we compute
\begin{align}
\bigg|\int_{\Sigma\cap \big\{ |x_3|\geq \rho(\tau)^{\frac{1}{10}}\}}  E \psi_i e^{-\frac{|x|^2}{4}} \bigg| &\leq C\int_{\Sigma\cap {\big\{ \rho(\tau)^{\frac{1}{10}}\leq |x_3| \leq \rho(\tau) \big\}}}\exp\Big(-\tfrac{1}{4}\rho(\tau)^{\frac{1}{5}}\Big)\nonumber\\
&\leq C\beta(\tau)^3
\end{align}
for $\tau\leq\mathcal{T}$.

On the other hand, if $|x_3|  \leq \tfrac12 \rho(\tau)$ we have $\varphi'=\varphi''=1-\varphi=0$. Therefore, for $|x_3| \leq \rho(\tau)^{\frac{1}{10}}$ we can estimate
\begin{align}
|E||\psi_i| &\leq C(z^2+2)  (|u|+|\nabla u|)^2(|u|+|\nabla u|+|\nabla^2u|) \\
& \leq C\rho^{\frac{3}{10}}\rho^{-2}(|u|^2+|\nabla u|^2)\\
& \leq C\rho^{-1}(|u|^2+|\nabla u|^2).
\end{align}
Thus, using also Proposition \ref{Gaussian density analysis} and Lemma \ref{alpha0 bound L2}, we infer that
\begin{align}
\bigg|\int_{\Sigma\cap \big\{|x_3| \leq \rho(\tau)^{\frac{1}{10}}\big\}}  E \psi_i e^{-\frac{|x|^2}{4}} \bigg| &\leq \frac{C}{\rho(\tau)}\int_{\Sigma\cap \{|x_3| \leq \rho(\tau)\}} (u^2+|\nabla u|^2) e^{-\frac{|x|^2}{4}}\\
&\leq C\rho(\tau)^{-1}\alpha(\tau)^2\\
&\leq C \beta(\tau)^{2+\tfrac{1}{5}}.
\end{align}
This proves the proposition.
\end{proof}

\bigskip

\subsubsection{Axis tilt decay}  Our next objective is to show that rotations, captured by the coefficients $\alpha_2,\alpha_3$ in the expansion \eqref{hatu_expansion}, are decaying rapidly. \\

Given a point $q\in \mathbb{R}^3$ and a direction $w \in S^2$, we denote by $K$ the \emph{normalized rotation vector field} that corresponds to a rotation by $\pi/2$ around the axis $W=\{q+wt:t \in \mathbb{R}\}$, namely
\begin{equation}
K(x) = SJS^{-1}(x-q),\qquad \textrm{ where } J = \begin{bmatrix} 0 & -1 & 0 \\ 1 & 0 & 0 \\ 0 & 0 & 0 \end{bmatrix},
\end{equation}
and $S\in \textrm{SO}_3$ is any rotation matrix with $Se_3=w$.\footnote{$S$ is only determined up to a one-parameter choice, but $K$ is well-defined.}

\begin{definition}[{c.f. \cite[Def. 4.3]{BC}}]\label{def_eps_sym_restated}
A point $X\in \mathcal{M}$ with $H(X)>0$ is called \emph{$\delta$-symmetric} if there exists a normalized rotation vector field $K$ such that
\begin{equation}
|\langle K,\nu \rangle \, H | \leq \delta \textrm { in the parabolic ball } P(X,10H^{-1}(X)).
\end{equation}
\end{definition}

\begin{lemma}
There exists some $\eta_0>0$ such that all points $X\in\mathcal{M}$ that correspond to points in $\bar{M}_{\tau}\cap \{|x_3|\leq \rho(\tau)\}$ are $2^{-\eta_0 \rho(\tau)}$-symmetric for $\tau\leq\mathcal{T}$.  
\end{lemma}

\begin{proof}
Fix some time $\tau_0 \leq \mathcal{T}$ and set $R:=2\rho(\tau_0)$. Let $\eps_0>0$ and $L<\infty$ be the constant from the neck-improvement theorem \cite[Thm. 4.4]{BC}.

Let $\bar{X}\in \bar{M}_{\tau}$ be a point with $|x_3|\leq R-L$ and $\tau \leq \tau_0$, and denote by $X\in \mathcal{M}$ the corresponding point in the unrescaled flow. Using \eqref{C4 estimate} we see that every $X'\in P(X,LH^{-1}(X))\cap \mathcal{M}$ lies on an $\eps_0$-neck and is $\eps_0$-symmetric. Hence, by the neck improvement theorem $X$ is $2^{-1}\eps_0$-symmetric.\\
Similarly, for any $\bar{X}\in \bar{M}_{\tau}$ with $|x_3| \leq R-2L$ and $\tau\leq \tau_0$, the arguments above shows that the corresponding point $X\in \mathcal{M}$ is $2^{-2}\eps_0$-symmetric. Iterating this $k$ times,  until 
\begin{equation}\label{how_many_iterate}
2R-(k+1)L< R,
\end{equation} 
we get that for all $\bar{X}\in \bar{M}_{\tau}$ with $|x_{3}| \leq R/2$ and $\tau\leq \tau_0$, the corresponding point $X\in\mathcal{M}$ is $2^{-k}\eps_0$-symmetric.  Since from \eqref{how_many_iterate} we get $k >R/L-1$, choosing $\eta_0=\frac{1}{L}$ gives the desires result.  
\end{proof}

\begin{proposition}[axis tilt decay]\label{thm_global.rot}
There exists a constant $\eta>0$ such that for all $\tau\leq\mathcal{T}$ we have
\begin{equation}
|u_\theta(x,\tau)| \leq e^{-\eta |\tau|^\gamma}
\end{equation}
for all $x\in \Sigma\cap\{ |x_3|\leq c|\tau|^\gamma\}$.
\end{proposition}

\begin{proof}
For each $\tau\leq\mathcal{T}$, consider a point $\bar{X}_{\tau}\in \bar{M}_{\tau}\cap\{x_3=0\}$, and let $\omega_{\tau}\in S^2$ be a direction (close to $e_3$) in which the corresponding point $X_\tau\in\mathcal{M}$ in the unrescaled flow is $2^{-\eta \rho(\tau)}$-symmetric. Since each $X_{\tau-1}$ is $C2^{-\eta \rho(\tau)}$-symmetric with respect to both $\omega_{\tau}$ and $\omega_{\tau-1}$, we get 
\begin{equation}
|\omega_{\tau}-\omega_{\tau-1}|\leq  C2^{-\eta_0 \rho(\tau)}.
\end{equation}
Hence, using also $\omega_{\tau}\to e_3$ and Proposition \ref{prop_improved_graphical_radius}, we obtain
\begin{equation}
|\omega_{\tau}-e_3| \leq  C\sum_{m=0}^{\infty}2^{-\eta_0 \rho(\tau-m)}\leq C\sum_{m=0}^{\infty}2^{-c\eta_0 |\tau-m|^\gamma}.
\end{equation}
Note that this can be safely estimated by $Ce^{-\eta |\tau|^{\gamma}}$, where $\eta:=\eta_0/10$. 

As any $x\in \bar{M}_{\tau}$ with $|x_3|\leq c|\tau|^\gamma$ is also $C2^{-\eta \rho(\tau)}$-symmetric with direction $\omega(x,\tau)$, we get that 
\begin{equation}
|\omega(x,\tau)-\omega_\tau|\leq C|\tau|^\gamma 2^{-\eta\rho(\tau) },  
\end{equation}
which, again, can be safely estimated by  $Ce^{-\eta |\tau|^{\gamma}}$. Thus, all such points are also $Ce^{-\eta |\tau|^{\gamma}}$-symmetric with respect to the $x_3$-axis.
Considering the normalized rotation vector field $K:=x_1 {\partial}_{x_2}- x_2\partial_{x_1}$ around the $x_3$-axis, we conclude that
\begin{equation}
\tfrac{1}{2}|u_{\theta}| \leq H |\langle K, \nu \rangle |\leq Ce^{-\eta |\tau|^{\gamma}}.
\end{equation}
Slightly decreasing $\eta$, this proves the proposition.
\end{proof}

\bigskip

\begin{corollary}\label{prop_rot.decay}
The coefficients $\alpha_2,\alpha_3$ from the expansion \eqref{hatu_expansion} satisfy
\begin{align}
|\alpha_2|+|\alpha_3|\leq Ce^{-\eta|\tau|^\gamma}\label{rot.decay}
\end{align}
for $\tau\leq \mathcal{T}$.
\end{corollary}

\begin{proof}
Using $\psi_2=\partial_\theta \psi_3$  and Proposition \ref{thm_global.rot} we obtain
\begin{equation}
|\alpha_2|=\big|\langle \hat u,\partial_\theta \psi_3 \rangle \big|= \big|\langle \partial_\theta \hat u, \psi_3 \rangle \big|\leq ||\partial_\theta \hat{u}||_{\mathcal{H}}\leq Ce^{-\eta |\tau|^\gamma}.
\end{equation}
Similarly, we can estimate $|\alpha_3|$ to complete the proof.
\end{proof}

\bigskip

\subsubsection{The inwards quadratic neck theorem} The next proposition gives an ODE for the coefficient $\alpha_1$ from the expansion \eqref{hatu_expansion}.

\begin{proposition}\label{prop_ODEs}
The coefficient $\alpha_1$ from the expansion \eqref{hatu_expansion} satisfies
\begin{align}
\tfrac{d}{d\tau} \alpha_1 &= -(\tfrac{e}{2\pi})^{\frac{1}{4}}\alpha_1^2  + o(\beta^2)+O(e^{-\eta |\tau|^\gamma}).
\label{ODE1}
\end{align}
\end{proposition}

\begin{proof}
Using Lemma \ref{lemma_taylor} and the identities $\mathcal{L}\psi_1=0$ and $\partial_\theta \psi_1=0$ we compute
\begin{align}
\tfrac{d}{d\tau} \alpha_1 &= \langle \partial_\tau \hat u , \psi_1 \rangle\\
&=\langle \mathcal L \hat u  -\tfrac{1}{2\sqrt 2} \hat u ^2+\tfrac{1}{2\sqrt{2}}\hat{u}_{\theta}^2-\tfrac{1}{\sqrt{2}}(\hat{u}\hat{u}_{\theta})_\theta+ E , \psi_1 \rangle\\
&=\langle -\tfrac{1}{2\sqrt{2}} \hat u ^2+\tfrac{1}{2\sqrt{2}}\hat{u}_{\theta}^2 + E , \psi_1 \rangle\\
&=-\tfrac{1}{2\sqrt{2}}\alpha_1^2 \langle \psi_1^2,\psi_1\rangle + o(\beta^2)+O(e^{-\eta |\tau|^\gamma}),
\end{align}
where in the last step we used Proposition \ref{prop_errortaylorest}, Proposition \ref{thm_global.rot} and Corollary \ref{prop_rot.decay} to estimate the error terms. Together with
\begin{equation}
\int \psi_1^3 \tfrac{1}{4\pi} e^{-|x|^2/4}=2^{\frac{3}{2}}(\tfrac{e}{2\pi})^{\frac{1}{4}},
\end{equation}
this implies the assertion.
\end{proof}

Recall that by \eqref{eqn_decay_beta} the quantity $\beta(\tau)$ decays at some definite rate as $\tau\to -\infty$. The following lemma shows that $\beta(\tau)$ does not decay rapidly.

\begin{lemma}\label{prop_L1fail.neutral}
We have
\begin{equation}
\limsup_{\tau\to -\infty} |\tau|^{10}\beta(\tau)=\infty.
\end{equation}
\end{lemma}

\begin{proof}
Suppose towards a contradiction, there is some $C<\infty$ such that $\beta(\tau)\leq C|\tau|^{-10}$ for all $\tau\leq\mathcal{T}$. Then, $\rho(\tau)\geq c |\tau|^2$, and equation \eqref{U_PNM_system}, together with the assumption that the neutral mode is dominant, implies
\begin{equation}
|\tfrac{d}{d\tau}\log U_0|\leq C|\tau|^{-2}.
\end{equation}
Hence, $|\log U_0|\leq C$ for all $\tau \leq \mathcal{T}$, which contradicts $\displaystyle\lim_{\tau\to -\infty}U_0(\tau)= 0$.
\end{proof}

\begin{theorem}\label{thm_rotation}
The coefficients $\alpha_1,\alpha_2,\alpha_3$ from the expansion \eqref{hatu_expansion} satisfy
\begin{align}\label{sol_alph1}
\alpha_1(\tau)=\frac{-2^{\frac{1}{4}}\pi^{\frac{1}{4}}e^{-\frac{1}{4}}+o(1)}{| \tau|},
\end{align}
and
\begin{align}\label{sol_alph2}
|\alpha_2|+|\alpha_3|=o(|\alpha_1|)
\end{align}
for $\tau\leq \mathcal{T}$.
\end{theorem}

\begin{proof} To keep track of the errors, we consider the monotone quantity
\begin{equation}\label{def_bar_alpha}
\bar{\alpha}_1(\tau):=\sup_{\sigma\leq \tau}|\alpha_1(\sigma)|.
\end{equation}
By Proposition \ref{prop_ODEs}, inequality \eqref{alpha0 and alpha}, and Corollary \ref{prop_rot.decay} for every $\eps>0$ we get
\begin{align}\label{ODE.sharp}
\left|\tfrac{d}{d\tau} \alpha_1(\tau) +(\tfrac{e}{2\pi})^{\frac{1}{4}}\alpha_1^2(\tau)\right|\leq    \eps \bar{\alpha}_1^2(\tau)+|\tau|^{-1000}
\end{align}
for $\tau \leq \mathcal{T}=\mathcal{T}(\eps)$. Moreover, taking also into account Lemma \ref{prop_L1fail.neutral} we have
\begin{equation}\label{limsupbaralpha1}
\limsup_{\tau\to -\infty} |\tau|^{10}\bar{\alpha}_1(\tau)=\infty.
\end{equation}
Consider the sets
\begin{align}
I:&=\left\{\tau \leq \mathcal{T} \, :\,  \bar{\alpha}_1(\tau)=|\alpha_1(\tau)|\right\}, \\
J:&=\left\{\tau \leq \mathcal{T} \, :\, \bar{\alpha}_1(\tau)\geq |\tau|^{-10}\right\}.
\end{align}
Thanks to $\bar{\alpha}_1(\tau)\to 0$ and \eqref{limsupbaralpha1} the set $I\cap J$ contains a sequence of numbers going to $-\infty$. Clearly, $I\cap J\subseteq (-\infty,\mathcal{T}]$ is closed.

Let $\tau_0\in I\cap J$. Then, by \eqref{ODE.sharp} we have
\begin{equation}\label{OED_ineq}
\frac{d}{d\tau}\alpha_1(\tau)\leq -\frac{1}{2}\left(\frac{e}{2\pi}\right)^{\frac{1}{4}}\alpha_1^2(\tau)\leq  -\frac{1}{4}\left(\frac{e}{2\pi}\right)^{\frac{1}{4}}\frac{1}{|\tau|^{100}}
\end{equation}
at $\tau=\tau_0$. In particular, remembering \eqref{def_bar_alpha}, we see that $\alpha_1(\tau_0)<0$. In fact, we can find a $\delta>0$ such that \eqref{OED_ineq} and $\alpha_1(\tau)<0$ hold for $|\tau-\tau_0|<\delta$.
Moreover, if there is some $\hat{\tau} \in I$ with $\hat\tau \in (-\infty,\tau_0]$ and $\alpha'(\hat{\tau})=0$ then \eqref{ODE.sharp} yields
\begin{equation}
|\alpha_1^2(\hat{\tau})|\leq 2|\hat{\tau}|^{-1000}\leq 2|\hat{\tau}_0|^{-1000}< \tfrac{1}{100} |\alpha_1^2(\tau_0)|.
\end{equation}
Hence, if  $ \frac{1}{10}|\alpha_1(\tau_0)|\leq |\alpha_1(\tau)|\leq |\alpha_1(\tau_0)|$ then $\tau\in I$.
Thus, possibly after decreasing $\delta$, we get $(\tau_0-\delta,\tau_0]\subseteq I$. Then, \eqref{ODE.sharp}  implies $(\tau_0-\delta,\tau_0]\subseteq J$.


Summing up, we can find $\alpha_1$ by solving the ODE
\begin{align}\label{ODE.sharp_rest}
\left|\tfrac{d}{d\tau} \alpha_1(\tau) +(\tfrac{e}{2\pi})^{\frac{1}{4}}\alpha_1^2(\tau)\right|\leq    \eps \alpha_1^2(\tau)+|\tau|^{-1000}
\end{align}
for $\tau\leq\mathcal{T}(\eps)$. This yields \eqref{sol_alph1}. Since $\alpha_2,\alpha_3$ are rapidly decaying we also get \eqref{sol_alph2}. This concludes the proof of the theorem.
\end{proof}

\bigskip

\begin{corollary}
For $\tau\leq\mathcal{T}$ the rescaled flow $\bar{M}_\tau^{X_0} \cap \{|z|\geq L_0\}$ is contained inside the cylinder $\Sigma$ of radius $\sqrt{2}$.
\end{corollary}

\begin{proof}
Theorem \ref{thm_rotation} implies that $\bar M_\tau^{X_0}$ satisfies
\begin{equation}
u(L_0,\tau) < 0
\end{equation}
for $\tau \leq \mathcal{T}$. Hence, given arbitrarily small $b>0$, the KM-barrier $\tilde{\Sigma}_b$ lies outside of $\bar M_\tau^{X_0}\cap\{x_3=L_0\}$. Using also Corollary \ref{cor_barrier} and the avoidance principle we infer that $\tilde{\Sigma}_b$ lies outside of $\bar M_\tau^{X_0}$ on $\{x_3 \geq L_0\}$. Namely, $\bar M_\tau^{X_0}$ satisfies $x_1^2+x_2^2\leq 2$ for $x_3 \geq L_0$. In the same manner, $\bar M_\tau^{X_0}$ satisfies $x_1^2+x_2^2\leq 2$ for $x_3 \leq L_0$. This proves the corollary.
\end{proof}

\begin{corollary}\label{thm_compact}
For an ancient low entropy solution, if the neutral mode is dominant, then the solution is compact.
\end{corollary}

\begin{proof}
Denote the rescaled flow by $\bar{M}_\tau=\bar{M}^{X_0}_\tau$, and the original flow by $\mathcal M= \{ M_t\}$. Then, by the above corollary, there exists some $t_1\in (-\infty,t_0)$ and $L<\infty$ such that $M_{t_1}\cap \{|z|\geq L\}$ is contained inside the cylinder of radius $\sqrt{2(t_0-t_1)}$. The lemma below implies that $M_t$ is compact for all $t>t_0$. Since $X_0$ was arbitrary, this proves the corollary.
\end{proof}

\begin{lemma}
Let $M_0\subset \mathbb{R}^3$ be a complete embedded surface such that
\begin{equation}
M_0\setminus B(0,L)=\Sigma\setminus B(0,L)
\end{equation}
for some $L<\infty$, where $\Sigma$ denotes the cylinder of radius $\sqrt{2}$. Then the mean curvature flow $\{M_t\}$ of $M_0$ is compact for all $t>1$.
\end{lemma}

\begin{proof}
Let $\mathcal M$ be the mean curvature flow with initial condition $M_0$. For $z_i\to \pm\infty$ consider the shifted flows $\mathcal M^i = \mathcal M - (z_i,0)$. For $i\to \infty$ the flows $\mathcal M^i$ converge to a limit $\mathcal M^\infty$, whose initial time slice is $\Sigma$. By uniqueness, $\mathcal M^\infty$ must be a round shrinking cylinder that becomes extinct at time $1$. We conclude that $M_t$ is compact (this includes the possibility that $M_0$ is noncompact, but $M_t$ becomes compact via a contracting cusp singularity at spatial infinity) for all $t> 1$.
\end{proof}

\bigskip

\section{Curvature bound and cap size control}\label{sec_cap_size}
Throughout this section, $\mathcal M$ will always be an ancient low entropy flow where the plus mode is dominant.
We will frequently use the fine neck theorem (Theorem \ref{thm Neck asymptotic}) and its corollary (Corollary \ref{fine_neck_cor}). By an affine change of coordinates, we can assume without loss of generality that
\begin{equation}
\bar{a}=1/\sqrt{2},
\end{equation}
that the axis of the asymptotic cylinder is the $x_3$ axis, and that
\begin{equation}
\min_{p\in M_0}x_3(p) =0 \textrm{ is attained at the origin.}
\end{equation}

\subsection{Global curvature bound}\label{sec_curv_bound}

We will prove a global in space-time curvature bound. To this end, we start with the following proposition.

\begin{proposition}\label{cylind_growth}
For every $\Lambda<\infty$ there exists a $\rho=\rho(\mathcal{M},\Lambda)<\infty$ with the following significance.
If $p_0,p\in M_{t_0}$ are points such that $Z(p_0,t_0)\leq \Lambda$ and 
\begin{equation} 
 d:=x_3(p)-x_3(p_0)\geq \rho,
\end{equation}
then
\begin{equation}
Z(p,t_0)> \Lambda.
\end{equation}    
\end{proposition}

\begin{proof}
Let $C=C(\mathcal{M})<\infty$ and $\mathcal{T}(\Lambda)>-\infty$ be the constants from the fine neck theorem (Theorem \ref{thm Neck asymptotic}) applied at any point with cylindrical scale at most $\Lambda$. We will show that if $p_0,p\in M_{t_0}$ are points with $Z(p_0,t_0)\leq \Lambda$ and $Z(p,t_0)\leq \Lambda$, then $d:=x_3(p)-x_3(p_0)$ is bounded from above.

Let $t_1=t_0-d^2$, and let $\tau$ be the corresponding rescaled time, namely
\begin{equation}\label{dtau}
\tau=-\log(t_0-t_1)=-2\log d.
\end{equation}
Observe that if $d$ is sufficiently large, at time $\tau$ we will see necks around $p_0$ and $p$. Those two necks have to align with each other, else looking further back in time they would intersect. Moreover, if $d$ is sufficiently large, then $\tau \leq \mathcal{T}(\Lambda)$, as required for applying the fine neck theorem. 

Letting $r_0$ and $r$ be (representitives of) the radii of the fine necks around $p_0$ and $p$ at time $t_1$, applying the fine neck theorem (Theorem \ref{thm Neck asymptotic})  at time $\tau$ with $z=0$ twice, once centered  at $(p_0,t_0)$ and once centered at $(p,t_0)$, we obtain  
\begin{equation}\label{r_dif_bd_1}
|r-r_0|\leq 2Ce^{\frac{\tau}{160}}.
\end{equation}
On the other hand, applying the fine neck theorem (Theorem \ref{thm Neck asymptotic}) centered at $(p_0,t_0)$ at time $\tau$ with $z=1$, we obtain 
\begin{equation}\label{r_dif_bd_2}
|r-r_0|\geq \frac{1}{\sqrt{2}}- Ce^{\frac{\tau}{160}}.
\end{equation}
Comparing \eqref{r_dif_bd_1} and \eqref{r_dif_bd_2} we conclude that $-\tau$ is bounded; hence $d$ is bounded above. This finishes the proof of the proposition.
\end{proof} 

\begin{corollary}\label{R_inf_bound}
For every compact interval $I\subseteq (-\infty,T_E(\mathcal{M}))$ there exist constants $\delta>0$, $\Lambda <\infty$ and $C<\infty$ (possibly depending on $\mathcal M$ and $I$) such that for every $p\in M_t$ with $t\in I$ and $x_3(p)\geq \Lambda$ either
\begin{equation}
R(p,t)\geq \delta,
\end{equation}
or
\begin{equation}
\textrm{$(p,t)$ is $\varepsilon$-spherical at scale $CR(p,t)$.}
\end{equation}
\end{corollary}

\begin{proof}
Suppose that there exists a sequence of points $p_k\in M_{t_k}$ with $t_k\in I$ such that $x_3(p_k)\rightarrow \infty$ and $R(p_k,t_k)\rightarrow 0$. 
Recall that by Section \ref{sec_necks_back_in_time}, there are constants $0\ll C \ll C'<\infty$ such that if $(p_k,t_k)$ is not $\varepsilon$-spherical at scale $CR(p_k,t_k)$, then $Z(p_k,t_k)\leq C' R(p_k,t_k)$. For $k$ large, the latter option is impossible by Proposition \ref{cylind_growth}, and so the assertion follows.
\end{proof} 

\begin{lemma}\label{no_comp_conn}
The surface $M_{t_0}$ has no compact connected components for any $t_0\in (-\infty,T_E(\mathcal{M}))$.
\end{lemma}

\begin{proof}
Suppose towards a contradiction that there is a compact connected component $N_{t_0}$ at some time $t_0$.
Let $\mathcal{N}$ be the space-time connected component of $N_{t_0}$, and denote its time $t$-slice by $N_t$. By Corollary \ref{fine_neck_cor} there is some $t_1<t_0$ such that $N_{t_1}$ is non-compact.

Let $\delta,\Lambda$, and $C$ be the constants from Corollary \ref{R_inf_bound}, corresponding to the interval $I=[t_1,t_0]$. 
Using Corollary \ref{fine_neck_cor} and spherical barriers we see that there is a constant $\Lambda'\in (\Lambda,\infty)$ such that
\begin{equation}\label{coarse_comp}
N_t\subset B(0,\Lambda')\cup \{x_3 \geq \Lambda'\}.
\end{equation}
for all $t\in I$. After possibly increasing the $\Lambda'$ we can assume that
\begin{equation}
\Lambda'\geq \sup_{p\in N_{t_0}} |p| + 10\sqrt{t_1-t_0}+10C\delta.
\end{equation}

Now if $N_t$ is compact at any given $t\in [t_1,t_0]$, then it must be the case that $R(p,t)\geq \delta$ for every $p\in N_t$ with $x_3(p)\geq \Lambda'$. Otherwise, Corollary \ref{R_inf_bound} would imply that $(p,t)$ is $\varepsilon$-spherical at a scale $\leq C\delta$, and $N_t$ could not reach $N_{t_0}$ at time $t_0$. Hence, 
\begin{equation}\label{speed_back}
-\frac{d}{dt}\sup_{p\in N_t} x_3(p)\leq 2\delta^{-1}
\end{equation}
whenever $N_t$ is compact and $\sup_{p\in N_t} x_3(p)\geq \Lambda'$.
Combining \eqref{coarse_comp} and \eqref{speed_back} we conclude that $N_t$ is compact for all $t\in [t_1,t_0]$; a contradiction.
\end{proof}

\begin{corollary}\label{R_growth}
There exist a constant $C=C(\mathcal M)<\infty$ such that 
\begin{equation}
Z(p,t) \leq CR(p,t)
\end{equation}
for all $(p,t)\in\mathcal M$. In particular, $\lim_{x_3(p)\rightarrow \infty} R(p,t)=\infty$.
\end{corollary}

\begin{proof}
By Lemma \ref{no_comp_conn} the flow cannot be $\varepsilon$-spherical at any scale. Thus, by Section \ref{sec_necks_back_in_time} the regularity scale and the cylindrical scale must be comparable. Together with Proposition \ref{cylind_growth} this yields $\lim_{x_3(p)\rightarrow \infty} R(p,t)=\infty$.
\end{proof}

Lemma \ref{no_comp_conn} also quickly implies a local type I curvature estimate:

\begin{proposition}\label{prop_type1}
For all $\Lambda < \infty$ there is a $C(\Lambda) < \infty$ with the following significance. If $\mathcal{M}$ reaches a point $p_0$ at time $t_0$, then we have the curvature bound
\begin{equation}
|A(p,t)| \leq \frac{C}{\sqrt{t_0 - t}}
\end{equation}
for all  $p \in B(p_0, \Lambda \sqrt{t_0-t})$ at all $t < t_0$.
\end{proposition}

\begin{proof}
Suppose towards a contradiction that for some $\Lambda<\infty$ there is a sequence of points $(p_i,t_i)\in \mathcal{M}$ such that the estimate fails for $C_i = i$.

After shifting $(p_i,t_i)$ to $(0,0)$ and rescaling, we get a sequence of flows $\mathcal M^i$ which reach the point $0$ at time $0$, such that there are points $q_i\in M^i_{-1}\cap B(0,\Lambda)$ with $|A|(q_i,-1)> i$. Moreover, the time slices of $\mathcal{M}^i$ have no compact connected components by Lemma \ref{no_comp_conn}.

By the entropy assumption we can pass to a subsequential limit $\mathcal M^i\to \mathcal M$, which is an ancient low entropy flow. By Brakke's clearing out lemma \cite{Brakke_book} the limit flow $\mathcal M$ reaches the origin at time $0$. Moreover, the limit $\mathcal M$ cannot have any compact connected component at any $t<0$.

Thus, $\mathcal{M}$ has no spherical singularities when $t<0$ and so $\mathcal{M}\cap \{ t<0\}$ is smooth. In particular, $\mathcal{M}$ satisfies $|A|\leq C$ on $B(0,2\Lambda) \times [-2,-1/2]$ for some $C<\infty$; for $i$ large enough this contradicts $|A|(q_i,-1)> i$.
\end{proof}

Combining the above results we now get a global curvature bound:

\begin{theorem}[Global curvature bound]\label{curv_bound}
Let $\mathcal M$ be an ancient low entropy flow such that the plus mode is dominant. Then $\mathcal M$ is eternal with globally bounded curvature, namely
\begin{equation}
T_E(\mathcal M)=\infty,
\end{equation}
and there is a constant $C=C(\mathcal M)<\infty$ such that
\begin{equation}
\sup_{(p,t)\in \mathcal{M}}|A|(p,t)\leq C.
\end{equation}
\end{theorem}

\begin{proof}
By Corollary \ref{R_growth} the flow is eternal, i.e. $T_E(\mathcal M)=\infty$, and by Lemma \ref{no_comp_conn} the time slices have no compact connected components.
If $(p,t)\in \mathcal M$ is any point with $R(p,t)<\frac{1}{10}$, then by Proposition \ref{cylind_growth} and Corollary \ref{R_growth} there exists a point $(q,t)$ with $0 < x_3(q)-x_3(p)< C\rho$ such that $R(q,t)=1$. But then, by unit regularity there is a point $(q',t+\delta)$ with $|q'-q| \leq 1$ and $R(q',t+\delta)\leq 2$, where $\delta>0$ is a uniform constant. Using the fine neck theorem (Theorem \ref{thm Neck asymptotic}) we see that $|x_1(p)-x_1(q)|+|x_2(p)-x_2(q)|$ is bounded. Applying the local type I curvature estimate (Proposition \ref{prop_type1}) centered at $(q',t+\delta)$ we conclude that $|A(p,t)|\leq C$.
\end{proof}

\bigskip

\subsection{Cap size control and asymptotics}\label{sec_cap_size_asymptotics}

The goal of this section is to prove Theorem \ref{thm_asympt_par}, which states that our surfaces $M_t$ have a cap of uniformly controlled size and open up like a parabola. To begin with, let us consider the ``height of the tip" function
\begin{equation}
\psi(t):= \inf_{x\in M_t} x_3.
\end{equation}

\begin{lemma}
The function $\psi$ is strictly increasing and satisfies
\begin{equation}
\lim_{t\to \pm \infty} \psi(t)=\pm \infty.
\end{equation}
\end{lemma}

\begin{proof}
Applying Corollary \ref{fine_neck_cor} we see that $\psi(t)>-\infty$ for every $t\in \mathbb{R}$ and that the infimum is actually a minimum (using also Corollary \ref{cor_barrier}). By the comparison with planes / the strong maximum principle, the function $\psi$ is strictly increasing. Applying Corollary \ref{fine_neck_cor} again we infer that $\lim_{t\to - \infty} \psi(t)=- \infty$.
Finally, by comparison with conical expanders (using also Corollary \ref{cor_barrier}) we have $\lim_{t\to  \infty} \psi(t)= \infty$.
\end{proof}

\begin{proposition}[Fast tip points]\label{fast_tip}
There exists a constant $Q=Q(\mathcal{M})<\infty$ such that every $p\in M_t$ with $x_3(p)=\psi(t)$ satisfies $R(p,t)\leq Q$.
\end{proposition}

\begin{proof}
If the assertion fails we can find a sequence $(p_j,t_j)\in \mathcal M$ with $x_3(p_j)=\psi(t_j)$ such that $Q_j:=R(p_j,t_j)\to \infty$. Let $\mathcal{M}^j$ be the flow that is obtained by parabolically rescaling by $Q_j^{-1}$ around $X_j=(p_j,t_j)$. Note that $\mathcal{M}^j$ has expansion parameter $\bar{a}^j=Q_j^{-1}\bar{a}\to 0$.

Up to a subsequence, we can pass to a limit $\mathcal{M}^\infty$. The limit $\mathcal{M}^\infty$ is itself a noncompact ancient low entropy flow. Since its zero time slice is contained in a half space, $\mathcal{M}^\infty$ cannot be the cylinder. Hence, by the fine neck theorem (Theorem \ref{thm Neck asymptotic}) it has an expansion parameter $a^{\infty}\neq 0$. However, it follows from the fine neck theorem and its proof that $\bar{a}^j\to  a^\infty$, contradicting $\bar{a}^j\to 0$. This finishes the proof of the proposition.
\end{proof}

For each time, select a point $p_t\in M_t$ such that $x_3(p_t)=\psi(t)$.

\begin{theorem}[Cap size control]\label{thm_asympt_par}
Let $\mathcal{M}$ be an ancient low entropy flow such that the plus mode is dominant. Then there exists a $C=C(\mathcal{M})<\infty$ such that for $t\in\mathbb{R}$ every point in $M_t\setminus B_C(p_t)$ lies on a fine neck. In particular, the surface $M_t$ has exactly one end.

Moreover, $M_t\setminus B_C(p_t)$ is the graph of a function $r$ in cylindrical coordinates around the $z$-axis satisfying  
\begin{equation}\label{expansion_cylindrical}
r(t,z,\theta)=\sqrt{2(z-\psi(t))}+o(\sqrt{z-\psi(t)})
\end{equation}
for $z\geq \psi(t)+C$, and the height of the tip function $\psi$ satisfies
\begin{equation}
\psi(t)= t + o(|t|)
\end{equation}
\end{theorem}

\begin{proof}
Let $p\in M_t$, and let $t_{\ast}$ be the time in which $\psi(t_{\ast})=x_3(p)$. 
By Proposition \ref{fast_tip} (fast tip points) we have $R(p_{t_{\ast}})\leq Q$, an by the global curvature bound (Theorem \ref{curv_bound}) we have $t_{\ast}-t \geq c(x_3(p)-\psi(t))$. Thus, applying the fine neck theorem (Theorem \ref{thm Neck asymptotic}) and Proposition \ref{prop_c0est} at $(p_{t_{\ast}},t_{\ast})$ we see that if $x_3(p)-\psi(t)$ is sufficiently large, then $p$ lies on a fine neck. Together with Corollary \ref{cor_barrier} and Corollary \ref{fine_neck_cor} this proves that we have a cap of controlled size $C=C(\mathcal M)<\infty$, and exactly one end.\\

Fix $t$. For $z\geq \psi(t)+C$, let $r(z)$ be (a representitive of) the radius of the fine neck of $M_{t}$ at $x_3(x)=z$. By the above and the fine neck theorem (Theorem \ref{thm Neck asymptotic}) we have
\begin{equation}\label{rec}
r\left(z+r(z)\right)=r(z)+1+o((z-\psi(t))^{-1/160})).
\end{equation}
This can be iterated to 
\begin{multline}
r\left(z+kr(z)+\left[1+o((z-\psi(t))^{-1/160}))\right]\tfrac{k(k-1)}{2}\right)\\
=r(z)+k\left[1+o((z-\psi(t))^{-1/160}))\right].
\end{multline}
Setting $h=z-\psi(t)$ we get
\begin{equation}\label{it_eq}
r\left(z+kr(z)+\left[1+o(h^{-1/160})\right]\tfrac{k(k-1)}{2}\right)=r(z)+k\left[1+o(h^{-1/160})\right].
\end{equation}
The argument of $r$ in the left hand side of \eqref{it_eq} is
\begin{equation}
s:=z+kr(z)+[1+o(h^{-1/160})]\tfrac{k(k-1)}{2},
\end{equation}
and solving for $k$, we get 
\begin{equation}\label{k_sol}
k=\frac{-(r(z)-\frac{1}{2})+\sqrt{[2+o(h^{-1/160})](s-z)}}{1+o(h^{-1/160})}.
\end{equation}
Combining \eqref{it_eq} and \eqref{k_sol}, and taking $h$ very large, we obtain
\begin{equation}\label{r_eq}
r(h)=\sqrt{2h}+o(\sqrt{h}).
\end{equation}
This proves \eqref{expansion_cylindrical}.\\

Finally, in light of \eqref{r_eq}, one can put a a bowl with tip curvature $(1+\delta)$ inside the domain bounded by $M_{t_0}$, and a bowl with tip curvature $(1-\delta)$ outside that domain. By the avoidance principle, this implies $(1-2\delta)t \leq \psi(t)\leq (1+2\delta) t$ for large enough $t$, hence 
\begin{equation}
\psi(t)=t+o(t)
\end{equation}
for positive $t$. In fact, the above steps are uniform in the original time $t_0$: For every $\delta$ there exist a uniform $s$ such that 
\begin{equation}
(1-2\delta)(t-t_0) \leq \psi(t)-\psi(t_0)\leq (1+2\delta) (t-t_0),
\end{equation}
whenever $t-t_0\geq s$. Thus, for every $t<0$ we have
\begin{equation}
(1-2\delta)|t| \leq -\psi(t) \leq (1+2\delta)|t|,
\end{equation}
and so 
\begin{equation}
\psi(t)=t+o(|t|)
\end{equation}
holds for all $t$. This finishes the proof of the theorem.
\end{proof}

\bigskip

\section{Rotational symmetry}\label{sec_rot_symm}

\subsection{Fine expansion away from the cap}

Let $\mathcal{M}$ be as in the previous section. The goal of this section is to prove Theorem \ref{thm_neck_asympt}, which shows that the cylindrical end becomes rotationally symmetric at very fast rate, and also controls the position of the cap in the $xy$-plane uniformly in time.\\

Given a point $q\in \mathbb{R}^3$ and a direction $w \in S^2$, we denote by $K$ the \emph{normalized rotation vector field} that corresponds to a rotation by $\pi/2$ around the axis $W=\{q+wt:t \in \mathbb{R}\}$, namely
\begin{equation}
K(x) = SJS^{-1}(x-q),\qquad \textrm{ where } J = \begin{bmatrix} 0 & -1 & 0 \\ 1 & 0 & 0 \\ 0 & 0 & 0 \end{bmatrix},
\end{equation}
and $S\in \textrm{SO}_3$ is any rotation matrix with $Se_3=w$.\footnote{$S$ is only determined up to a one-parameter choice, but $K$ is well-defined.}

\begin{definition}[{c.f. \cite[Def. 4.3]{BC}}]\label{def_eps_sym}
A point $X\in \mathcal{M}$ with $H(X)>0$ is called \emph{$\delta$-symmetric} if there exists a normalized rotation vector field $K$ such that
\begin{equation}
|\langle K,\nu \rangle \, H | \leq \delta \textrm { in the parabolic ball } P(X,10H^{-1}(X)).
\end{equation}
\end{definition}

The following proposition shows that $\mathcal M$ becomes \emph{$\delta$-symmetric} at a very fast rate if one moves away from the cap.

\begin{proposition}
There exist a constant $C=C(\mathcal{M})<\infty$ with the following significance. If $X=(x,t) \in \mathcal{M}$ is any point with
\begin{equation}
x_3-\psi(t)  \geq C,
\end{equation}
then 
\begin{equation}
\textrm{$X$ is $\left(x_3-\psi(t)\right)^{-300}$-symmetric.}
\end{equation}
\end{proposition}

\begin{proof}
We recall from Theorem \ref{thm_asympt_par} (cap size control) that there exists a constant $C_0<\infty$ such that every point $x\in M_t\setminus B_{C_0}(p_t)$ lies on the center of a fine neck. In particular, every point $X=(x,t)$ with 
\begin{equation}
x_3-\psi(t)  \geq C_0
\end{equation}
satisfies the a-priori assumptions of the neck-improvement theorem from Brendle-Choi \cite[Thm. 4.4]{BC}. Hence, if $X=(x,t)\in\mathcal{M}$ is a point with
\begin{equation}
x_3-\psi(t)\geq 2^{\frac{j}{400}}C_0
\end{equation}
then we can iteratively apply the neck-improvement theorem \cite[Thm. 4.4]{BC} to conclude that $X$ is $2^{-j}\varepsilon$-symmetric. This implies the assertion. 
\end{proof}

\begin{corollary}[strong symmetry]\label{cor_strong_symm}
There exist a constant $C=C(\mathcal{M})<\infty$ with the following significance. If $X=(x,t) \in \mathcal{M}$ is any point with $x_3-\psi(t)  \geq C$,
then there exist a direction $w_X \in S^2$ and a point $q_X\in\mathbb{R}^3$ with
\begin{equation}
|w_X-e_3|\leq \tfrac{1}{100},\qquad \langle q_X,e_3\rangle = x_3,
\end{equation}
such that the normalized rotation vector field $K_X(y)=S_XJS_X^{-1}(y-q_X)$, where $S_X\in\textrm{SO}_3$ with $S_Xe_3=w_X$, satisfies the estimate
\begin{equation}
\sup_{P(X,10H^{-1}(X))}
|\langle K_X,\nu \rangle \, H |\leq  \left(x_3-\psi(t)\right)^{-300}.
\end{equation}
\end{corollary}

\begin{proof}
Since fine necks are very close to the asymptotic cylinder, we always have $|w_X - e_3|\leq \tfrac{1}{100}$. In addition, by moving $q_X$ along the axis $W$ we can always arrange that $\langle q_X,e_3\rangle = x_3$. Hence, the corollary follows from the proposition.
\end{proof}

\begin{definition}\label{Strong sym definition}
We call any triple  $(X,w_X,q_X)$ that satisfies the conclusion of Corollary \ref{cor_strong_symm} a \emph{strongly symmetric triple}.
\end{definition}

The following lemma shows that nearby strongly symmetric triples at the same time align well with each other.

\begin{lemma}[alignment]\label{Local estimate for strong symmetry}
There exists a constant $C=C(\mathcal{M})<\infty$ with the following significance. If $(X,w_{ X},q_{ X})$ and $( Y,w_{{Y}},q_{ Y})$ are strongly symmetric triples with $ X=( x, t)$, $Y=( y, t)$ and $ | x- y|H(X)\leq 1$, then
\begin{equation}\label{local axis estimate}
|w_X-w_Y | \leq C ( x_3-\psi(t))^{-300},
\end{equation}
and
\begin{equation}\label{local center estimate}
|\langle q_{X} - q_{ Y},e_1\rangle | + |\langle q_{ X} - q_{Y},e_2\rangle | \leq C( x_3-\psi( t))^{-\frac{599}{2}}.
\end{equation}
\end{lemma}

\begin{proof}
Without loss of generality, after suitable rotations and translations, we can assume that $t=0$, $x_3=0$, $w_{{X}}=e_3$, $q_{ X}=0$, and
\begin{equation}
S_{{Y}}=\begin{bmatrix} 1 & 0 & 0 \\ 0 & \cos\varphi& -\sin\varphi \\ 0 & \sin\varphi & \cos\varphi \end{bmatrix},
\end{equation}
where $\varphi$ is a fixed angle with $0\leq \varphi \leq \tfrac{1}{10}$.

\bigskip

We express $\mathcal{M}\cap P({X},10H^{-1}({X}))$ in cylindrical coordinates over the $z$-axis, namely we parametrize as
\begin{equation}
(\theta,z,t)\mapsto \big(r(\theta,z,t)\cos\theta,r(\theta,z,t)\sin\theta,z\big).
\end{equation}
 In these coordinates, one can directly compute that 
\begin{align}
\nu=&\frac{\left( \cos\theta+r^{-1}r_\theta\sin\theta,\sin\theta-r^{-1}r_\theta\cos\theta,-r_z\right)}{\sqrt{1+r^{-2}|r_\theta|^2+|r_z|^2}},
\end{align}
and
\begin{equation}
\big|\langle K_{X},\nu\rangle\big|=\frac{ |r_\theta |}{\sqrt{1+r^{-2}|r_\theta|^2+|r_z|^2}}.
\end{equation}
Since $X$ is the center of a (fine) neck, we have
\begin{equation}\label{eq_neck_simple}
r^{-2}|r_\theta|^2+|r_z|^2\leq 1+10\varepsilon,
\end{equation}
and
\begin{equation}
\frac{1-\varepsilon}{r}\leq H(X) \leq \frac{1+\varepsilon}{r}.
\end{equation}
Combining these equations with Corollary \ref{cor_strong_symm} (strong symmetry) we infer that
\begin{equation}
\frac{|r_\theta|}{r}\leq 2 (-\psi(0))^{-300}.
\end{equation}
Together with Theorem \ref{thm_asympt_par} this yields the estimate
\begin{equation}\label{eq_est_r_theta}
|r_\theta| \leq C r^{-599}
\end{equation}
in the parabolic ball $P({X},10H^{-1}( X))$.

\bigskip

Now, for the normalized rotation vector with center $q_Y=(q_1,q_2,q_3)$ and axis $w_Y=(0,-\sin \varphi,\cos\varphi)$, we compute
\begin{align}
K_{Y}(x)=S_{ Y}JS_{ Y}^{-1}(x-q_{Y})=\begin{bmatrix} (x_2-q_2)\cos\varphi+(x_3-q_3)\sin\varphi \\ -(x_1-q_1)\cos\varphi \\ -(x_1-q_1)\sin\varphi \end{bmatrix},
\end{align}
and
\begin{align}
&\!\!\!\!\!\!\!\!\!\!\!\!\!\!\!\!\langle K_{Y}, \nu\rangle \sqrt{1+r^{-2}|r_\theta|^2+|r_z|^2}\notag\\
=&\, r^{-1}(q_1x_2-x_1q_2)\cos\varphi+r^{-1}x_1(x_3-q_3)\sin\varphi+r_z(x_1-q_1)\sin\varphi \notag\\
&+r^{-2}r_\theta \big[x_2(x_2-q_2)\cos\varphi+x_2(x_3-q_3)\sin\varphi+x_1(x_1-q_1)\cos\varphi \big]\label{rotation vector comparison I}.
\end{align}

Arguing as above, using Theorem \ref{thm_asympt_par} and Corollary \ref{cor_strong_symm} we obtain
\begin{equation}\label{eq_est_ky}
|\langle K_{ Y}, \nu\rangle |  \leq Cr^{-599}
\end{equation}
in the parabolic ball $P(X, 8 H^{-1}(X))$.
In addition, we have the rough estimate
\begin{equation}\label{q position bound}
 |q_1|+|q_2| + |q_3| \leq 10 r.
\end{equation}
Now, from equation \eqref{rotation vector comparison I}, using the estimates \eqref{eq_neck_simple}, \eqref{eq_est_r_theta}, \eqref{eq_est_ky} and \eqref{q position bound}, we infer that
\begin{equation}\label{rotation vector comparison II}
\frac{q_1x_2-x_1q_2}{r} \cos\varphi +\frac{x_1}{r} (x_3-q_3)\sin\varphi+r_z(x_1-q_1)\sin\varphi\leq Cr^{-599}.
\end{equation}
At time $t=0$, we consider the points with $\theta=0$. Then, $(x_1,x_2)=(r,0)$ and equation \eqref{rotation vector comparison II} yields
\begin{equation}
-q_2 \cos\varphi+ \left( x_3- q_3+r_z(r-q_1)\right)\sin\varphi\leq Cr^{-599}.
\end{equation}
In the case $q_2 \leq 0$, we consider the points with  $x_3 =20r$. Then, using also $| q_3| \leq 10r$ and $\cos\varphi \geq \tfrac12$, we obtain
\begin{equation}
\tfrac12 |q_2|+\left( 10r+r_z(r-q_1)\right)\sin\varphi\leq Cr^{-599}.
\end{equation}
Moreover, since $X$ lies on a (fine) neck, we have $|r_z| \leq \varepsilon$. Hence, \begin{equation}\label{bound for S+q2 }
\tfrac12 |q_2|+5 r\sin\varphi\leq Cr^{-599}.
\end{equation}
Since $\sin\varphi\geq 0$, we infer that
\begin{align}\label{est_q2}
|q_2| \leq Cr^{-599},
\end{align}
and
\begin{align}\label{est_sinp}
 |\sin\varphi| \leq Cr^{-600}.
\end{align}
In the case $q_2 \geq 0$, we obtain the same estimates by considering points with $ x_3=-20r$. Similarly, considering points with $\theta=\pi/2$ we obtain
\begin{equation}\label{est_q1}
|q_1|\leq Cr^{-599}.
\end{equation}
Since $|w_X-w_Y|\leq C|\sin \varphi|$, these inequalities prove the lemma.
\end{proof}

We are now ready to prove the main theorem of this section:

\begin{theorem}[fine asymptotics]\label{thm_neck_asympt}
There exist a fixed vector $ q=( q_1, q_2,0)\in \mathbb{R}^3$ and a large constant $C<\infty$ (both depending on $\mathcal M$) such that for all $t\in\mathbb{R}$ the surface $(M_t- q)\cap \{ x_3-\psi(t)\geq C\}$ can be expressed in cylindrical coordinates over the $z$-axis with the estimate
\begin{equation}\label{global angular derivative}
|\partial_\theta r|(\theta,z,t) \leq  r(\theta,z,t)^{-100}.
\end{equation} 
\end{theorem}

\begin{proof}
Given any time $ t \in \mathbb{R}$, we choose a sequence of points $ y_j \in M_{ t}$ satisfying
\begin{equation}
\langle  y_j,e_3\rangle=\psi( t)+j.
\end{equation}
By Corollary \ref{cor_strong_symm} (strong symmetry), for $j$ large enough (which we will always assume from now on), these points are part of strongly symmetric triples $({Y}_j,w_j, q_j)$, where ${Y}_j=( y_j, t)$.
Lemma \ref{Local estimate for strong symmetry} gives the estimates
\begin{equation}\label{est_delta_s}
|w_{j+1}-w_j| \leq C j^{-300},
\end{equation}
and
\begin{equation}\label{est_delta_qu}
|\langle  q_{i+1}- q_i,e_1\rangle |+|\langle  q_{i+1}-  q_i,e_2\rangle | \leq C j^{-299}
\end{equation}
Since $w_j$ converges to $e_3$, using \eqref{est_delta_s} we get
\begin{equation}
|w_j-e_3| \leq C \sum_{k=j}^{\infty}k^{-300}\leq C j^{-299}.
\end{equation}
Similarly, equation \eqref{est_delta_qu} implies that $\langle q_{j},e_1\rangle e_1+\langle  q_{j},e_2\rangle e_2$ converges to a limit  $q\in \mathbb{R}^2\times \{0\}$ with the estimate
\begin{equation}
|\langle  q_{j}-q,e_1 \rangle  |+ |\langle  q_{j}-q,e_2 \rangle  |
 \leq C \sum_{k=j}^{\infty}k^{-299}\leq C j^{-298}.
\end{equation}

By translating $\mathcal{M}$, we may assume $q=0$. Then, we have
\begin{align}
&|w_j-e_3| \leq  C j^{-299},\label{global axis bound}\\
& |\langle  q_j,e_1\rangle |+ |\langle  q_j,e_2\rangle | \leq C j^{-298}\label{global q position bound}.
\end{align}
Together with Theorem \ref{thm_asympt_par} and Corollary \ref{cor_strong_symm} (strong symmetry), this yields the assertion.
\end{proof}

\bigskip

\subsection{Moving plane method}

The goal of this section is to prove Theorem \ref{thm_rot_symm}, which says that $\mathcal M$ is rotationally symmetric.\\

Let us start with some preliminaries about the Hopf lemma and the strong maximum principle for graphical mean curvature flow.
Suppose $u$ and $v$ are graphical solutions of the mean curvature flow in the parabolic ball $P(0,\delta)$ around the origin of size $\delta>0$, namely
\begin{align}
u_t&=\sqrt{1+|Du|^2} \textrm{div} \left( \frac{Du}{\sqrt{1+|Du|^2}}  \right),\\
v_t &=\sqrt{1+|Dv|^2}  \textrm{div} \left( \frac{Dv}{\sqrt{1+|Dv|^2}}  \right).
\end{align}
Then, their difference $w=u-v$ satisfies the linear equation
\begin{equation}\label{linear eq for w=u-v}
w_t=a_{ij}w_{ij}+b_iw_i,
\end{equation}
where
\begin{equation}
a_{ij}(x,t)=\delta_{ij}-\frac{u_iu_j}{ 1+|Du|^2}(x,t),
\end{equation}
and
\begin{equation}
b_i(x,t)=-\frac{v_{ij}(u_j+v_j)}{ 1+|Du|^2}(x,t)+\frac{v_{kl} v_kv_l(u_i+v_i)}{( 1+|Du|^2)(1+|Dv|^2)}(x,t).
\end{equation}
Hence, standard results for linear parabolic equations (see e.g. \cite{Lieberman}) yield:

\begin{lemma}[Hopf lemma]\label{Hopf lemma for graphs}
Let $u$ and $v$ be graphical solutions of the mean curvature flow in $P(0,\delta)\cap \{x_1\leq 0\}$.
Suppose that $u(0,0)=v(0,0)$, and that $u< v$ in $P(0,\delta)\cap \{x_1< 0\}$. Then, $\tfrac{\partial u}{\partial x_1}(0,0)>\tfrac{\partial v}{\partial x_1}(0,0)$.
\end{lemma}

\begin{proposition}[strong maximum principle]\label{Strong maximum principle for graphs}
Let $u$ and $v$ be graphical solutions of the mean curvature flow in $P(0,\delta)$. Suppose that $u(0,0)=v(0,0)$, and that $u\leq v$ in $P(0,\delta)$. Then, $u=v$ in $P(0,\delta)$.
\end{proposition}

\bigskip

Now, as in the previous section, let $\mathcal{M}=\{M_t\}_{t\in\mathbb{R}}$ be an ancient low entropy flow in $\mathbb{R}^3$, where the plus mode is dominant. By the global curvature bound (Theorem \ref{curv_bound}) and the entropy assumption, there exists some uniform graphical scale $\delta_0>0$ at which the Hopf lemma (Lemma \ref{Hopf lemma for graphs}) and the strong maximum principle (Proposition \ref{Strong maximum principle for graphs}) can be applied.\\

To set up the (parabolic variant of the) moving plane method, given a constant $\mu \geq 0$ we consider the surfaces
\begin{align}
M_t^{\mu -}&=M_t\cap \{ x_1 < \mu\},\\
M_t^{\mu +}&= M_t\cap \{ x_1 > \mu\}.
\end{align}
Moreover, we set $M_t^{\mu}=M_t\cap \{ x_1 = \mu\}$, and denote by $M_t^{\mu <}$ the surface that is obtained from $M_t^{\mu +}$ by reflection about the plane $\{x_1=\mu\}$, namely
\begin{equation}
M_t^{\mu <}=\left\{(2\mu-x_1,x_2,x_3):x\in M_t^{\mu +} \right\}.
\end{equation}

Similarly, denoting by $K_t\subset\mathbb{R}^3$ the closed domain bounded by $M_t$ we consider the regions
\begin{align}
K_t^{\mu -}&=K_t\cap \{ x_1 < \mu\},\\
K_t^{\mu +}&= K_t\cap \{ x_1 > \mu\},
\end{align}
and 
\begin{equation}
K_t^{\mu <}=\left\{(2\mu-x_1,x_2,x_3):x\in K_t^{\mu +} \right\}.
\end{equation}

\begin{definition}
We say \emph{the moving plane can reach $\mu$} if for all $\tilde{\mu}\geq \mu$ we have the inclusion $K_t^{\tilde{\mu} <}\subseteq K_t^{\tilde{\mu} -}$ for all $t\in \mathbb{R}$.
\end{definition}

\bigskip

The following proposition shows that the reflected domain cannot touch at spatial infinity.

\begin{proposition}[no contact at infinity]\label{Infinity Dirichlet}
For every $\mu>0$, there exists a constant $h_\mu <\infty $ such that
\begin{equation}
K_t^{\mu  <}\cap \{x_3 \geq \psi(t)+h_\mu \}\subseteq \textrm{Int}(K_t^{\mu-})
\end{equation}
for every $t\in\mathbb{R}$. 
\end{proposition}

\begin{proof}
Denote the position vector of $M_t\cap \{x_3 \geq \psi(t)+C\}$ by
\begin{equation}
X(\theta,z,t)=\left( r(\theta,z,t)\cos\theta,r(\theta,z,t)\sin\theta,z\right).
\end{equation}
Given $(z_0,t_0)$ with $z_0-\psi(t_0)\geq C$, using Theorem \ref{thm_neck_asympt} (fine asymptotics), we obtain
\begin{equation}
|r(z_0,\theta,t_0)-r(z_0,\theta_0,t_0)| \leq \frac{2\pi}{r(z_0,\theta_0,t_0)^{100}},
\end{equation}
from which we directly infer that
\begin{equation}
K_{t_0}^{\mu <}\cap \{x_3=z_0\}\cap \{x_1 \leq \tfrac{1}{2}\mu\} \subseteq \textrm{Int}(K_{t_0}^{\mu -})
\end{equation}
for sufficiently large $z_0-\psi(t_0)$.

Let $\tilde{\mu}\geq \mu$. Using Theorem \ref{thm_neck_asympt} (fine asymptotics), we see that if $\theta \in (-\tfrac{\pi}{2},0)$ is such that
\begin{equation}
r(\theta,z_0,t_0)\cos\theta \in [\tfrac{1}{2}\tilde{\mu},\tfrac{3}{2}\tilde{\mu}],
\end{equation}
then for sufficiently large $z_0-\psi(t_0)$ we have
\begin{equation}
\langle X_\theta,e_2\rangle=r_\theta \sin\theta+r\cos\theta \geq -r^{-100}|\sin\theta|+r\cos\theta >0.
\end{equation}
Hence, we have a graph with positive slope that can be reflected.
Repeating the same argument for $\theta \in (0,\tfrac{\pi}{2})$ we conclude that
\begin{equation}
K_{t_0}^{\mu <}\cap \{x_3=z_0\}\cap \{x_1 \geq \tfrac{1}{2}\mu\} \subseteq \textrm{Int}(K_{t_0}^{\mu -})
\end{equation}
for sufficiently large $z_0-\psi(t_0)$. This proves the proposition.
\end{proof}

\bigskip

\begin{corollary}[start plane]\label{cor_start_plane}
There exists some $\mu<\infty$, such that the moving plane can reach $\mu$.
\end{corollary}

\begin{proof}
By Theorem \ref{thm_asympt_par} (cap size control) and Theorem \ref{thm_neck_asympt} (fine asymptotics) there exists a $C=C(\mathcal M)<\infty$ such that whenever $x_3 - \psi(t)\leq C$ we have
\begin{equation}\label{uniform cap position control}
x_1^2+x_2^2 \leq 10C,
\end{equation}
namely we uniform control for the position of the cap. Together with Proposition \ref{Infinity Dirichlet} (no contact at infinity) this implies the assertion.
\end{proof}

\bigskip

Now,  for $h_\mu$ as in the proposition and $\delta>0$ we define
\begin{align}
&E_t^{\mu}=\{x_3 \leq \psi(t)+h_{ \mu /2}\}, && E_t^{\mu,\delta}=\{x\in E_t^{\mu}: d(x,M_t^\mu)\geq \delta\}.
\end{align}

\begin{lemma}[distance gap]\label{distance gap}
Suppose the moving plane can reach $\mu>0$. Then, there exists a positive increasing function $\alpha:(0,\delta_0)\to \mathbb{R}_+$ such that 
\begin{equation}
d(M_t^{\mu -},K_t^{\mu <}\cap  E_t^{\mu,\delta}) \geq \alpha(\delta)>0
\end{equation}
for all $t\in \mathbb{R}$.
\end{lemma}

\begin{proof}
We will first show that 
\begin{equation}\label{showfirst}
K_t^{\mu <}\subseteq \textrm{Int}(K_t^{\mu -})
\end{equation}
for all $t\in \mathbb{R}$. Indeed, by definition, we have $K_t^{\mu <}\subseteq K_t^{\mu -}$. If \eqref{showfirst} fails, then there must be some $t\in \mathbb{R}$ and some $p\in M_t^{\mu <}\cap M_t^{\mu -}$; this contradicts the strong maximum principle (Proposition \ref{Strong maximum principle for graphs}). This proves \eqref{showfirst}.

Now, suppose towards a contradiction that for some $\delta>0$ we have
\begin{equation}
\inf_{t\in\mathbb{R}} d(M_t^{\mu -} ,K_t^{\mu <}\cap  E_t^{\mu,\delta})=0.
\end{equation}
Choose a sequence of space-time points $(x_i,t_i)\in \mathcal{M}$ such that $x_i \in M_{t_i}^{\mu <}\cap  E_{t_i}^{\mu,\delta}$ and $\lim_{i\to \infty} d(x_i,K_{t_i}^{\mu -})=0$. By Proposition \ref{Infinity Dirichlet} (no contact at infinity), Theorem \ref{thm_neck_asympt} (fine asymptotics), and the uniform cap position control \eqref{uniform cap position control}, the distance between $x_i$ and the point $\psi(t_i)e_3$ is uniformly bounded. Hence, by Theorem \ref{curv_bound} (global curvature bound) and Theorem \ref{thm_asympt_par} (cap size control), we can take subsequential limits $\overline{\mathcal{M}}$ and $\bar x$ of the flows $\mathcal{M}-(\psi(t_i)e_3,t_i)$ and points $x_i-\psi(t_i)e_3$. Applying the strong maximum principle (Proposition \ref{Strong maximum principle for graphs}) for $\overline{\mathcal{M}}$ at the spacetime point $(\bar x,0)$ gives a contradiction. This proves the lemma.
\end{proof}

\bigskip

\begin{lemma}[angle gap]\label{Angle gap}
Suppose that the moving plane can reach $\mu>0$. Then, there exists a positive constant $\theta_\mu>0$ such that $|\langle \nu(x,t),e_1\rangle|\geq \theta_\mu$ holds on $M_t^{\mu}\cap E_{t}^{\mu}$ for all $t\in\mathbb{R}$.
\end{lemma}

\begin{proof}
First, the Hopf lemma (Lemma \ref{Hopf lemma for graphs}) shows that $|\langle \nu(x,t),e_1\rangle|\neq 0$.

Now, suppose towards a contradiction there is a sequence $(x_i,t_i)\in \mathcal{M}$ such that $x_i \in M_{t_i}^{\mu}\cap E_{t_i}^{\mu}$ and $\lim_{i\to \infty} |\langle \nu(x_i,t_i),e_1\rangle|=0$. Then, as in the proof of Lemma \ref{distance gap}, we can take subsequential limits $\overline{\mathcal{M}}$ and $\bar x$ of the flows $\mathcal{M}-(\psi(t_i)e_3,t_i)$ and points $x_i-\psi(t_i)e_3$. Applying the Hopf lemma (Lemma \ref{Hopf lemma for graphs}) for $\overline{\mathcal{M}}$ at the spacetime point $(\bar x,0)$ gives a contradiction. This proves the lemma.
\end{proof}

\bigskip

\begin{theorem}\label{thm_rot_symm}
$\mathcal{M}$ is rotationally symmetric.
\end{theorem}

\begin{proof}
It is enough to show that the moving plane can reach $\mu=0$.
Consider the interval
\begin{equation}
I=\{ \mu \geq 0: \text{the moving plane can reach}\; {\mu}\}
\end{equation}
Note that  $I \neq \emptyset$ by Corollary \ref{cor_start_plane} (start plane). Let $\mu:=\inf I$, and observe that $\mu\in I$. Suppose towards a contradiction that $\mu>0$.

First, by Proposition \ref{Infinity Dirichlet} (no contact at infinity) we have
\begin{equation}
K_t^{\frac{\mu}{2} <}\cap (E_t^{\mu})^c \subseteq  K_t^{\frac{\mu}{2} -}
\end{equation}
for all $t\in \mathbb{R}$.

Next, by Lemma \ref{Angle gap} (angle gap) there exists a $\delta_1 \in (0,\min\{ \delta_0,\tfrac{\mu}{2}\})$ such that for $\delta\in (0,\delta_1)$ we have
\begin{equation}
K_t^{(\mu-\delta) <}\cap E_t^{\mu}\cap \{x_1 \geq \mu-2\delta_1\} \subseteq K_t^{(\mu-\delta)-}
\end{equation}
for all $t\in \mathbb{R}$.

Finally, combining the above with Lemma \ref{distance gap} (distance gap) we conclude that every $\delta\in (0,\min\{\delta_1,\alpha(\delta_1)\})$ we have
\begin{equation}
K_t^{(\mu-\delta) <}\subseteq K_t^{(\mu-\delta)-}.
\end{equation}
for all $t\in \mathbb{R}$. Hence, the moving plane can reach $\mu-\delta$; a contradiction. This proves the theorem.
\end{proof}

\bigskip

\section{Classification of ancient low entropy flows}\label{sec_classification_completion}

\subsection{The noncompact case}\label{sec_classification_noncompact}

Let $\mathcal M$ be an ancient noncompact low entropy flow that is not a round shrinking cylinder or a flat plane. By Corollary \ref{thm_compact}, the plus mode is dominant. As before, we normalize such that $a=a(\mathcal M)=1/\sqrt{2}$.

\begin{theorem}\label{thm_class_noncompact}
$\mathcal M$ is the bowl soliton.
\end{theorem}

\begin{proof}
From Section \ref{sec_cap_size} we know that $\mathcal M$ has a cap of controlled size and opens up like a parabola. From Section \ref{sec_rot_symm} we know that $\mathcal M$ is rotationally symmetric. Thus, at each time $M_t$ can be described by rotating a curve around the $z$-axis.  The curve $\gamma_t(s)$ ($0\leq s<\infty$) is in the $xz$-plane and satisfies $x(\gamma(0))=0$ and $\gamma'(0)=(1,0,0)$.\\

We will first show that the height function $z:M_t\to \mathbb{R}$ does not have any local maxima. Assume towards a contradiction that at some time $t_0$ there is a local maximum of the height function. Denote its value by $h(t_0)$. Since the mean curvature vector points downward, $h(t)$ is decreasing. On the other hand, the minimum of the height function $\psi(t)$ is an inreasing function. In particular, $h-\psi$ is decreasing. Since local maxima of $z$ cannot disapear as we go back in time (c.f. \cite{Angenent_sturm}), it follows that the height function has a local maximum for all $t\leq t_0$. Using Proposition \ref{fast_tip} we see that
\begin{equation}
\lim_{t\to -\infty} (h-\psi)(t)=\infty;
\end{equation}
this contradicts the fact that the cap size is bounded. Thus, the height function $z$ does not have any local maxima, and the $z$-componenent of the vector $\gamma_t'(s)$ is nonzero for all $s\neq 0$.\\

We can now describe $M_t$ by a function $r(z,t)$, where $\psi(t)\leq z<\infty$. To proceed we need the following lemma.

\begin{lemma}\label{lemma_radiusfunction}
The radius function $r=r(z,t)$ satisfies
\begin{equation}
r r_z = 1+o(1)
\end{equation}
uniformly as $z-\psi(t)\to \infty$. In particular, the radius function does not have critical points outside a cap of controlled size.
\end{lemma}

\begin{proof}[{Proof of Lemma \ref{lemma_radiusfunction}}]

We recall from \cite[Lem. 6.5]{BC} that on $\varepsilon$-necks one has the estimate
\begin{equation}
|rr_z|+|r^2r_{zz}|\leq C,
\end{equation}
in particular
\begin{equation}\label{second order derivative of level area}
|\partial_{z}^2r^2|\leq C_0/r.
\end{equation}
Now, we consider a point $(z_0,t_0)$ in a neck satisfying $\partial_z(r^2)(z_0,t_0)\neq 2$. Without loss of generality, we assume $\partial_z(r^2)(z_0,t_0)-2=2\delta>0$. Then, we have  $\partial_z(r^2)(z,t_0) \geq 2+\delta$ for $|z-z_0| \leq \tilde{\delta} r_0$ where $r_0= r(z_0,t_0)$ and $\tilde{\delta}=\min\{1,\tfrac12 C_0^{-1}\delta\}$. Therefore, we obtain
\begin{equation}
r^2(z_0+\tilde{\delta} r_0,t_0)-r^2(z_0,t_0)\geq (2+\delta)\tilde{\delta} r_0.
\end{equation}
Since $((0,0,z_0),t_0)$ is the center of a fine neck with rescaled time parameter $\tau$ satisfying $e^{-\frac{\tau}{2}}=r_0/\sqrt{2}$, using Theorem \ref{thm_asympt_par} we see that
\begin{equation}
\left|r\Big(z_0+ \tilde{\delta}r_0,\, t_0\Big )-r(z_0,t_0)-\tilde{\delta}\right|\leq  C r_0^{-\frac{1}{40}}.
\end{equation}
Combining the above inequalities yields
\begin{equation}
Cr_0^{\frac{39}{40}}+O(1) \geq \tilde{\delta}\delta  r_0.
\end{equation}
Therefore, $\tilde{\delta}=\tfrac12 C_0^{-1}\delta$ for sufficiently large $r_0$, and thus $\delta \leq Cr_0^{-\frac{1}{80}}$. Namely, 
\begin{equation}
\partial_z(r^2)-2 = O(r^{-\frac{1}{80}}).
\end{equation}
This proves the lemma.
\end{proof}

Continuing the proof of the theorem, suppose towards a contradiction that $r$ has a local maximum at some time $t_0$. The mean curvature vector at such a maximum points towards the $z$-axis and has size at least $1/r(z,t)$, hence
\begin{equation}
\frac{d}{dt} r(z_{max}(t),t) \leq -1/r(z_{max}(t),t).
\end{equation}
Integrating this differential inequality from $t_0$ back to time $-\infty$ (again, using \cite{Angenent_sturm} to conclude that the maximum point persists), we obtain that $r^2(z_{\max}(t),t)$ grows at least linearly in $-t$. In particular, $z_{\max}(t)-\psi(t)\rightarrow \infty$ as $t\rightarrow -\infty$, which, as $r_z(z_{\max}(t),t)=0$, contradicts Lemma \ref{lemma_radiusfunction}.\\

By the above, the (upwards) unit normal vector $\nu$ of our surface can never be horizontal or vertical (except at the tip). Thus
\begin{equation}
f:=\langle \nu, e_3\rangle >0.
\end{equation}
By Lemma \ref{lemma_radiusfunction} we have the asymptotics
\begin{equation}\label{f_asy}
f(z,t) = \frac{1+o(1)}{r}
\end{equation}
uniformly as $z-\psi(t)\to \infty$. In particular, there exists some universal $h_0<\infty$ such that $f(\psi(t)+h,t)\geq \frac{1}{10\sqrt{h}}$ for all $h\geq h_0$. We claim that there exists some universal $\delta>0$, such that $f(z,t)\geq \delta$ whenever $\psi(t)\leq z \leq \psi(t)+h_0$. Otherwise, taking a sequence of $(t_i,z_i)$ with $f(t_i,z_i)\rightarrow 0$ and passing to a subsequential limit, we obtain a flow (low entropy, with dominant plus mode, rotationally symmetric, with $a=\frac{1}{\sqrt{2}}$) with $r_z(z_0,0)=0$ for some $z_0$. By \cite{Angenent_sturm}, for all $t<0$, the radius function $r$ would have a local maximum, which was previously seen to be impossible. Hence, we have established the existence of a positive function $g:[0,\infty)\to \mathbb{R}_+$ such that
\begin{equation}\label{eq_lower_fg}
f(z,t)\geq g(z-\psi(t)).
\end{equation}
Moreover, by Theorem \ref{thm_asympt_par} we have the asymptotics
\begin{equation}
H(z,t) = \frac{1+o(1)}{r}
\end{equation}
uniformly as $z-\psi(t)\to \infty$. Combining this with \eqref{f_asy} and \eqref{eq_lower_fg}, we see that $H/f$ converges to $1$ at spatial infinity uniformly for all $t$. In particular, there is some $C<\infty$ such that
\begin{equation}
\left| \frac{H}{f}\right|\leq C.
\end{equation}
Suppose towards a contradiction that 
\begin{equation}
c:=\inf \frac{H}{f} < 1.
\end{equation}
Select $X_i=(x_i,t_i)\in \mathcal M$ such that $\frac{H}{f}(X_i)\to c$. Let $\mathcal{M}^i$ be the flow that is obtained from $\mathcal M$ by shifting $X_i$ to the space-time origin, and pass to a limit $\mathcal{M}^\infty$. For the limit $\mathcal{M}^\infty$ the function $\frac{H}{f}$ attains a minimum $c<1$ at the space-time origin. Together with $\frac{H}{f}\to 1$ at spatial infinity, this contradicts the strong maximum principle for the evolution equation
\begin{equation}
\partial_t \frac{H}{f}=\Delta \frac{H}{f}+ 2 \langle \nabla \log f, \nabla \frac{H}{f}\rangle.
\end{equation}
Hence,
\begin{equation}
 \inf \frac{H}{f} \geq 1.
\end{equation}
A similar argument shows that 
\begin{equation}
 \sup \frac{H}{f} \leq 1.
\end{equation}
Thus $H=f$. Hence, $M_t$ is a mean convex and noncollapsed (since the cap is compact) translating soliton, and thus by \cite{Haslhofer_bowl} must be the bowl soliton.
\end{proof}

\bigskip

\subsection{The compact case}

In this section, we treat the case where the neutral mode is dominant.

\begin{theorem}\label{thm_class_neutral}
If the neutral mode is dominant, then $\mathcal M$ is an ancient oval.
\end{theorem}

\begin{proof}
From Section \ref{sec_fine_neutral} we know that  $\mathcal M=\{M_t \}_{t\in (-\infty,T_E(\mathcal{M}))}$ is compact. Hence, the flow becomes extinct in a point at $T_E(\mathcal{M})<\infty$, and each time-slice has the topology of a sphere.

\bigskip

Since the blowdown for $t\to -\infty$ is a cylinder, for $t\leq \mathcal{T}$ there is a central neck that divides $M_t$ into two halves $M_t^\pm$. We will show that $M_t^+$ is mean convex and $\alpha$-noncollapsed for $t\leq \mathcal{T}$.

\bigskip

Consider the ``height of the tip" function
\begin{equation}
\psi_{+}(t)=\max_{p\in M_t} x_3(p).
\end{equation}
Since the blowdown for $t\to -\infty$ is a cylinder, we have
\begin{equation}
\lim_{t\to -\infty}\frac{\psi_+(t)}{\sqrt{|t|}}=\infty.
\end{equation}
Let $t_j\to -\infty$ be a sequence such that $\psi_{+}'(t_j)\sqrt{|t_j|}\to \infty$. Then we can find points $p_j\in M_{t_j}^+$ with $\psi_{+}(t_j)=x_3(p_j)$ and 
\begin{equation}
\lim_{j\to \infty} \frac{R(p_j,t_j)}{\sqrt{|t_j|}}=0.
\end{equation}
Let $\mathcal{M}^j$ be the sequence of flows that is obtained from $\mathcal{M}$ by shifting $(p_j,t_j)$ to the origin and parabolically rescaling by $R(p_j,t_j)^{-1}$, and pass to a limit $\mathcal{M}^\infty$. The limit $\mathcal{M}^\infty$ is an eternal low entropy flow, which is non-flat. Hence, by Theorem \ref{thm_class_noncompact} it is a translating bowl soliton.

\bigskip

Therefore, for each $j$ large enough $M_{t_j}^+$ is an embedded disk which has the central neck $Z_{t_j}$ as a collar and also another neck $Z_{t_j}^+$ bounding a convex cap $C_{t_j}^+$ that is $\alpha$-noncollapsed, say for $\alpha=\tfrac{1}{100}$. Let $N_{t_j}^+$ be the topological annulus with collars $Z_{t_j}$ and $Z_{t_j}^+$. Let $\{N_t\}_{t\leq t_j}$ be the time-dependent topological annulus which is obtained from $N_{t_j}^+$ by following the points backwards in time along the mean curvature evolution. By the preservation of necks backwards in time (see Section \ref{sec_necks_back_in_time}), the annulus $N_t$ has two necks as collars for all $t\leq t_j$. Moreover, for $t\ll t_j$ the entire domain $N_t$ constitutes a single neck. Hence, by the parabolic maximum principle (for $H$ and $\frac{|A|}{H}$) applied with two collar boundaries, the annulus $N_t$ is mean convex and satisfies $\frac{|A|}{H}\leq 2$ for all $t\leq t_j$.\\

Now, adding the cap $C_{t_j}^+$, we infer that $M_{t_j}^+$ is mean convex and satisfies $|A|/H\leq 100$. Since $t_j\to -\infty$, by applying the parabolic maximum principle again, but this time for the disk $M_t^+$ which has only one collar boundary, the central neck $Z_{t}$, we conclude that $M_t^+$ is mean convex and satisfies $|A|/H\leq 100$ for $t\leq \mathcal{T}$.

The same argument applies to $M_t^-$. Hence, $\mathcal M=\{M_t \}_{t\in (-\infty,T_E(\mathcal{M}))}$ is mean convex and satisfies $|A|/H\leq 100$. Together with the entropy assumption, this implies that $\mathcal M=\{M_t \}_{t\in (-\infty,T_E(\mathcal{M}))}$ is $\alpha$-noncollapsed for some $\alpha>0$. Thus, by \cite{ADS2} it is an ancient oval.
\end{proof}

\bigskip

\section{Applications}

\subsection{Proof of the mean convex neighborhood conjecture}\label{sec_proof_mean_conv_nbd}

The purpose of this section is to prove Theorem \ref{thm_mean_convex_nbd_intro} (mean convex neighborhoods at all times). Since spherical singularities, by compactness, clearly have mean convex neighborhoods, it suffices to prove:

\begin{theorem}[mean convex neighborhood theorem]\label{Mean_convex_thm}
Let $X_0=(x_0,t_0)$ be a (backwardly) singular point of $\{M_t\}_{t\geq 0}$  and suppose that 
\begin{equation}\label{tangent_assumption}
\lim_{\lambda\rightarrow \infty} \lambda (K_{t_0+\lambda^{-2}t}-x_0)=\bar{B}^2(\sqrt{2(t_0-t)})\times \mathbb{R},
\end{equation}
smoothly with multiplicity one. Then there exist an $\varepsilon=\varepsilon(X_0)>0$ such that whenever $t_0-\varepsilon < t_1<t_2 < t_0+\varepsilon$ then
\begin{equation}\label{mean_convex_conc}
K_{t_2}\cap B(x_0,\varepsilon) \subseteq K_{t_1} \setminus M_{t_1}.
\end{equation}
The same result holds with $M$ and $K$ replaced with $M'$ and $K'$.
\end{theorem}

We recall that given a closed embedded surface $M\subset\mathbb{R}^3$, we denote by $\{M_t\}_{t\geq 0}$ the outer flow. Observe that if $x\in M_t$ is a regular point, then there exists a $\delta>0$ such that $M_t\cap B(x,\delta)$ splits $B(x,\delta)$ into two connected components: one in $\mathrm{Int}(K_t)$ and the other in $\mathbb{R}^3\setminus K_t$.\\

Theorem \ref{Mean_convex_thm} will be proved after the following two auxiliary results.

\begin{proposition}\label{basic_reg}
Under the assumptions of Theorem \ref{Mean_convex_thm}, there exists a constant $\delta=\delta(X_0)>0$ and a unit-regular, cyclic, integral Brakke flow $\mathcal{M}=\{\mu_t\}_{t\geq t_0-\delta}$ whose support is $\{M_t\}_{t\geq t_0-\delta}$ such that 
\begin{enumerate}
\item[(i)] The tangent flow to $\mathcal M$ at $X_0$ is a multiplicity one cylinder. 
\item[(ii)]  The flow $\mathcal M$ has only multiplicity one cylindrical and spherical singularities in $\bar{B}(x_0,2\delta)\times [t_0-\delta,t_0+\delta]$. 
\item[(iii)]  $\mathcal M$ is smooth in $\bar{B}(x_0,2\delta)$ for a.e. $t\in [t_0-\delta,t_0+\delta]$ and is smooth outside of a set of Hausdorff dimension $1$ for every $t\in [t_0-\delta,t_0+\delta]$.
\item[(iv)]  $H\neq 0$ for every regular point of $\mathcal M$ in $\bar{B}(x_0,2\delta)\times [t_0-\delta, t_0+\delta]$.
\item[(v)] There exist $A<\infty$ and $c>0$ such that if $X=(x,t)$ is a point of $\mathcal M$ in $\bar{B}(x_0,2\delta)\times [t_0-\delta, t_0+\delta]$ with $R(X) \leq c$ then $\mathcal{M}$ is smooth and connected in $P(X,AR(X))$ and there is a point $X'=(x',t')\in \mathcal{M}\cap P(X,AR(x))$ with $R(X') \geq 2R(X)$ and with $|x'-x_0|\leq \max\{|x-x_0|-cR(X'),\delta/2\}$.
\end{enumerate} 
\end{proposition}

We remark that condition (v) will be used below to show -- via a connectedness argument -- that the mean curvature on the regular set does not change sign in a suitable space-time neighborhood of $X_0$.

\begin{proof}
It follows from \cite[Thm. B3]{HershkovitsWhite} that there exists an outer Brakke flow starting from $M$, whose support is $\{M_t\}_{t\geq 0}$. In particular, this, together with monotonicity, implies that $\mathrm{Ent}[\mathcal{H}^n\lfloor M_t]$ is uniformly bounded. Thus, \eqref{tangent_assumption} implies that there exists $t_{\ast}<t_0$ such that  
\begin{equation}\label{low_dens}
\frac{1}{(4\pi(t_0-t_{\ast}))^{n/2}}\int \exp\Big(\frac{-|x-x_0|^2}{4(t_0-t_{\ast})}\Big)d\mathcal{H}^n\lfloor M_{t_{\ast}}<2. 
\end{equation}
Therefore, applying \cite[Thm. B3]{HershkovitsWhite} once more, this time with the initial time $t_{\ast}$, we get a unit-regular, cyclic, integral Brakke flow $\mathcal{M}:=\{\mu_t\}_{t\geq t_{\ast}}$ whose support is $\{M_t\}_{t\geq t_\ast}$. Together with \eqref{tangent_assumption}  and \eqref{low_dens} it follows that $\mathcal M$ has a multiplicity $1$ tangent cylinder at $(x_0,t_0)$. This proves (i). \\

(ii) now follows from the upper semi-continuity of the density (see Section \ref{sec_prel_monotonicity}), the proof of Theorem \ref{partial_regularity} (partial regularity) and the gap result from Bernstein-Wang \cite{BW}.\\

(iii) follows from (ii) and standard stratification, see e.g. \cite{White_stratification}.\\

Suppose towards a contradiction that (iv) does not hold. Then there exists a sequence of smooth point $X_i=(x_i,t_i) \to X_0=(x_0,t_0)$ with  $H(X_i)=0$. Denoting by $r_i=R(X_i)$ the regularity scale, since $X_0$ is singular and $X_i\to X_0$, we have $r_i\rightarrow 0$.

Let $\mathcal{M}^i$ be the flow that is obtained from $\mathcal M$ by shifting $X_i$ to the origin and parabolically rescaling by $1/r_i$, and pass (using Ilmanen's compactness theorem \cite{Ilmanen_book}) to a Brakke flow limit $\mathcal{M}^\infty$.\\

By construction as a blowup limit, $\mathcal M^\infty$ is an ancient, integral Brakke flow. By the local regularity theorem for the mean curvature flow and \cite[Thm. B5]{HershkovitsWhite}, $\mathcal{M}^{\infty}$ is unit regular. By a result of White \cite{White_Currents}, the limit $\mathcal M^\infty$ is cyclic. Furthermore, we have
\begin{equation}\label{entropy_bound_to_show}
\mathrm{Ent}[\mathcal M^\infty]\leq \mathrm{Ent}[ S^1\times\mathbb{R}]. 
\end{equation}
This entropy bound follows from the same reasoning as in the proof of \cite[Thm. 1]{Hershkovits_translators}. For the sake of completeness, we include it here as well. Since the $\mathcal{M}^i$ have uniformly bounded entropy (by monotonicity), and as $\mathcal{M}^i\rightarrow \mathcal{M}^\infty$, it follows from the definition of the entropy of $\mathcal{M}^\infty$  that there exist a sequences of points $(x'_k,t'_k)$ in the original space-time with  $(x'_k,t'_k)\rightarrow (x_0,t_0)$ and scales $r'_k\rightarrow 0$, such that for every $\varepsilon>0$, for $k$ large enough,   
\begin{equation}
\frac{1}{4\pi {r'_k}^2}\int\exp(-|x-x'_k|^2/4{r'_k}^2)d\mu_{t'_k-{r'_k}^2} \geq \mathrm{Ent}[\mathcal{M}^\infty]-\varepsilon.
\end{equation}
On the other hand, since $(x_0,t_0)$ is a cylindrical singularity, there exist $r>0$ such that for every point in $(p,t)\in B(x_0,r)\times [t_0-r,t_0+r]$ we have
\begin{equation}
\frac{1}{4\pi r^2}\int\exp(-|x-p|^2/4r^2)d\mu_{t-r^2} \leq \mathrm{Ent}[\mathbb{S}^1\times\mathbb{R}]+\varepsilon.
\end{equation}
By Huisken's monotonicity formula and by the arbitrariness of $\varepsilon$, this implies \eqref{entropy_bound_to_show}. Hence, $\mathcal M^\infty$ is an ancient low entropy flow.\\

Now, by Theorem \ref{thm_ancient_low_entropy} (classification of ancient low entropy flows) the limit $\mathcal{M}^\infty$ is either a flat plane, round shrinking sphere, round shrinking cylinder, a bowl soliton, or an ancient oval. If the limit is non-flat, then for $i$ large enough that contradicts $H(X_i)=0$. If the limit is a flat plane, then we obtain a contradiction with the fact $X_i$ has regularity scale $r_i\to 0$, by the local regularity theorem. This proves (iv).\\

For (v), observe by inspection of the four non-planar ancient low entropy flows that there exist  $A,C<\infty$ such that if $X=(x,t)$ is a point on such a flow $\mathcal{M}^{\infty}$, then  for every unit vector $v$, there exists a point in $X'\in P(X,AR(X))\cap \{y:(y-x)\cdot v\geq R(X)\}\cap \mathcal{M}^{\infty}$ with 
\begin{equation}
4R(X) \leq R(X')\leq CR(X).
\end{equation}
Using this, (v) follows by a contradiction argument similarly as in the proof of (iv). This finishes the proof of the proposition.
\end{proof}

By wiggling with the constant $\delta$ we can assume that, in addition to (i)-(v) of Proposition \ref{basic_reg} we also have:
\begin{enumerate}
\item[(vi)] $t_0-\delta$ is a smooth time for the flow in $\bar{B}(x_0,2\delta)$, and $M_{t_0-\delta}\cap \bar{B}(x_0,2\delta)$ is $\varepsilon_0$-cylindrical, where $\varepsilon_0>0$ is fixed. 
\end{enumerate}

\bigskip

Let $\mathcal{N}$ the unique (space-time) connected component of $\mathcal{M}\cap (\bar{B}(x_0,\delta)\times [t_0-\delta,t_0+\delta])$ which contains the point $X_0=(x_0,t_0)$.\footnote{In this section $B(x,r)$ always denotes an open ball, and $\bar{B}(x,r)$ a closed ball.} In particular, there exist $\varepsilon>0$ such that  
\begin{equation}\label{N_neigh}
\mathcal{M}\cap \Big(B(x_0,2\varepsilon)\times [t_0-\varepsilon,t_0+\varepsilon]\Big)\subseteq \mathcal{N}.
\end{equation}
We will show that $B(x_0,\varepsilon)\times [t_0-\varepsilon,t_0+\varepsilon]$ is a mean convex neighborhood. 

For any smooth point $X=(x,t)\in\mathcal{M}\cap(\bar{B}(x_0,2\delta)\times [t_0-\delta,t_0+\delta])$, if $\mathbf{H}(X)$ points into $K_t$ we say that $H(X)>0$, and if $\mathbf{H}(X)$ points outside of $K_t$, we say that $H(X)<0$. 

\begin{lemma}\label{fixed_sign_smooth}
$H>0$ for all regular points in $\mathcal{N}$.
\end{lemma}
\begin{proof}
For every time $s\in [t_0-\delta,t_0+\delta]$, let $\mathcal{A}^s$ be the set of (space-time) connected components of $\mathcal{N}\cap (\mathbb{R}^3\times [t_0-\delta,s])$, and let $ \mathcal{A}^s_{-}$ be the set of space time connected components of  $\mathcal{N}\cap (\mathbb{R}^3\times [t_0-\delta,s))$. Set
\begin{equation}
\begin{aligned}
I:= \{s\in [t_0-\delta,t_0+\delta]\;|\; &\textrm{For every }\mathcal{F}\in \mathcal{A}^s\; \textrm{and }X_1,X_2\in \mathrm{Reg}\cap \mathcal{F},\\
& \mathrm{Sign}(H(X_1))=\mathrm{Sign}(H(X_2))\}.
\end{aligned}
\end{equation}

Note that $t_0-\delta \in I$ by property (vi). Also note that if $s\notin I$, then $s'\notin I$ for every $s'>s$. Thus $I$ is an interval containing $t_0-\delta$. The assertion of the lemma is therefore  equivalent to the assertion that $t_0+\delta\in I$. Suppose towards a contradiction that $t_0+\delta\notin I$, and set
\begin{equation}
s_0:=\sup I.
\end{equation} 

\begin{claim}\label{claim_not_in_I}
$s_0\notin I$.
\end{claim}
\begin{proof}[{Proof of Claim \ref{claim_not_in_I}}]
If $s_0=t_0+\delta$ this is clear by assumption. Assume that $s_0<t_0+\delta$ were in $I$. Since $\mathcal{N}\cap \bar{B}(x_0,\delta)$ is compact, each two connected components of $\mathcal{F}_1\neq \mathcal{F}_2$ in $\mathcal{A}^{s_0}$ are a positive distance apart. Moreover, there will be no new connected components entering $\bar{B}(x_0,\delta)$ for some time (again, since supports of Brakke flows are closed). Thus, there exists an $s>s_0$ such that each connected components $\mathcal{F}^{s}\in A^{s}$ has a unique connected component $\mathcal{F}^{s_0}\in A^{s_0}$ such that $\mathcal{F}^{s_0} \subseteq \mathcal{F}^s$. 

We claim that if $s-s_0$ is sufficiently small, every smooth point in $\mathcal{F}^s$ can be connected by a path of smooth points in $\bar{B}(x_0,\delta)$ to a point in $\mathcal{F}^{s_0}$. Indeed, by repeated use of Proposition \ref{basic_reg} (v), any regular point $X=(x,t)\in \mathcal{F}^s$, can be connected via regular points to some  $(x',t')=X'\in \mathcal{F}^s$ with $R(X')\geq c$ and with $|x'-x_0|\leq \delta-c^2$. Thus, is $s-s_0$ is sufficiently small, $X'$ can be connected via smooth points to a point in $\mathcal{F}^{s_0}$ (this argument is similar to \cite[Propsition 3.5]{Kleiner-lott-weak-Ricci}).

As $H\neq 0$, this implies that $H$ does not change sign on $\mathcal{F}^s$.  This proves the claim.
\end{proof}

Thus, if $t_0+\delta\notin I$ then $I^c=[s_0,t_0+\delta]$ for some $s_0\geq t_0-\delta$.  As $s_0\in I^c$, there exists some connected component $\mathcal{F}\in \mathcal{A}^{s_0}$ such that $H$ does not have a fixed sign over regular points in $\mathcal{F}$.  Since for each $s<s_0$, we have $s\in I$, there are three potential ways in which this can happen for the first time at time $s_0$. We will rule them all out.

\bigskip

\noindent\emph{1. Appearance of a new connected component}: $\mathcal{F} \subseteq \partial B(x_0,\delta) \times \{s_0\}$, i.e. a new connected component enters through the boundary. In this case, all points of $\mathcal{F}$ are regular, as the tangent flows are contained in half-spaces. Since $\mathcal{F}$ is connected (by definition), and as $H\neq 0$, this implies that the sign of $H$ is constant along $\mathcal{F}$. Thus, this case is excluded.

\bigskip

\noindent\emph{2. Merger of two connected components}: There exist two distinct space-time connected components $\mathcal{F}_{\pm}\in \mathcal{A}^{s_0}_{-}$ such that $H>0$ (resp. $H<0$) on $\mathcal{F}_{+}$ (resp. $\mathcal{F}_{-}$) and such that $\overline{\mathcal{F}_+}\cap \overline{\mathcal{F}_{-}}\neq \emptyset$. Since $H\neq 0$ at regular points, the points of $\overline{\mathcal{F}_+}\cap \overline{\mathcal{F}_{-}}$ are all singular (in particular, $s_0$ is a singular time). Moreover, by avoidance, $\overline{\mathcal{F}_+}\cap \overline{\mathcal{F}_{-}}\subseteq \partial B(x_0,\delta)\times \{s_0\}.$

\begin{claim}\label{claim_case_two}
There exists a time $t<s_0$ such that $M_t\cap B(x_0,2\delta)$ is smooth and some \textit{spatial} connected component of $M_t\cap \bar{B}(x_0,2\delta)$ contains points both from $\mathcal{F}_{+}$ and from $\mathcal{F}_{-}$.
\end{claim}  

\begin{proof}[{Proof of Claim \ref{claim_case_two}}]
Consider the spatial connected components of such $\mathcal{F}_{\pm}$, $\mathcal{F}_{\pm}(t)$, and observe how they extend to $\bar{B}(x_0,2\delta)$. If they never intersect on smooth times, then via a density argument (or, indeed, by the maximum principle), we see that, at singular times as well, if $(x',t')$ is a point reached by both smooth flows then $x'\in \partial B(x_0,2\delta)$. We thus obtain two disjoint  unit multiplicity Brakke flows in $B(x_0,3\delta/2)$, which have a point in  $\overline{\mathcal{F}_+}\cap \overline{\mathcal{F}_{-}}$ as a joint point they reach at time $s_0$, but this is again impossible since the Huisken density near $(x_0,t_0)$ is less than $2$. 
\end{proof}
As $H\neq 0$ on $M_t\cap \bar{B}(x_0,2\delta)$, this gives a contradiction to the definitions of  $\mathcal{F}_{\pm}$. Indeed, at time $t$, the signs of $H$ on $\mathcal{F}_{\pm}$ have to be the same. Thus, this case is excluded as well.

\bigskip

\noindent\emph{3. Merger of an existing component with a new boundary component}: In this case, we may assume without loss of generality that   $\mathcal{F}$ contains $\mathcal{F}_{+}\cup \mathcal{F}_{\partial}$, where on $\mathcal{F}_{+}\in \mathcal{A}^{s_0}_{-}$ we have $H>0$ on regular points and where $\mathcal{F}_{\partial}\subseteq \mathcal{F}\cap (\partial B(x_0,\delta) \times \{s_0\})$ is connected, intersects $\overline{\mathcal{F}_{+}}$, and contains at least one point which is not in 
\begin{equation}
\overline{\mathcal{M}\cap (\bar{B}(x_0,\delta)\times [t_0-\delta,s_0))},
\end{equation}
and all such points have $H<0$ (note that, as in the first case, such points are always regular). This case is excluded similarly to Case 2 above: Taking such a point, and looking at the connected component of the point reaching it at a slightly earlier smooth time $t$ in a slightly bigger ball, we see it intersects with the continuation of $\mathcal{F}_{+}(t)$ in that ball, which cannot be.

\bigskip

We have thus excluded the existence of such an $\mathcal{F}$, so $I^c=\emptyset$. This concludes the proof of the lemma.
\end{proof}

\begin{proof}[Proof of Theorem \ref{Mean_convex_thm}]
Having established Proposition \ref{basic_reg} and Lemma \ref{fixed_sign_smooth}, the final step of the proof of the mean convex neighborhood theorem is now similar to the final step of the proof of Theorem 3.5 in \cite{HershkovitsWhite}. For convenience of the reader, we include the argument here as well. 

\begin{claim}\label{weak-containment}
If $t_0-\varepsilon \leq t_1 < t_2 \leq t_0+\varepsilon$, then
\begin{equation}
    K_{t_2}\cap B(x_0,\varepsilon) \subseteq K_{t_1}.
\end{equation}
\end{claim}
\begin{proof}[{Proof of Claim \ref{weak-containment}}]
Suppose to the contrary that there is point $y\in K_{t_2}\cap B(x_0,\varepsilon)$ that is not in $K_{t_1}$ for some $t_1<t_2$ with $t_1,t_2\in [t_0-\varepsilon,t_0+\varepsilon]$.
Let
\begin{equation}
  0< \rho < \textrm{dist}(y, K_{t_1} \cup \partial B(x_0,\varepsilon)).
\end{equation}
Let $t\in (t_1,t_0)$ be the first time $\ge t_1$ such that 
\begin{equation}
  \textrm{dist}(y, K_t) = \rho.
\end{equation}
Let $p\in K_t$ be a point such that $\textrm{dist}(y,p)=\textrm{dist}(y,K_t)$.
Note that the tangent flow at $(p,t)$ lies in a half-space (more specifically the half-space $\{x: x\cdot (y-p) \le 0\}$).
Hence $(p,t)$ is a regular point of the flow.
By the previous lemma, the mean curvature at $(p,t)$ is positive and so for $s<t$ very close to $t$, we have $\textrm{dist}(y, K_s) < \rho$; a contradiction. This proves the claim.
\end{proof}

Continuing the proof of the theorem, it only remains to show that for $t_1<t_2$ with $t_1,t_2\in [t_0-\varepsilon,t_0+\varepsilon]$, if 
\begin{equation}
   x\in B(x_0,\varepsilon)\cap K_{t_2},
\end{equation}
then $x$ is in the interior of $K_{t_1}$.
If $x$ is the interior of $K_{t_2}$, then by Claim~\ref{weak-containment} it is in the interior of $K_{t_1}$.
Thus we may assume that $x$ is in the boundary of $K_{t_2}$.
For $s \in [t_1,t_2)$ sufficiently close to $t_2$, the point $x$ is in the interior of $K_s$.
If $(x,t_2)$ is a regular point, this is because the mean curvature is positive.  If $(x,t_2)$ is a singular point, this is true by (ii) of Proposition \ref{basic_reg}. 
Since $x$ is in the interior of $K_s$, and since $K_s\subseteq K_{t_1}$, we conclude that $x$ is in the interior of $K_{t_1}$.  This completes the proof of the mean convex neighborhood theorem.
\end{proof}

\bigskip

\subsection{Proof of the uniqueness conjectures for weak flows}

In this final section, we prove the uniqueness results for mean curvature flow through singularities, namely Theorem \ref{thm_nonfattening_intro} (nonfattening) and Theorem \ref{thm_uniqueness_two_spheres} (evolution of two-spheres).

\begin{proof}[{Proof of Theorem \ref{thm_nonfattening_intro}}]
Observe that if $T\leq T_{\mathrm{disc}}$ then, by definition, the outer flow $\{M_t\}_{t\in [0,T]}$ agrees with the level set flow of $M$ and the inner flow of $M$. In particular, if $(x_0,T)$ is a multiplicity one cylindrical (resp. spherical) singularity of $\{M_t\}$ then \textit{either} $\mathcal{K}_{X,\lambda}$, \textit{or} $\mathcal{K}_{X,\lambda}'$ converges for $\lambda\to 0$ smoothly with multiplicity one to a round shrinking solid cylinder or a round shrinking solid ball (see also item (i) of Proposition \ref{basic_reg} and its proof). 

The result now follows from combining the main theorem of Hershkovits-White \cite{HershkovitsWhite}, which establishes that $T< T_{\mathrm{disc}}$ assuming the existence of mean convex neighborhoods a priori, and Theorem \ref{thm_mean_convex_nbd_intro}, which proves the existence of mean convex neighborhoods.
\end{proof}

\begin{proof}[{Proof of Theorem \ref{thm_uniqueness_two_spheres}}]
Let $M\subset\mathbb{R}^3$ be an embedded two-sphere, and let $\mathcal{M}=\{M_t\}$ be the outer flow starting at $M$. By a result of White \cite{White_topology}, whenever $M_t$ is smooth, then it has genus zero.

Suppose towards a contradiction that there is some singular point with a tangent flow that is not spherical or cylindrical. Let $T$ be the first time when this happens (more precisely, the infimum of times where this happens). Then $T\leq T_{\mathrm{disc}}$ (by Theorem \ref{thm_nonfattening_intro}) and for a.e. $t\in [0,T]$ the flow is smooth. Hence, we can run the argument from Ilmanen \cite{Ilmanen_monotonicity}, and get, assuming the multiplicity one conjecture, that at time $T$ at every singular point there exists a tangent flow that is smooth with multiplicity one.

Such a tangent flow, by smoothness and unit-regularity, cannot be a nontrivial static or quasistatic cone. Hence it is a properly embedded shrinker. Suppose towards a contradiction there are two loops in the shrinker $N$ that have nonvanishing intersection number modulo two. Then for some smooth time $t'<T$ we would see two loops in $M_{t'}$ with nonvanishing intersection number modulo two; a contradiction. Hence, by a result of Brendle \cite{Brendle_sphere} the shrinker $N$ is a sphere, cylinder or a plane. Thus, arguing as in the first paragraph of the proof of Theorem \ref{thm_nonfattening_intro}, by the mean convex neighborhood theorem (Theorem \ref{thm_mean_convex_nbd_intro})  and compactness, there is some $\delta>0$ such that the flow has only cylindrical and spherical singularities on $[0,T+\delta)$; this contradicts the definition of $T$. Hence, the flow $\mathcal M$ has only spherical and cylindrical singularities.
Thus, by Theorem \ref{thm_nonfattening_intro} (nonfattening) the solution is unique.
\end{proof}

\bibliography{CHH_final}

\bibliographystyle{alpha}

\vspace{5mm}

{\sc Kyeongsu Choi, School of Mathematics, Korea Institute for Advanced Study, 85 Hoegiro, Dongdaemun-gu, Seoul, 02455, South Korea}\\

{\sc Robert Haslhofer, Department of Mathematics, University of Toronto,  40 St George Street, Toronto, ON M5S 2E4, Canada}\\

{\sc Or Hershkovits, Institute of Mathematics, Hebrew University, Givat Ram, Jerusalem, 91904, Israel}\\

\emph{E-mail:} choiks@kias.re.kr, roberth@math.toronto.edu, or.hershkovits@mail.huji.ac.il

\end{document}